\documentclass[12pt,oneside]{amsart}

\title{Equivariant Factorization Homology 
	\\of global quotient orbifolds}
\author{T.A.N. Weelinck}

\usepackage{amsmath,amssymb,amsthm, adjustbox}
\usepackage{verbatim}
\usepackage{enumitem} 
\usepackage{tikz,tikz-cd}
\usepackage{stmaryrd} 
\usepackage{fullpage} 
\usepackage{relsize}
\usepackage{scalerel}
\usepackage{braids} 
\usetikzlibrary{knots} 
\usetikzlibrary{arrows} 
\usepackage{caption} 
\usepackage{color} 
\usepackage{enumitem} 
\usepackage{mathtools,extarrows} 

\usepackage{hyperref}
\hypersetup{linktocpage}
\usepackage{xcolor}
\hypersetup{
	colorlinks,
	linkcolor={blue!80!black},
	citecolor={blue!50!black},
	urlcolor={blue!80!black}
}

\let\oldtocsection=\tocsection
\let\oldtocsubsection=\tocsubsection
\renewcommand{\tocsection}[2]{\hspace{0em}\oldtocsection{#1}{#2}}
\renewcommand{\tocsubsection}[2]{\hspace{2em}\oldtocsubsection{#1}{#2}}

\theoremstyle{plain}
\newtheorem{thm}{Theorem}[section]
\numberwithin{equation}{thm}
\newtheorem{prop}[thm]{Proposition}
\newtheorem{lem}[thm]{Lemma}
\newtheorem{cor}[thm]{Corollary}

\theoremstyle{definition}
\newtheorem{defn}[thm]{Definition}
\newtheorem{ex}[thm]{Example}

\newtheorem{rmk}[thm]{Remark}

\theoremstyle{remark}

\newtheorem*{conv}{Convention}
\newtheorem*{nota}{Notation}

\newcommand{\R}{\mathbb{R}}
\newcommand{\N}{\mathbb{N}}
\newcommand{\Z}{\mathbb{Z}}
\newcommand{\C}{\mathbb{C}}
\newcommand{\bE}{\mathbb{E}}

\newcommand{\cS}{\mathcal{S}}
\newcommand{\cF}{\mathcal{F}}
\newcommand{\kk}{\mathbb{K}}

\newcommand{\vect}{\mathbf{Vect}}

\newcommand{\Rex}{\mathbf{Rex}}

\newcommand{\Orb}{\mathbf{Orb}}
\newcommand{\Disk}{\mathbf{Disk}}
\newcommand{\Mfld}{\mathbf{Mfld}}
\newcommand{\Spaces}{\mathbf{Spaces}}
\newcommand{\Quot}{\mathbf{Quot}}
\newcommand{\AB}{\mathbf{Ab}}
\newcommand{\Ch}{\mathbf{Ch}}

\newcommand{\gar}{\Gamma^\rho}
\newcommand{\ev}{\mathrm{ev}}
\newcommand{\Env}{\mathrm{Env}}
\newcommand{\Alg}{\mathrm{Alg}}
\newcommand{\colim}{\mathrm{colim}\;}

\newcommand{\catC}{\mathcal{C}}
\newcommand{\catA}{\mathcal{A}}
\newcommand{\cH}{\mathcal{H}}

\newcommand{\catD}{\mathcal{D}}
\newcommand{\catM}{\mathcal{M}}
\newcommand{\catN}{\mathcal{N}}
\newcommand{\catO}{\mathcal{O}}

\DeclareSymbolFont{bbold}{U}{bbold}{m}{n}
\DeclareSymbolFontAlphabet{\mathbbold}{bbold}

\newcommand{\g}{\mathfrak{g}}
\newcommand{\e}{\mathbf{e}}
\newcommand{\triv}{\mathbb{I}}
\newcommand{\sign}{\sigma}
\newcommand{\fr}{\mathrm{fr}}

\newcommand{\U}{\mathcal{U}}

\newcommand{\ga}{\Gamma}
\newcommand{\uqg}{\U_q(\g)}

\newcommand{\tensor}{\otimes}

\newcommand{\ep}{\varepsilon}

\newcommand{\Fr}{\mathrm{Fr}}
\newcommand{\Map}{\mathrm{Map}}

\newcommand{\Emb}{\mathrm{Emb}}
\newcommand{\EEmb}{\mathbb{E}\mathrm{mb}}
\newcommand{\Fun}{\mathrm{Fun}}

\newcommand{\id}{\mathrm{id}}

\newcommand{\Ad}{\mathrm{Ad}}

\newcommand{\Fin}{\mathrm{Fin}}
\newcommand{\Ind}{\mathrm{Ind}}
\newcommand{\hocolim}{\mathrm{hocolim}\;}

\newcommand{\DD}{\mathbb{D}}

\begin{document}

\begin{abstract}
We introduce equivariant factorization homology, extending the axiomatic framework of Ayala-Francis to encompass multiplicative invariants of manifolds equipped with finite group actions. 
Examples of such equivariant factorization homology theories include Bredon equivariant homology and (twisted versions of) Hochschild homology.
Our main result is that equivariant factorization homology satisfies an equivariant version of $\tensor$-excision, and is uniquely characterised by this property.
We also discuss applications to representation theory, such as constructions of categorical braid group actions. 
\end{abstract}

\maketitle
\tableofcontents
\section{Introduction}

The goal of this paper is to generalise the theory of factorization homology of manifolds to an equivariant setting. We thereby construct new invariants of so-called global quotient orbifolds, which are smooth manifolds equipped with finite group actions. 

The origins of factorization homology lie in conformal field theory. 
In \cite{bd04} Beilinson and Drinfeld defined chiral homology, an (algebro-geometric) abstraction of the space of conformal blocks of a vertex algebra. 
A topological analogue of their construction, known as factorization homology, was defined by Lurie \cite{lurietft,lha} and further developed by Ayala, Francis and Tanaka \cite{af15, aft17}.
Additionally, factorization homology has roots in the labeled configuration spaces of Segal \cite{segal73, segal10}, McDuff \cite{mcduff75}, Salvatore \cite{salvatore01} and others. 
The invariants constructed by factorization homology were also described in the thesis of Andrade \cite{andrade10}, and are closely related to the blob homology of Morrison-Walker \cite{mw12}. Furthermore, factorization homology can be understood through the factorization algebras of Costello-Gwilliam \cite{cg17}, a point to which we will return below.
We recommend \cite{ginot15} as a survey.

Like most homology theories, the invariants of factorization homology are constructed by fixing an algebraic input $A$, that provides gluing rules, and then applying these gluing rules to a manifold $M$, which is viewed as a gluing pattern. 
The obtained invariant is denoted $\int_M A$.
The invariants are multiplicative, in the sense that they take values in a given symmetric monoidal category $\catC^{\tensor}$ and the invariant $\int_{M \amalg N} A$ assigned to a disjoint union of manifolds is naturally isomorphic to the tensor product $\int_M A \tensor \int_N A$.

Factorization homology can be defined in diverse geometric settings: for framed manifolds, oriented manifolds, as well as manifolds with singularities \cite{af15,aft17}. 
The algebraic input needed depends on the geometric context considered. To define invariants of framed $n$-manifolds one needs an $E_n$-algebra, whereas the invariants of oriented $n$-manifolds require a (confusingly named, in this context) framed $E_n$-algebra.
For example, an associative algebra $A$ can be viewed as an $E_1$-algebra in chain complexes, and the invariant $\int_{S^1} A$ that is assigned to the circle, is the Hochschild homology of $A$ \cite[Ex. 5.5.3.14]{lha}.
The invariants are functorial with respect to manifold embeddings.
This recovers the well-known circle action on Hochschild homology in the previous example. 

The interpretation of factorization homology in topological quantum field theory (abbreviated TQFT) is important, both motivationally and conceptually. 
In \cite{cg17} Costello-Gwilliam defined factorization algebras, axiomatising the structure of observables within a quantum field theory.
In the case of a TQFT the observables are completely determined by the local observables, which are a locally constant factorization algebra on $\R^n$. These can be identified with $E_n$-algebras by a result of Lurie \cite[Th. 5.4.5.9]{lha}.
Factorization homology then provides a construction of the global observables $\int_M A$ on the space-time $M$ out of the local observables $A$. 
This is an instance of the \emph{locality principle} in quantum field theory: the global observables should be determined by the local observables.
Further connections to TQFTs are provided by the work of Scheimbauer \cite{scheimbauer}, who has used factorization homology with coefficients in $E_n$-algebras to construct fully extended n-dimensional TQFTs in the sense of the Atiyah-Segal axiomatisation of TQFTs \cite{atiyah88,segal04}.
Here \emph{fully extended}, or \emph{fully local}, refers to the cobordism hypothesis \cite{bd95,lurietft}, where the locality principle is encoded via cutting-and-gluing rules.

In the context of factorization homology locality takes the form of $\tensor$-excision. 
This property states that if one decomposes a manifold along a collar-gluing $M \cong M_+ \cup_{M_0 \times \R} M_-$ then the invariant $\int_M A$ can be computed as a relative tensor product $\int_{M_+} A \tensor_{\int_{M_0 \times \R} A} \int_{M_-} A$.
A key result of \cite{af15} is that factorization homology satisfies $\tensor$-excision and is uniquely characterised, as a symmetric monoidal functor, by this property. 
This generalises the Eilenberg-Steenrod axioms for singular homology. 
The $\tensor$-excision property is then thought of as a multiplicative analogue of the Mayer-Vietoris property of singular homology.

In this work we generalise factorization homology further. 
We define factorization homology of smooth manifolds equipped with finite group actions. 
These actions are not required to be free, so that we are constructing invariants of global quotient orbifolds.  
Our main result is that this equivariant factorization homology theory satisfies an equivariant version of $\tensor$-excision, and is uniquely characterised by this property.
This extends the characterisation result of Ayala-Francis, and is reminiscent of Bredon's axioms for equivariant homology. 

\subsection{Factorization homology of manifolds} \label{subsec:summary}

Before stating our main results we will introduce the geometric context of factorization homology in greater detail, and recall some of the basic definitions in preparation of our new definitions for orbifolds. 
Ayala-Francis utilise the classical notion of \emph{$G$-structures} on manifolds, under the name of \emph{$G$-framed manifold}, as a unifying framework to treat different tangential structures such as framings and orientations.\footnote{To be precise, in \cite{af15} they encode the $G$-structure topologically via maps into classifying spaces. See \cite[\S 7 \& \S 10]{mitchell} for a clear exposition relating the topological and smooth perspective.}

The idea of a $G$-structure is in the spirit of Klein's Erlangen program. 
Rather than defining a structure on the tangent bundle, one considers the frames\footnote{A \emph{frame} on $U\subset M$ is a collection of sections of $TM$ over $U$ that provide a basis in $T_xM$, $\forall x\in U$.} on the manifold that are invariant under the corresponding structure group. 
For example, instead of studying a metric, one studies the $O(n)$-invariant frames on the manifold.
These $O(n)$-invariant frames form an $O(n)$-subbundle of the frame bundle, which becomes the central object of study.
A general $G$-structure is then simply a $G$-subbundle of the frame bundle. 
For example, a framing is exactly an $\e$-structure for the trivial group $\e$ and orientations are $SO(n)$-structures.

Ayala-Francis then define the $\infty$-category $\Mfld_n^G$ of $n$-dimensional $G$-framed manifolds, where maps are embeddings that preserve the $G$-structure up to homotopy. 
This categorical set-up has several benefits. Firsly, the invariants $\int_M A$ are simple to define.
Secondly, the invariants will be multiplicative and functorial by definition.  
Namely, one considers a subcategory $\Disk_n^G \subset \Mfld_n^G$, consisting only of disjoint unions of the trivially $G$-framed disk $\R^n$.
Then one fixes a \emph{coefficient system} A with values in a given symmetric monoidal category $\catC^\tensor$. 
They are referred to as a \emph{$\Disk_n^G$-algebras} in \cite{af15}. 
Concretely, a $\Disk_n^G$-algebra valued in $\catC^\tensor$ is a symmetric monoidal functors
$A: \Disk_n^G \rightarrow \catC^\tensor$,
where the monoidal structure on $\Disk_n^G$ is given by disjoint union. 

\begin{ex}
	The notion $\Disk_n^\e$-algebra coincides with that of $E_n$-algebra. 
\end{ex}
The factorization homology with coefficients in $A$ is then defined as the left Kan extension of $A$ along the inclusion $\Disk_n^G \subset \Mfld_n^G$.
This yields a symmetric monoidal functor 
\[ \int_\bullet A: \Mfld_n^G \rightarrow \catC^{\tensor}, \]
from which the invariant $\int_M A$ is obtained by evaluating the functor on $M$. 
Intuitively, the left Kan extension replaces the manifold $M$ with a colimit of a covering of disks $\R^n \subset M$, and then computes the invariant as $\int_M A := \colim_{\R^n \subset M } A(\R^n)$.

\subsection{Factorization homology of orbifolds}
Recall that a global quotient orbifold is a smooth manifold equipped with a finite group action. 
We wish to consider various notions of tangential structures on these: framings, orientation, almost complex structures, etc. 
These notions were not yet defined for orbifolds, so we propose definitions of our own.

A first definition of a framing of $[M/\ga]$ might be a $\ga$-invariant framing of $M$. 
This, however, is too restrictive: there are essentially no examples.
Rather, we require the framing to be $\ga$-equivariant with respect to a given representation $\rho: \ga \rightarrow GL(\R^n)$. 
This inspires our definition for general $G$-structures, which captures the intuition of $G$-invariant and $\ga$-equivariant frames.
Concretely, this means lifting the $\ga$-action on $M$ to the $G$-structure, where the interaction between $G$ and $\ga$ is controlled by a representation $\rho$. 
This data defines what we dub a $(\rho,G)$-structure on $[M/\ga]$, or $(\rho,G)$-framed global quotient.

We then define $\infty$-categories $\gar\Orb^{G}_n$ of $(\rho,G)$-framed global quotients, where maps are $\ga$-equivariant embeddings that preserve the $(\rho,G)$-structure up to homotopy. 
This generalises the categorical set-up in \cite{af15}, and allows us to define our invariants similarly.  
We pause to highlight an important feature of the category of orbifold disks.  
A manifold is covered by disks $\R^n$, but to cover an orbifold one needs local models $[\R^n/I]$ for the various subgroups $I$ of $\ga$.
Therefore, we need to include all these local models in our subcategory of orbifold disks $\gar\Disk^{G}_n \subset \gar\Orb^{G}_n$.
Otherwise our set-up is the same: a $\gar\Disk^{G}_n$-algebra is a symmetric monoidal functor
$A:  \gar\Disk_n^G \rightarrow \catC^\tensor$,
and the equivariant factorization homology with coefficients in $A$ is defined as the left Kan extension of $A$ along the inclusion $\gar\Disk^{G}_n \subset \gar\Orb^{G}_n$.
 
Let us give some examples of coefficient systems.
\begin{ex} (Proposition \ref{prop:equivariantalgebras})
	Consider the trivial representation $\triv$.\footnote{By trivial representation we mean $\triv: \ga \rightarrow GL(\R^n)$ mapping everything to $\mathrm{Id} \in GL(\R^n)$.} A $\ga^\triv\Disk^{\e}_n$-algebra is exactly a $\ga$-equivariant $E_n$-algebra, which in turn is equivalent to the data of a $\ga$-equivariant locally constant factorization algebra in the sense of Costello-Gwilliam \cite[\S 3.7]{cg17}.
\end{ex}

\begin{ex} (Example \ref{ex:diskrotalg}) 
	Let $\sign: \Z_2 \rightarrow GL(\R^2)$ be the sign representation. A $\Z_2^{\sign}\Disk_2^{\e}$-algebra in vector spaces is a pair $(A,M)$ where $A$ is a commutative algebra with involution $\phi: A \rightarrow A$, and $M$ is an $A$-module so that $\phi(a) \cdot m = a \cdot m,$ $\forall a\in A$, $m\in M$.
\end{ex}

\begin{ex} (Example \ref{ex:brpair}, \cite[Thm 1.3]{TOQSP}) \label{ex:qsp}
	Consider the $2$-category $\Rex$ of $\kk$-linear categories (see \S \ref{subsec:crt}). 
	A $\Z_2^{\sign}\Disk^{\e}_{2}$-algebra in $\Rex$ consists of a braided tensor category $\catA$ with an anti-involution $\Phi: \catA \rightarrow \catA$, and an $\catA$-module category $\catM$ with a $\Z_2$-cylinder braiding. Examples are provided by the categories of representations of a quantum symmetric pair. 	
	
	Recall, a quantum symmetric pair $(\uqg,B)$ is a quantization of the symmetric pair $(\g, \g^{\theta})$, where $\g$ is a complex semisimple Lie algebra and $\g^{\theta} \subset \g$ is a sub-Lie algebra fixed by some given involutive automorphism $\theta: \g \rightarrow \g$ (see \cite{letzter99,letzter02,letzter03}).
\end{ex}

\subsection{Main results}

Our first main result establishes $\tensor$-excision. 

\begin{conv}
	As in \cite{af15} we make a restriction on target categories . We assume target categories $\catC^\tensor$ are $\tensor$-cocomplete, meaning that $\catC^\tensor$ is cocomplete and that the tensor product preserves certain colimits (see Definition \ref{def:tensorcocomplete}). 
\end{conv}

\begin{thm} (Theorem \ref{thm:excision})\label{intro:excision}
	Let $\catC^{\tensor}$ be $\tensor$-cocomplete, and let $A$ be a $\gar\Disk^{G}_n$-algebra in $\catC^{\tensor}$.
	For every $\ga$-equivariant collar-gluing 
	\[M  \cong M_- \cup_{N \times \R} M_+\] 
	of a global quotient the following natural equivalence holds:
	\begin{align}
		\int_{[M_+/\ga]} A \; \; \underset{\int_{[N \times \R/\ga]} A}{\bigotimes} \; \; \int_{[M_-/\ga]} A \quad \longrightarrow \quad \int_{[M/\ga]} A. \label{introeq:excision}
	\end{align} 
	We refer to Equation \eqref{introeq:excision} as the $\tensor$-\textbf{excision} property. 
\end{thm}

Our second main result characterises equivariant factorization homology via $\tensor$-excision.

\begin{thm} (Theorem \ref{thm:classification}) \label{intro:class}
	Let $\catC^{\tensor}$ be $\tensor$-cocomplete, and let $H: \gar\Orb^{G}_n \rightarrow \catC^{\tensor}$ be a symmetric monoidal functor satisfying $\tensor$-excision. We denote with $A_H$ the $\gar\Disk^{G}_n$-algebra obtained by restricting $H$ to $\gar\Disk^{G}_n$. 
	Then $H$ is computed by factorization homology with coefficients in $A_H$:
	\[ H[M/\ga] \cong  \int_{[M/\ga]} A_H. \]
\end{thm}

Our characterisation in Theorem \ref{intro:class} allows us to recover the classical equivariant homology theories of Bredon as examples of equivariant factorization homology.
Recall that Bredon equivariant homology $H^\ga_*$ in its simplest form assigns the singular homology of the Borel construction to a global quotient \cite{bredon67a,bredon67b}. 
Thus $H^\ga_*(M) = H_*(M\times \mathrm{E\ga}/\ga )$. 
Given a finite group $\ga$, any choice of coefficients $A$ for Bredon homology yields a $\gar\Disk_n$-algebra\footnote{There are no restrictions on $\rho$, and $G=GL(\R^n)$.}  $C^\ga_*(-,A)$ valued in chain complexes, which assigns Bredon equivariant chains.
\begin{cor} (Proposition \ref{prop:beh})
Bredon homology is computed by factorization homology:
\[  H^\ga_*(M,A) \cong \int_{[M/\ga]} C_*^\ga(-,A).\]
\end{cor}

The following examples show equivariant factorization homology provides a natural framework for various (higher) twisted traces. 
We let $\triv$ be the trivial representation of $\Z_2$, and consider the global quotient $[S^1/\Z_2]$ where $\Z_2$ acts on the circle by rotations. 

\begin{ex} (Proposition \ref{prop:twhh})
A unital $\kk$-algebra with involution $(A,\phi)$ is an example of a $\Z_2^{\triv}\mathrm{Disk}^{\e}_{1}$-algebra in vector spaces as well as in chain complexes.
The $\phi$-twisted Hochschild homology of A, denoted $HH_*^\phi(A)$, is defined to be $HH_*(A, {}_A A_\phi)$ where ${}_A A_\phi$ is the $A$-$A$-bimodule where the right action is twisted by $\phi$. 	
Note that $HH_0^\phi(A)$ is a twisted cocenter of $A$, in the sense that the linear dual of $HH^\phi_0(A)$ is naturally identified with the twisted traces
\[ \mathrm{Tr}^\phi(A) := \{f: A \rightarrow \kk \; | \; f(ab) = f(b \phi(a)) \; \; \forall a,b \in A\}.\]
The $\phi$-twisted Hochschild homology of A is computed by factorization homology:
\begin{align*}
 &\int_{[S^1/\Z_2]} A \simeq HH_0^\phi(A) \quad\quad\quad \text{ when evaluated in vector spaces},\\
 &\int_{[S^1/\Z_2]} A \simeq HH_*^\phi(A) \quad\quad\quad \text{ when evaluated in chain complexes}.
\end{align*}
\end{ex}

One also has a categorified version of (twisted) traces, provided by the notion of monoidal traces of \cite{dsps13}. For details see \S \ref{subsec:crt}.

\begin{ex} (Proposition \ref{prop:twdrinfeld})
A tensor category $\catA$ with monoidal involution $\Phi$ is a $\Z_2^{\triv}\mathrm{Disk}^{\e}_{1}$-algebra in $\Rex$.
The $\Phi$-twisted trace of $\catA$, denoted $\mathrm{tr}^\Phi(\catA)$, is computed by factorization homology:
\[ \int_{[S^1/\Z_2]} \catA \cong \mathrm{tr}^\Phi(\catA).\]
\end{ex}

\subsection{Factorization homology and quantum groups} \label{subsec:outlook}
Important motivation for this work is provided by the applications of factorization homology to representation theory as developed by Ben-Zvi, Brochier, and Jordan. 
In their works \cite{bzbj,bzbj2}, they compute the factorization homology of surfaces with coefficients in the $E_2$-algebra of quantum group representations. 
The functoriality of factorization homology naturally equips such invariants with (categorical) braid group actions. 
They recover the braid group actions of \cite{jordan09}, which were constructed via a generators-and-relations method there.
In \cite{jordan09} these braid group actions were used to construct representations of the type $A$ double affine Hecke algebra (abbreviated DAHA). 

In \cite{jordanma} orbifold braid group actions of an orbifold torus are constructed, via a generators-and-relations method and using type $A$ quantum symmetric pairs.
These actions are then used to construct representations of the $C^\vee C_n$ DAHA. 
In forthcoming work we will compute the equivariant factorization homology of orbifold surfaces with coefficients in the $\Z_2^\sign\Disk_2^{\e}$-algebra of quantum symmetric pair representations (as in Example \ref{ex:qsp}). 
Equivariant factorization homology naturally equips such invariants with orbifold braid group actions (Proposition \ref{prop:brgract}).
We expect this recovers the orbifold braid group actions of \cite{jordanma}.
This would then extend the constructions of braid group actions in \cite{jordanma} to quantum symmetric pairs of any type and orbifold surfaces of any genus. 

\subsection{Organisation}  

The contents of this paper are laid out as follows. 

In Section \ref{sec:orbifolds} we recall the classical definition of global quotient orbifolds, and consider various geometric structures on them.
In \S \ref{subsec:framings} we introduce a notion of framings of global quotients  (Definition \ref{frgqdef}) and discuss examples. 
We then generalise the definition of framing in \S \ref{subsec:tangstruc} to more general tangential structures, and discuss examples such as orientations.
In \S \ref{subsec:framedemb} we consider notions of morphism between framed global quotients with tangential structure. 
We first recall the classical notion of strongly framed embeddings.
Quickly thereafter, we replace these with better behaved maps we call framed embeddings (Definition \ref{def:hfremb}).
In \S \ref{subsec:cats} we define $\infty$-categories of global quotients, utilising the framed embeddings of \S \ref{subsec:framedemb}.

In Section \ref{sec:confspaces} we recall the definition of equivariant configuration spaces and explain their role in factorization homology. The main result is Theorem \ref{htype} which proves a weak equivalence between certain spaces of framed embeddings in terms of equivariant configuration spaces. The result may be of independent interest. However, we mostly use Theorem \ref{htype} as a technical result in this paper, for example in the proof of $\tensor$-excision. Thus Section \ref{sec:confspaces} can be skipped on a first reading. 

Section \ref{sec:facthom} forms the heart of the paper. 
In \S \ref{subsec:diskalg} we introduce the coefficient systems of equivariant factorization homology.
We then discuss various examples, such as $\ga$-equivariant $E_n$-algebras. 
In \S \ref{subsec:invariants} we define the equivariant factorization homology of a framed global quotient with values in a coefficient system. 
The remainder of \S \ref{subsec:invariants} is technical in nature: we provide two colimit formulas for equivariant factorization homology.
The first main result of this paper is contained in \S \ref{subsec:excision}. 
We define collar-gluings and $\tensor$-excision there, and prove that equivariant factorization homology satisfies $\tensor$-excision (Theorem \ref{thm:excision}).
In \S \ref{subsec:eht} we come to the second main result of the paper: Theorem \ref{thm:classification}, states that $\tensor$-excision uniquely characterises equivariant factorization homology.
We also discuss various consequences from this result, such as a Fubini formula for invariants. 

Section \ref{sec:applications} contains various examples of equivariant factorization homology. 
Theorems \ref{thm:excision} and \ref{thm:classification} are crucial to the results of this section. 
However, if the reader is willing to assume the results, one can read this section without knowledge of the technical details in \S \ref{sec:facthom}.  
In \S \ref{subsec:bredon} and \S\ref{subsec:crcoh} we recover classical orbifold homology theories, Bredon equivariant homology and Chen-Ruan cohomology, as examples of equivariant factorization homology.
In \S \ref{subsec:thh} we study equivariant factorization homology in the linear settings of vector spaces and chain complexes. We recover twisted cocenters and twisted Hochschild homology, as the invariant assigned to an orbifold circle.
In \S \ref{subsec:crt}, we consider a setting of $\kk$-linear categories, and recover a twisted categorical trace as the invariant assigned to an orbifold circle. 
We also prove equivariant factorization homology provides a canonical construction of categorical orbifold braid group actions.

In Section \ref{sec:variations} we discuss two variations of equivariant factorization homology that are useful for practical applications.
In \S \ref{subsec:restrsing} we explain one does not need the full data of a coefficient system to compute invariants of a specific orbifold. 
It suffices to only consider the types of singularities that occur in the orbifold, which reduces the algebraic input needed to construct invariants.    
In \S \ref{subsec:colouring} we introduce a slight generalisation of the coefficient systems, by defining colourings of singular strata in an orbifold. 
This creates more flexibility in constructing invariants of orbifolds.

\addtocontents{toc}{\protect\setcounter{tocdepth}{1}}
\subsection*{Acknowledgements}
First and foremost, I wish to thank my supervisor David Jordan for his invaluable guidance. 
I also want to thank Claudia Scheimbauer, for conversations on factorization homology and TQFTs that inspired Corollary \ref{cor:efhreduces}. 
Finally, I am grateful to Owen Gwilliam, for his sharp questions and useful feedback.  
This work was supported by the Engineering and Physical Sciences Research Council [grant no. 1633460].
This work was also supported by the European Research Council
(ERC) under the European Union's Horizon 2020 research and innovation
programme (grant agreement no. 637618).
\addtocontents{toc}{\protect\setcounter{tocdepth}{2}}

\section{Global quotient orbifolds} \label{sec:orbifolds}

In this section we introduce the geometric structures that form the set-up for equivariant factorization homology.  

Orbifolds, as introduced by Satake \cite{satake56,satake57} and Thurston \cite{thurston92}, are singular manifolds where singularities are of the form $[\R^n/\ga]$ for $\ga$ a finite group acting on $\R^n$. 
See \cite{ademleidaruan}, and references therein, for the general theory. 
So called global quotients are the simplest examples of orbifolds.
\begin{defn} \label{def:globalquotient}
	A \emph{$\ga$-global quotient} $[M/\ga]$ is a smooth manifold $M$ together with a $\ga$-action such that if $g\in \ga$ fixes all points in a connected component of $M$ then $g=e$.
\end{defn}
One studies such an orbifold $[M/\ga]$ via the equivariant geometry of $M$. 
\begin{defn}
A smooth map $f: [M/\ga] \rightarrow [N/\ga]$ between global quotients is a $\ga$-equivariant smooth map $f: M \rightarrow N$. 
\end{defn}
The philosophy of studying $[M/\ga]$ equivariantly informs our definitions of tangential structures on global quotients: they are equivariant analogues of those of manifolds. 
We'll also discuss a natural notion of morphism between global quotients with tangential structure. 
Before turning our attention to framings, we recall some basic terminology and facts.

\begin{conv} \label{conv:smoothness}
	Without further comment we will only consider smooth manifolds and smooth maps between them. So manifold abbreviates smooth manifold, embedding abbreviates smooth embedding, actions on manifolds are smooth, and so forth. 
\end{conv}

\begin{nota}
	$\ga$ denotes a given finite group, $\e$ denotes the trivial group, and $GL(n) = GL(\R^n)$. 
\end{nota}

\begin{rmk}
	We follow standard terminology \cite{thurston92,ademleidaruan} by referring to points in $M$ that are fixed by some element $g\in \ga$ as \emph{singular points} of $[M/\ga]$. 
	The \emph{smooth locus}, i.e. set of non-fixed points, is open and dense in $M$.\footnote{This follows immediately from the corresponding fact for faithful actions on connected manifolds, see e.g. \cite[Lemma 2.10]{moerdijkmrcun03}.}
\end{rmk}
\begin{nota}
	Let $X$ be a space with a $\ga$-action. For $x\in X$ we denote $\mathrm{Stab}_x := \{g\in \ga : gx = x \}$. For a subgroup $I\subset \ga$ we denote with $X^I :=\{ x\in X: \mathrm{Stab}_x = I\}$ the $I$-fixed points.
\end{nota}
Thus the smooth locus of $[M/\ga]$ is exactly $M^\e$.
\begin{nota}
	For a map $f: M \rightarrow N$ we denote with $Df: TM \rightarrow TN$ the differential.
\end{nota}
The following classical result will be useful:
\begin{lem} \label{isotrick} \emph{(Isometry Trick)} Let $[M/\ga]$ be a global quotient. If $g\in \ga$ fixes a point $m \in M$ and acts as the identity the tangent space, $ (Dg)_m = \id_{T_m M}$, then $g=e$.
\end{lem}
\begin{proof} Let $M_m$ denote the connected component of $M$ containing $m$.
	Fix a $\ga$-invariant metric on $M$ by averaging so that $g$ defines an isometry of $M_m$. Any isometry of a connected manifold that fixes a point and has identity differential there must fix all points. Therefore $g$ fixes $M_m$. By assumption of $[M/\ga]$ being a global quotient it follows that $g=e$.
\end{proof}
The Isometry Trick implies that any $\ga$-invariant open $U \subset M$ in a global quotient $[M/\ga]$ defines a global quotient $[U/\ga]$. 
For $m\in M$ we will call a $\ga$-invariant open containing $m$ a \emph{$\ga$-neighbourhood of $m$}. 

\subsection{Framings on global quotients} \label{subsec:framings}

The definition of framings for manifolds in differential geometry suggests that a framing of a global quotient should be a global section of the frame bundle. 
We will explain why this definition is not suitable, before suggesting an alternative. 
\begin{nota}
We let $\Fr (M) \rightarrow M$ denote the frame bundle of a manifold $M$. For a map $f: M\rightarrow N$ of manifolds we denote with $f_*: \Fr(M) \rightarrow \Fr(N)$ the differential of $f$.
\end{nota}
Let $[M/\ga]$ be a global quotient, then we have an induced $\ga$-action on $\Fr(M)$:
\begin{align*}
	&\Gamma \times \Fr(M) \rightarrow \Fr(M),\\
	& (g,p) \mapsto g_* (p).
\end{align*}
For a global quotient $[M/\ga]$ the Isometry Trick implies that the induced $\ga$-action on $\Fr(M)$ is free.
We view the $\ga$-equivariant projection $\pi: \Fr(M) \rightarrow M$ as a map $\pi: [\Fr(M)/\ga] \rightarrow [M/\ga]$ of orbifolds. This is called \emph{the frame bundle of} $[M/\ga]$ \cite[Def. 1.22]{ademleidaruan}. 

\begin{lem} Let $[M/\ga]$ be a global quotient with at least one non-trivial fixed point. 
Then $\pi: [\Fr(M)/\ga] \rightarrow [M/\ga]$ has no global sections.
\end{lem}
\begin{proof} 
Suppose there exists be a global section $s$ of $[\Fr(M)/\ga]$ i.e. a $\ga$-equivariant map $s: M \rightarrow \Fr(M)$ so that $\pi\circ s = \id_M$.
By assumption there is a non-trivial $g\in \ga$ fixing some point $m \in M$. 
Since $s$ is a $\ga$-equivariant global frame it follows that $g$ acts as the identity on $T_m M$. 
Lemma \ref{isotrick} implies $g$ must act as the identity on the connected component of $m$. 
By definition of $[M/\ga]$ being a global quotient we find that $g=e$, a contradiction.
\end{proof}
Thus framings in the above sense are too restrictive.
We propose a different notion of framing for global quotients, which seems unexplored in the literature.
\begin{defn} \label{frgqdef} Let $\rho: \ga \rightarrow GL(n)$ be a representation, and $[M/\ga]$ an $n$-dimensional global quotient.
A \emph{$\rho$-framing} of $[M/\ga]$ is a choice of $\ga$-equivariant vector bundle isomorphism $TM \cong M \times \R^n$ where the actions are given by $g\cdot (m,v) =  (gm, D_m(g)v)$ and $g \cdot (m,v) = (gm, \rho(g)v)$ respectively. 
We call $\rho$ the \emph{framing representation} of $[M/\ga]$.
\end{defn}

We introduce notation for two representations that we will come across often.
\begin{nota}
We denote with $\triv: \ga \rightarrow GL(n)$ the trivial representation, where $\ga$ acts trivially on $\R^n$. 
We denote with $\sign: \Z_2 \rightarrow GL(n)$ the sign representation, where $\Z_2$ acts on $\R^n$ by multiplying with $\pm 1$.  
\end{nota}

\begin{ex} \label{frameex} We give some examples of framed global quotients:
\begin{enumerate}
\item Any faithful representation $\rho: \ga \rightarrow GL(n)$ defines a global quotient $[\R^n/\ga]$. The canonical framing of $\R^n$ together with $\rho$ defines a $\rho$-framing on $[\R^n/\ga]$.
\item We view the torus $\mathbb{T}$ as the quotient $\R^2/\Z^2$. The canonical framing of $\R^2$ induces a framing on $\mathbb{T}$. The rotation action of $\Z_2$ on $\R^2$ induces an action on $\mathbb{T}$ with four fixed points. Then $[\mathbb{T}/\Z_2]$ is a $\sign$-framed global quotient.
\item We view the 6-dimensional torus $\mathbb{T}^6$ as the quotient $\C^3/\Z^6$. 
Consider the representation $\rho: \Z_2\times \Z_2 \rightarrow GL(\C^3)$ defined on generators $\sigma_1$ and $\sigma_2$ by:
\[ \rho(\sigma_1)(z_1,z_2,z_3) = (-z_1,-z_2,z_3), \quad \quad \rho(\sigma_2) = (-z_1,z_2,-z_3). \]
The action on $\C^3$ descends to $\mathbb{T}^6$ \cite{vafawitten}. Then $[\mathbb{T}^6/(\Z_2\times \Z_2)]$ is $\rho$-framed.\footnote{Where $\rho$ is the 6-dimensional real representation underlying the 3-dimensional complex representation.} 
\end{enumerate}
\end{ex}

\begin{rmk}
	Note that we recover the naive definition of framing as a special case: a global section of the frame bundle $[\Fr(M)/\ga] \rightarrow [M/\ga]$ corresponds to an $\triv$-framing.
\end{rmk}

We pause briefly to highlight the simplest type of $\rho$-framed global quotient: the local models. 
The local models are exactly the charts that can occur within a $\rho$-framed global quotient. 
We will later view general global quotients as glued out of these simpler pieces.
\begin{nota}
	We write $I \le \ga$ when $I$ is a subgroup of $\ga$. For a given representation $\rho: \ga \rightarrow GL(n)$ we write $I \leq_\rho \ga$ if $I$ is a subgroup of $\ga$ such that $\rho|_I: I \rightarrow GL(n)$ is faithful.
	We denote by $\ga/I$ the set of left cosets and by $[\ga:I]$ the \emph{index} (the number of left cosets). 
\end{nota}
\begin{defn} \label{def:locmod} 
	Given a finite group $\ga$, $I \leq_\rho \ga$, and a representation $\rho: \ga \rightarrow GL(n)$.
	We let $\ga$ act on $\ga \times \R^n$ by $g\cdot (h,x) = (gh, \rho(g)x)$ and $I$ act on the right via $(h,x) \cdot i = (hi,x)$. We denote $\DD_I := \ga \times_I \R^n$ and frame it via the canonical framing of $\R^n$. The \emph{$\ga$-local model with $I$-singularities} is the $\rho$-framed $n$-dimensional global quotient $[\DD_I/\ga]$.
\end{defn}

\begin{nota}
	We denote the connected components of $\DD_I = \amalg_{[\ga:I]} \R^n$ by $\R^n_{gI}$ for $gI \in \ga/I$.
\end{nota}

\subsection{General tangential structures} \label{subsec:tangstruc}
We quickly review $G$-structures on manifolds, as a unifying framework for tangential structures. 
Standard references are \cite{sternberg64,chern66,kobayashi72}. 

\begin{nota}
	Throughout $G\subset GL(n)$ denotes a Lie group. Elements of $\ga$ are denoted by letters like $g$ and $h$, whereas elements of $G$ are denoted by letters like $A$ and $B$.
\end{nota}
Recall, a \emph{principal $G$-bundle} on a manifold $M$ is a manifold $P$ endowed with a right $G$-action and a $G$-invariant surjective map $\pi: P \rightarrow M$, that is locally trivial, i.e. for every $m\in M$ there exists a neighbourhood $U$ of $M$ with a $G$-equivariant diffeomorphism $P|_U \cong U \times G$ respecting the fibres. Here $U \times G$ is equipped with the right $G$-action $(u,A)\cdot B = (u,AB)$.

\begin{ex} The frame bundle $\Fr (M) \rightarrow M$ of a $n$-dimensional manifold $M$, together with the canonical right action of $GL(n)$, is a principal $GL(n)$-bundle.\footnote{Recall that the canonical right action of $GL(n)$ on $\Fr (M)$ rotates the frames fiberwise.}
\end{ex}

The following well-known facts will be useful:
\begin{lem} \label{lem:leefacts} Let $G$ be a Lie group and $M$ be a manifold.
	\begin{enumerate}
		\item A manifold $P$ with a right $G$-action and $G$-invariant surjective map $\pi: P \rightarrow M$ is a principal $G$-bundle iff $\pi$ is a submersion and $G$ acts freely and transitively on the fibers of $\pi$. 
		\item If $G$ acts freely and properly on $M$, then $M \rightarrow M/G$ is a principal $G$-bundle.
		\item If $\pi: P \rightarrow M$ is a submersion and $f: P \rightarrow N$ is a  map that is constant on the fibers of $\pi$, then there is a unique map $\bar{f}: M \rightarrow N$ such that $\bar{f}\circ \pi = f$.
	\end{enumerate}
\end{lem}
\begin{proof} 
	For part one, we recall that if $\pi$ is a submersion, then every point $m \in M$ is in the domain of a local section of $\pi$ \cite[Prop. 7.16]{lee03}.
	We claim such a local section provides a trivialisation. Namely, a local section $(s:U \rightarrow P)$ provides a $G$-equivariant map
	\begin{align*} 
		&U \times G \rightarrow P,\\
		&(u,A) \mapsto s(u) \cdot A.
		\end{align*}
	Since $G$ acts freely and transitively on the fibers of $\pi$ this map is a diffeomorphism. 
	For part two and part three, see \cite[Th. 9.16]{lee03} respectively \cite[Prop. 7.18]{lee03}.
	\end{proof}
\begin{rmk}
	Properness is automatic for actions by compact Lie groups \cite[Cor. 9.14]{lee03}. In particular Lemma \ref{lem:leefacts}.2 holds for free actions by finite groups. 
\end{rmk}
\begin{rmk}
	Due to our convention of not mentioning smoothness Lemma \ref{lem:leefacts} part 3 might seem trivial. However, the crux is that the continuous map $\bar{f}: M \rightarrow N$ is smooth.
\end{rmk}
\begin{nota}
	For spaces $X$ and $Y$ with $G$-actions of some group $G$ we denote with $X\times_G Y$ the quotient space where $(gx,y) \sim (x,gy)$ for all $g\in G$, $x \in X$ and $y\in Y$.
\end{nota}
For a principal $H$-bundle $P$ where $H \subset G$, we have the \emph{induced principal $G$-bundle}:
\[ \mathrm{Ind}_H^G P := P\times_H G.\]

\begin{defn}  
	A \emph{G-structure} on a manifold $M$ is a choice of a principal $G$-bundle $P_G(M) \rightarrow M$ together with an isomorphism of principal $GL(n)$-bundles 
	\[\iota_G: \Fr(M) \cong \Ind_G^{GL(n)} P_G(M).\]
	We refer to a manifold endowed with a $G$-structure as a \emph{G-framed manifold}.
\end{defn} 
We recall the following classical examples (see \cite[\S VII.2]{sternberg64} or \cite[\S I.2]{kobayashi72} for details).
\begin{ex} \label{gstrex} Let $M$ be an $n$-dimensional manifold. 
	\begin{enumerate}
		\item A framing on $M$ corresponds to an $\e$-structure, in which case $\Fr(M) \cong M \times GL (n)$.
		\item A Riemannian metric on $M$ corresponds to an $O(n)$-structure.
		\item An orientation on $M$ corresponds to a $GL^+(n)$-structure.
		\end{enumerate}
		Let $M$ be a $2n$-dimensional manifold. 
	\begin{enumerate}
		\item An almost complex structure on $M$ corresponds to a $GL(\C^n)$-structure. 
		\item An almost symplectic structure on $M$ corresponds to a $Sp(n)$-structure.
	\end{enumerate}
\end{ex}

For a $G$-structure $(P_G(M) \rightarrow M,\iota_G)$ we will view $P_G(M)$ as a subbundle of the frame bundle as follows
\[ P_G(M) \quad \cong \quad  P_G(M) \times_G G \quad \subset \quad P_G(M) \times_G GL(n)\quad \xrightarrow{\iota_G} \quad \Fr(M). \]

\begin{rmk}
Let us re-emphasize the intuition, discussed in the introduction, that a $G$-structure corresponds to a choice of $G$-invariant frames on $M$. 
This is made precise by the embedding $P_G(M) \hookrightarrow \Fr (M)$ defined above.\footnote{In fact, S. Sternberg chooses to define a $G$-structure as a subbundle of the frame bundle in \cite{sternberg64}.}
\end{rmk}

This concludes our review of $G$-structures.
Let us now rephrase our notion of $\rho$-framings on global quotients (Definition \ref{frgqdef}) in the terminology of $G$-structures.
\begin{defn} \label{def:frgs} Let $[M/\ga]$ be a global quotient, together with a representation $\rho: \ga \rightarrow GL(n)$. 
A \emph{$\rho$-framing} of $[M/\ga]$ is an $\e$-structure on $M$ so that the induced left action of $\ga$ on $\Fr(M) \cong M \times GL(n)$ satisfies $g_*(m,A) = (gm,\rho(g)A)$.
\end{defn}
Note that the Isometry Trick implies the induced $\ga$-action on $\Fr (M)$ cannot preserve $P_\e(M) \subset \Fr (M)$, unless $\rho$ is trivial or $\ga$ acts freely on $M$. 
However, $P_\e (M)$ is preserved by a different $\ga$-action on $\Fr (M)$.
Namely, let $\ga$ act by sending $g$ to $c_g^M$, where
\begin{align}
	c^M_g: \Fr(M) &\rightarrow \Fr(M),\\
	p &\mapsto g_*(p) \cdot \rho(g^{-1}).
\end{align}
Now observe that, for a $\rho$-framed global quotient, the maps $c^M_g$ preserve $P_\e(M)$. 
\begin{nota}
		For $A \in GL(n)$ we have the adjoint action $\Ad_A: GL(n) \rightarrow GL(n), B \mapsto A B A^{-1}$.
\end{nota}

We obtain a left $\ga \ltimes_{\Ad(\rho)} GL(n)$-action on $\Fr(M)$ via
\begin{align*}
	 (g,A) \cdot p = c_g (p) \cdot A^{-1},\quad \text{and} \quad  (h,B) \cdot (g,A) = (hg, B \cdot \Ad_{\rho(h)}(A) ).
\end{align*} 

Next we will generalise Definition \ref{def:frgs} to other $G$-structures on global quotients, but first we must generalise the notion of principal $G$-bundles. 

\begin{nota} 
	For $G \subset GL_n$ we denote by $N_{ GL(n) }(G) := \{ A \in GL(n) : \Ad_A(G) \subset G \}$ the normaliser group of $G$ in $GL(n)$.
	Let $\rho: \ga \rightarrow N_{ GL(n) }(G)$ be a representation, the group $\ga \ltimes_\rho G$ denotes the semi-direct product group where $\ga$ acts on $G$ via $\Ad(\rho (-))$. 
\end{nota}

\begin{defn}
	 Let $[M/\ga]$ be a $\ga$-global quotient, and fix a representation $\rho: \ga \rightarrow N_{GL(n)}(G)$. 
	 A principal \emph{$(\rho,G)$-bundle on $[M/\ga]$} consists of 
	 \begin{itemize}
	 \item a manifold $P$ equipped with a left $\ga \ltimes_{\rho} G$-action, 
	 \item a $\ga \ltimes_{\rho} G$-equivariant surjective map $\pi: P \rightarrow M$, where $M$ is endowed with the trivial $G$-action such that for every $m\in M$ there exists a \emph{local trivialisation}. 
	 \end{itemize} 
	 By local trivialisation we mean every $m\in M$ admits a $\ga$-neighbourhood $U$ with a $\ga \ltimes_\rho G$-equivariant diffeomorphism $P|_U \cong U \times G$ respecting the fibres.
	  Here $U\times G$ is endowed with the following left $\ga \ltimes_{\rho} G$-action\footnote{Where $G$ acts inversely since we're combining the right $G$-action with the left $\ga$-action.} 
	\[(g,A) \cdot (u,B) = (gu, \Ad_{\rho(g)}(B) A^{-1}).\]
\end{defn}
\begin{rmk}
	If $G=GL(n)$, $GL^+(n)$ or $\e$ we have $N_{GL(n)}(G)=GL(n)$ so that there are no restrictions on the representation $\rho$.
\end{rmk}

\begin{rmk}
Note that after restricting $P$ to the smooth locus $M^\e$ Lemma \ref{lem:leefacts} implies that we have a principal $\ga \ltimes_{\rho} G$-bundle $P|_{M^\e} \rightarrow M^\e/\ga$.
\end{rmk}

We have the following two important properties of $(\rho,G)$-principal bundles.
\begin{lem} \label{lem:indbundle}
	Let $H \subset G$, $\rho: \ga \rightarrow N_{ GL(n) }(G) \cap N_{ GL(n) }(H)$ and $P_H \rightarrow M$ be a $(\rho,H)$-principal bundle on $M$. 
	Then $\Ind_H^G P_H$ is naturally a $(\rho, G)$-principal bundle.
\end{lem}
\begin{proof}
Indeed, let $g \in \ga$, $B \in H$, $A \in G$. The computation
\begin{align*}
	g \cdot (p, BA) &= (gp, \Ad_{\rho(g)}(BA)) = (gp, \Ad_{\rho(g)}(B) \Ad_{\rho(g)}(A)) \sim ((gp) \Ad_{\rho(g)}(B), \Ad_{\rho(g)}(A)),\\
	&= (g (pB), \Ad_{\rho(g)}(A)) = g \cdot (pB,A)
\end{align*}
shows $\Ind_H^G P_H$ has a canonical $\ga$-action defined by $g \cdot [p,A] = [gp, \Ad_{\rho(g)}(A)]$. 
By construction we obtain a $\ga \ltimes_\rho G$-action on $P_G(M)$.
To complete the proof, we observe that a $\ga \ltimes_\rho H$-equivariant trivialisation $P_H|_U \cong U \times H$ induces a $\ga \ltimes_\rho G$-equivariant trivialisation $\Ind_H^GP_H|_U \cong U \times G$. 
\end{proof} 

\begin{prop} (Covering Homotopy Theorem) \label{thm:CHT}
Let and $P_G(M) \rightarrow M$ and $P_G(N) \rightarrow N$ be principal $(\rho,G)$-bundles.
For any $\ga$-equivariant homotopy $f_t: M \rightarrow N$ with $\ga \ltimes_\rho G$-equivariant lift $F_0$ of $f_0$, there is a $\ga \ltimes_\rho G$-equivariant homotopy $F_t$ covering $f_t$. 
\end{prop}

\begin{proof}
	Mutatis mutandis, the classical proof of Steenrod works \cite[Th. 11.3]{steenrod57}. 
\end{proof}

We now come to the definition of $G$-structures on global quotients.

\begin{defn}  \label{def:rgfr}
	A \emph{$(\rho,G)$-structure} on a global quotient $[M/\ga]$ is a choice of a principal $(\rho,G)$-bundle $P_G(M) \rightarrow M$ together with a $\ga \ltimes_\rho GL(n)$-equivariant diffeomorphism
	\[\iota_G: \Fr(M) \cong \Ind_G^{GL(n)} P_G(M).\]
	A global quotient endowed with a $(\rho,G)$-structure is a \emph{$(\rho,G)$-framed global quotient}.
\end{defn} 

We give some examples of $(\rho,G)$-framed global quotients. 

\begin{ex}
	A $(\rho,\e)$-framed global quotient is exactly a $\rho$-framed global quotient. 
\end{ex}

\begin{ex}The local models of Definition \ref{def:locmod} are $\rho$-framed, and hence by Lemma \ref{lem:indbundle} naturally $(\rho,G)$-framed for all subgroups $G \subset GL(n)$ for which $\rho: \ga \rightarrow N_{GL(n)}(G)$.
\end{ex}

\begin{ex} We consider the trivial representation $\triv$. 
	\begin{enumerate}
		\item Lemma \ref{isotrick} implies a $\triv$-framed global quotient $[M/\ga]$ consists of a manifold $M$ with a free $\ga$-action and a $\ga$-equivariant frame on $M$. 
		\item An $(\triv,GL^+(n))$-framed global quotient $[M/\ga]$ consists of an oriented manifold $M$ with a free $\ga$-action preserving the orientation. 
		\item Any global quotient $[M/\ga]$ where the $\ga$-action is free admits an $(\triv,O(n))$-structure.
	\end{enumerate}
\end{ex}

\begin{ex} \label{ex:stringytopology}
	Let $M$ be $G$-framed. Then $M^{\times k}$ is canonically $G^{\times k}$-framed. Let  the symmetric group $S_k$ act on $M^{\times k}$ by permuting coordinates, and $\pi: S_k \rightarrow GL(k)$ be the permutation representation. 
	Then $[M^{\times k}/S_k]$ is naturally a $(\triv \tensor \pi, G^{\times k})$-framed global quotient.
\end{ex}

\begin{rmk} \label{rmk:localmodels}
	Not all singularities can occur in a $(\rho,G)$-framed global quotient, rather one finds that $M = \bigcup_{I \leq_\rho \ga} M^I$.\footnote{This is a consequence of our faithfulness requirement in Definition \ref{def:globalquotient}.}
	If $\rho$ is trivial, one thus only has the free local model $[\DD_\e/\ga]$.
\end{rmk}

When $G = GL^+(n)$ we will say \emph{$\rho$-oriented global quotient} rather than $(\rho,GL^+(n))$-framed global quotient. 
It makes sense to restrict to representations of the form $\rho: \ga \rightarrow GL^+(n)$, so that the $\ga$-action on $M$ is orientation preserving. In classical definitions of oriented orbifolds the $\ga$-action is always required to be orientation preserving \cite{satake56,ademleidaruan}.
\begin{ex}
	Let $\mathrm{rot}: \Z_g \rightarrow SO(2)$ be the rotation representation, and consider $\Sigma_g$ the genus $g$ surface $(g\geq2$) as a quotient of a polygon $P$ with $4g$ sides. The induced  $\Z_g$-action on $P$ descends to $\Sigma_g$ so that $[\Sigma_g/\Z_g]$ is $\mathrm{rot}$-oriented. It has two $\Z_g$-singularities.  
\end{ex}

\begin{ex}
Let $\Z_2$ rotate a genus $g$ surfaces $\Sigma_g$ as in Figure \ref{fig:surface}. Then $[\Sigma_g/\Z_2]$ is a $\sign$-oriented global quotient, for the sign representation $\sigma$.
\end{ex}

\begin{figure}[h]
	\includegraphics[trim={.5cm 13.5cm .5cm 11.5cm},clip,scale=.6]{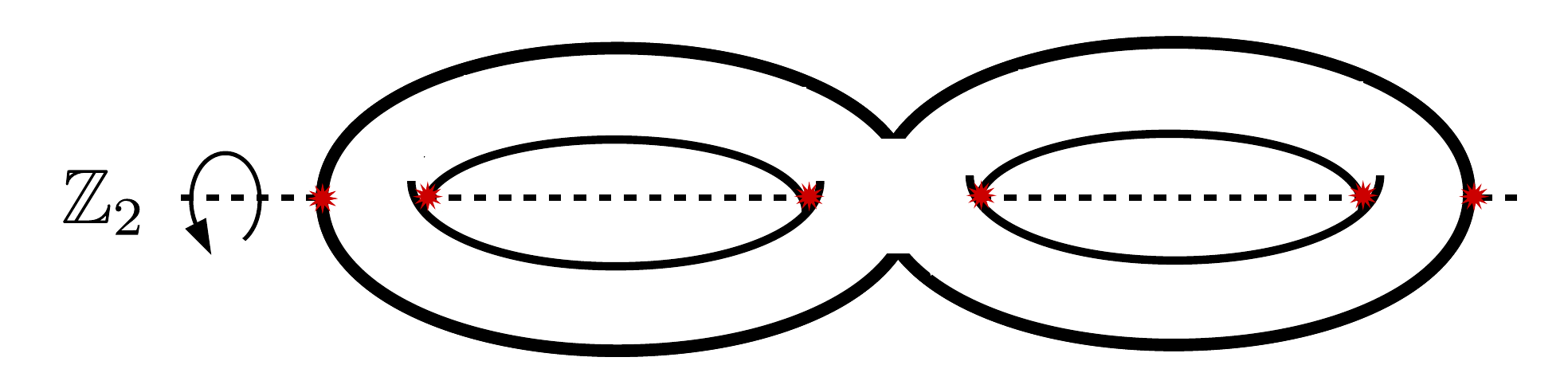}
	\caption{An oriented $\Z_2$-action on a genus two surface, with fixed points drawn as red stars.}
	\label{fig:surface}
\end{figure}

Though any manifold has a canonical $GL(n)$-structure, given by the frame bundle itself, the same is not true for global quotients as the next example shows.

\begin{ex}
	Consider the M\"obius band $M$ as a quotient space of the square by inversely gluing the top and bottom edge. 
	The rotation action of $\Z_2$ on the square descends to $M$ and acts via $\sign$ at the isolated fixed point, but acts via the representation $\begin{pmatrix}
	1 & 0 \\ 0 & -1 \end{pmatrix}$ at the fixed line. 
	Therefore, there is no representation $\rho$ such that $[M/\Z_2]$ admits a $(\rho,GL(2))$-structure. 
\end{ex}
\subsection{Framed embeddings between global quotients} \label{subsec:framedemb}

We now consider notions of morphism between $(\rho,G)$-framed global quotients.
Recall, a map between two given principal $G$-bundles, $\pi_i: P_i \rightarrow M_i$, is a $G$-equivariant map $f: P_1 \rightarrow P_2$.\footnote{\label{underlyinggmap} Some authors prefer to define $G$-bundle maps as pairs $(f_{b},f)$ where $f$ is as before and $f_{b}: M_1 \rightarrow M_2$ is a map such that $\pi_2\circ f = f_b \circ \pi_1$. However, $f$ canonically induces $f_b$ via Lemma \ref{lem:leefacts}.}
For $G$-framed manifolds $M$ and $N$, induction of principal $G$-bundles provides an inclusion 
\[ \Ind_G^{GL(n)}: \Map^G(P_G(M),P_G(N)) \hookrightarrow \Map^{GL(n)}(\Fr(M),\Fr(N)) \] where the image is the subspace of those $GL(n)$-bundle maps that send $P_G(M)$ to $P_G(N)$.
This naturally leads to the classical notion of strongly framed maps.

\begin{defn}
	A \emph{strongly framed embedding} between $(\rho,G)$-framed global quotiens $[M/\ga]$ and $[N/\ga]$ is an embedding $f:[M/\ga] \rightarrow [N/\ga]$ such that $f_*(P_G(M)) \subset P_G(N)$.
\end{defn}

\begin{rmk}
The name strongly framed is non-standard; we use it to highlight the difference between these maps and our framed embeddings introduced in Definition \ref{def:hfremb}.
\end{rmk}

\begin{ex} Let $\rho$ be any $\ga$-representation.
	\begin{enumerate}
		\item A strongly framed embedding $f: [M/\ga] \rightarrow [N/\ga]$ between $\rho$-framed global quotients is a $\ga$-equivariant embedding $f: M \rightarrow N$ such that the diagram 
		\begin{tikzcd}
		TM \arrow[d,"\cong"] \arrow[r,"{Df}"] & TN \arrow[d,"\cong"]\\
		M \times \R^n \arrow[r,"{(f,\id_{\R^n})}"] & N \times \R^n
	\end{tikzcd}
		commutes.
		\item A strongly framed embedding $f: [M/\ga] \rightarrow [N/\ga]$ between $(\rho,GL^+(n))$-framed global quotients is a $\ga$-equivariant oriented embedding $f: M \rightarrow N$. 
		\item A strongly framed embedding $f: [M/\ga] \rightarrow [N/\ga]$ between $(\rho,GL(n))$-framed global quotients is a $\ga$-equivariant embedding $f: M \rightarrow N$.
	\end{enumerate}
\end{ex}

The notion of strongly framed embedding is very rigid for certain structure groups:

\begin{ex} \label{rigidex}
	Consider $[\DD_\ga/\ga]$ the $\rho$-framed local model of Definition \ref{def:locmod}.
	\begin{enumerate}
		\item The map $\id_{\R^n}$ is the only strongly framed self-embedding of $[\DD_\ga/\ga]$.
		\item Let $[M/\ga]$ be a $\rho$-framed global quotient and let $\{v_i\}_{1\leq i \leq n}$ be the vector fields on $M$ trivializing $TM$. Let $U \subset M$ be an open such that $[v_i,v_j]|_{U} \neq 0$ for some $i$ and $j$. Then there are no strongly framed embeddings of $[\DD_\ga/\ga]$ into $[M/\ga]$ that cover $U$.\footnote{This is restrictive: the $n$-torus is the only compact $n$-manifold that admits a global frame whose vector fields commute everywhere.}
		\item Equip $\R^2$ with its canonical $O(2)$-structure, and the sphere with the standard $O(2)$-structure.\footnote{The standard $O(2)$-structure on the sphere is induced from the canonical $O(3)$-structure on $\R^3$.} There are no strongly framed embeddings $\R^2 \rightarrow S^2$ \cite[Ex. V.8.2]{andrade10}.
	\end{enumerate}
\end{ex}

We wish to view global quotients as glued from the local models of Definition \ref{def:locmod}. However, Example \ref{rigidex}.2 shows not all $(\rho,G)$-framed global quotients can be covered by local models via strongly framed embeddings.
To salvage this we have to consider a more general type of map than the strongly framed embeddings.
\begin{rmk}
	This is analogous to the use of embeddings between manifolds that are framed up to a specified homotopy. This is standard in factorization homology \cite[Constr. 4.1.1.18]{lurietft}, \cite[\S V.8]{andrade10}, \cite[Def. 2.7]{af15}.
\end{rmk}
In order to weaken the strongly framed maps, we first consider the \emph{space} of strongly framed embeddings. 
We start by fixing notation and conventions.
\begin{conv}
	Any set of maps will be topologised using the compact-open topology. 
\end{conv}

\begin{nota}
	$\Map(M,N)$ denotes the space of maps between given manifolds $M$ and $N$.
\end{nota}

\begin{ex}	Let $[M/\ga]$ and $[N/\ga]$ be global quotiens. 
	\begin{enumerate}
		\item We have an induced action of $\ga$ on $\Emb(M,N)$ by letting $(g\cdot f)(x):= gf(g^{-1}x)$.	
		The space of $\ga$-equivariant embeddings from $M$ to $N$ are exactly the $\ga$-fixed points $\Emb(M,N)^\ga$ and will be denoted $\Emb^\ga(M,N)$.
		\item Similarly, the space of $G$-bundle maps is $\Map^G(P_G(M),P_G(N))$.
	\end{enumerate}
\end{ex}

\begin{rmk}
We have two distinct actions of $\ga$ on $\Fr (M)$, given by $g_*$ and $c_g$, that induce two different actions on $\Map( \Fr(M), \Fr(N) )$. However, for any bundle map $f$ one has $c^N_g \circ f \circ c^M_{g^{-1}} = g_* \circ f \circ g_*^{-1}$. Therefore, the two $\ga$-actions coincide on $\Map^{GL(n)}(\Fr(M),\Fr(N))$. 
\end{rmk}

We can now characterise the space of strongly framed embeddings as follows. 
Let 
\[D: \Emb^\ga (M,N) \rightarrow   \Map^{\ga \times GL(n)}(\Fr(M),\Fr(N)), \quad \quad f \mapsto f_*\] 
be the continous map that sends an embedding to its differential. 
\begin{prop} \label{frpb}
	Let $[M/\ga]$ and $[N/\ga]$ be $(\rho,G)$-framed global quotients. 
	The space of strongly framed embeddings from $[M/\ga]$ to $[N/\ga]$ is naturally identified with the pullback of the diagram 
	\begin{equation} \label{pbdiagram}
	\begin{tikzcd}
	&\Map^{\ga\ltimes G}(P_G(M), P_G (N)) \arrow[hookrightarrow,d] \\
	\Emb^\ga (M,N) \arrow[r,"{D}"] &  \Map^{\ga \times GL(n)}(\Fr (M),\Fr (N) ).
	\end{tikzcd}
	\end{equation}
\end{prop}
\begin{proof}
	Immediate from the definition: the pullback describes those $\ga$-equivariant embeddings whose differential preserves the $G$-structure.
\end{proof}
All spaces of the Diagram \ref{pbdiagram} have natural maps to the space $\Map^\ga(M,N)$ making two triangles commute as in Proposition \ref{hpbmodel}.\footnote{For a $G$-bundle map we are taking the induced map on the base space, see footnote \ref{underlyinggmap}.}
We can then weaken the notion of strongly framed embedding by using the relative homotopy pullback.
\begin{conv}
	We use the classical Quillen model structure on topological spaces i.e. weak equivalences are weak homotopy equivalences and fibrations are Serre fibrations. 
\end{conv}
\begin{defn} \label{def:hfremb}
	Let $[M/\ga]$, $[N/\ga]$ be two $(\rho,G)$-framed global quotients. The space of \emph{framed embeddings}, denoted $\EEmb^{G}_{\fr}([M/\ga],[N/\ga])$, is defined to be the homotopy pullback of Diagram \ref{pbdiagram}
	over the space $\Map^\ga(M,N)$.
\end{defn}
To clarify Definition \ref{def:hfremb} we will recall the definition and properties of (relative) homotopy pullbacks. We follow the exposition in \cite[\S V.9]{andrade10}.
Note that the relative homotopy pullback, or homotopy pullback in the over category, is only defined up to weak equivalence. 
However, there is a convenient model we can make use of.
\begin{nota}
	Let $[0,1]$ denote the unit interval. For a topological space $X$ let $X^{[0,1]}$ denote the space of continuous maps from $[0,1]$ to $X$ endowed with the compact-open topology.
\end{nota}

\begin{prop} \cite[\S V.9]{andrade10} \label{hpbmodel} Let 
	\begin{center}
		\begin{tikzcd}
			& Y \arrow[d,"f"] \arrow[rdd, bend left, "\pi_Y"]& \\
			X \arrow[r,"g"] \arrow[rrd, bend right, "\pi_X"] & Z \arrow[rd, "\pi_Z"]& \\
			& & W
		\end{tikzcd}
	\end{center}
	be a diagram of topological spaces, such that $\pi_Y$ and $\pi_Z$ are fibrations. 
	\begin{enumerate}
	\item A model for the homotopy pullback in the over-category $Top_{/W}$ is given by the space
	\[ \left\{ (x,\gamma,y) \in X \times Z^{[0,1]} \times Y: 
	\begin{array}{c} \gamma_0 = g(x), \quad \gamma_1 = f(y), \\
	\pi_Z(\gamma_t) = \pi_X(x) = \pi_Y(y) \; \forall t \end{array} \right\}.\]
	\item The canonical projection of the relative homotopy pullback onto $X \times_W Z$ is a fibration.
	\end{enumerate}
\end{prop} 
Thus the relative homotopy pullback sits inside the usual model of the homotopy pullback of $X \rightarrow Z \leftarrow Y$ as the subspace whose homotopies $\gamma$ have constant underlying path in $W$. As an object of $Top_{/W}$ the relative homotopy pullback has a canonical map to $W$. 
\begin{prop} \cite[\S V.9]{andrade10}\label{hpbinclusion}
	Let $X,Y,Z,W$ be as in Proposition \ref{hpbmodel}. The inclusion of the relative homotopy pullback into the homotopy pullback $X \times_Z^h Y$ is a weak equivalence. 
\end{prop}

To interpret Definition \ref{def:hfremb}, we use:
\begin{prop} \label{hpbserrefibr}
	The projections 
	\begin{align*}
	&\Map^{\ga\ltimes G} (P_G (M), P_G(N)) \rightarrow \Map^\ga (M,N), \\
	&\Map^{\ga \times GL (n)}(\Fr (M),\Fr (N)) \rightarrow \Map^\ga (M,N), 
	\end{align*} 
	sending a map to the induced map on the base are fibrations.
\end{prop}
\begin{proof}
	We adapt the standard proof due to of I.M. James \cite{james63}.
	We need to show that we can always find a filler for diagrams of the form
	\begin{tikzcd}
		D \arrow[d, "i_0"] \arrow[r] & \Map^{\ga\ltimes G} (P_G (M), P_G(N)) \arrow[d] \\
		D \times [0,1] \arrow[r] \arrow[ru,dotted] & \Map^\ga (M,N)
	\end{tikzcd}
	where $D$ is some $m$-dimensional cube. If we endow $D$ with a trivial $\ga \ltimes_\rho G$-action we use the equivariant inner hom to obtain the diagram
	\begin{tikzcd}
		P_G (M) \times D   \arrow[d, "i_0"] \arrow[r] & P_G(N) \arrow[d] \\
		P_G(M) \times D \times [0,1] \arrow[r] \arrow[ru,dotted] \arrow[d]& N \\
		M \times D \times [0,1] \arrow[ru]
	\end{tikzcd}
	Proposition \ref{thm:CHT} provides a $\ga \ltimes_\rho G$-equivariant homotopy of bundle maps $P_G(M) \times D \times [0,1] \rightarrow P_G(N)$.
	We now use the equivariant inner hom again to transform back and obtain the desired lift $D \times [0,1] \rightarrow \Map^G(P_G(M), P_G(N))$.
	The second proof is analogous. 
\end{proof}
Proposition \ref{hpbmodel} immediately gives the following corollary.
\begin{cor} \label{hfrmodel}
	A model for the space $\EEmb^G_{\fr}([M/\ga],[N/\ga])$ is given by the space of triples $(f,\gamma,r)$ where $f: [M/\ga] \rightarrow [N/\ga]$ is an embedding, $r: \Fr(M) \rightarrow \Fr(N)$ is a $\ga$-equivariant bundle map so that $r(P_G(M)) \subset P_G(N)$ and $\gamma$ is a homotopy of $\ga$-equivariant bundle maps between $f_*$ and $r$ that lies over the constant path $f$ in $\Map^\ga(M, N)$. 
\end{cor}

We will study the (spaces of) framed embeddings further in Section \ref{sec:confspaces}, and show they are better behaved than the strongly framed embeddings. 
In particular, in Corollary \ref{cor:cover}, we show that any $(\rho,G)$-framed global quotient is covered by local models via framed embeddings.

\begin{rmk} \label{rmk:underlyingemb}
	Proposition \ref{hpbinclusion} implies another model for $\EEmb^G_{\fr}([M/\ga],[N/\ga])$ is given by the homotopy pullback. We prefer the relative homotopy pullback so that a framed embedding has a canonical underlying embedding.
\end{rmk}

The space of strongly framed embeddings includes into the space of framed embeddings as the subspace of triples $(f,c,f_*)$, where $c$ is constant at $f_*$.
Sometimes this inclusion is a homotopy equivalence.

\begin{prop} \label{prop:hfrisgfr}
	Let $G$ be $GL(n)$ or $GL^+(n)$. The inclusion of the space of strongly framed embeddings into the space of framed embeddings is a weak equivalence.
\end{prop}
\begin{proof}
	For $GL(n)$ and $GL^+(n)$ the inclusion 
	\[\Map^{\ga \ltimes G} (P_G(M), P_G (N)) \rightarrow \Map^{\ga  \times GL(n)}(\Fr (M),\Fr (N) )\]
	is respectively the identity or the inclusion of a connected component.
	In particular, the inclusion is a fibration in both cases. 
	Consequently, the homotopy pullback can be modeled by a strict pullback.
	Now apply Proposition \ref{frpb}.
\end{proof}

\subsection{Categories of global quotients with tangential structure} \label{subsec:cats}

\begin{conv} 
	From now on we will always use the concrete model of Proposition \ref{hfrmodel} as our model for the space of framed embeddings.
\end{conv} 

Composition of framed embeddings is then easily defined by
\[(f,\gamma,r) \circ (f',\gamma',r') := (f \circ f', \gamma \circ \gamma',r \circ r').\]

Note this composition is strictly associative, so that we can define a topological category of framed global quotients. In contrast, in \cite{af15} the framed embeddings between topological framed manifolds do not compose strictly associatively, and hence naturally define an $\infty$-category rather than a topological category.

\begin{defn}
	Let $\rho: \ga \rightarrow N_{GL(n)}(G)$ be a group homomorphism. 
	The topological category $\gar\Orb^{G}_n$ has as objects $n$-dimensional $(\rho,G)$-framed global quotients and morphisms spaces are the spaces of framed embeddings. 
\end{defn}
\begin{conv}
	We allow the empty $(\rho,G)$-framed global quotient $\emptyset \in \gar\Orb^{G}_n$. 
\end{conv}
\begin{rmk}
	As the hom-spaces in $\gar\Orb^{G}_n$ are only defined up to homotopy equivalence, the topological category is only defined up to weak equivalence. Correspondingly, the associated $\infty$-category is defined up to equivalence of $\infty$-categories.
\end{rmk}

\begin{rmk} \label{rmk:trivrep}
	If $\rho = \triv$ we find that $c_g = g_*$ and obtain a $\ga \times G$-action on $P_G(M)$. Moreover, the action of $\ga$ on $M$ must be free by Remark \ref{rmk:localmodels}.
\end{rmk}

\begin{nota} Motivated by Remark \ref{rmk:trivrep} we will write $\ga\Quot^{G}_n$ rather than $\ga^\triv\Orb^{G}_n$. We call the objects \emph{$G$-framed free $\ga$-quotients}.
\end{nota}
Note that for two $(\rho,G)$-framed global quotients $[M/\ga]$ and $[N/\ga]$ the disjoint union $[M \amalg N/\ga]$ canonically inherits the structure of a $(\rho,G)$-framed global quotient. The category $\gar\Orb^{G}_n$ is symmetric monoidal with tensor product $\amalg$ and unit object $\emptyset$. 
\begin{rmk} \label{rmk:starismflds}
	The category $\e^\triv\Orb^{G}_n$ recovers the $\infty$-category of $G$-framed smooth manifolds as in \cite{andrade10,af15}.\footnote{To be precise, in \cite{af15} topological manifolds are considered, but by smoothing theory outside of dimension 4 the category of smooth and topological $G$-framed manifolds coincide. In dimension 4 we recover the subcategory of $G$-framed smooth manifolds and smooth embeddings. See also \cite[Remark 3.29]{af15}.}
\end{rmk}

\begin{nota}
Motivated by Remark \ref{rmk:starismflds} we will henceforth denote $\e^\triv\Orb^{G}_n$ by $\Mfld_n^G$.
Furthermore, we denote the categories $\gar\Orb^{G}_n$ for $G=\e$, $GL^+(n)$, $\mathrm{GL}(n)$ respectively by $\gar\Orb^{\mathrm{fr}}_n$, $\gar\Orb^{\mathrm{or}}_n$ and $\gar\Orb_n$.
\end{nota}

We will not be much concerned with relations between the different categories $\gar\Orb^G_n$ for different groups $\ga$, but let us consider one specific example.

\begin{prop} \label{prop:timesandquotient}
	There are canonical symmetric monoidal functors 
	\[ \ga \times: \Mfld_n^G \rightarrow \ga\Quot_n^G \quad \text{and} \quad /\ga: \ga\Quot_n^G \rightarrow \Mfld_n^G \]
	such that $/\ga \circ \ga\times \simeq \id_{\Mfld_n^G}$.
\end{prop}

\begin{proof}
	We will construct these functors as enriched functors i.e. functors at the level of the topological categories. 
	Let us first construct the functor $\ga \times $. For a $G$-framed manifold $M$ the manifold $\ga \times M$ is naturally $G$-framed and has a free $\ga$-action making $[\ga \times M/\ga]$ a $G$-framed free $\ga$-quotient. The obvious maps 
	\begin{align*}
	&\Emb(M,N) \rightarrow \Emb^\ga(\ga \times M, \ga \times N),\\
	&\Map^G(P_G(M),P_G(N)) \rightarrow \Emb^{\ga \times G}(P_G(\ga \times M), P_G(\ga \times N)),\\
	&\Map^{GL(n)}(\Fr(M), \Fr(N)) \rightarrow \Map^{\ga \times GL(n)}(\Fr(\ga \times M), \Fr(\ga \times N))
	\end{align*}
	are functorial and all lie over the obvious map $\Map(M,N) \rightarrow \Map(\ga \times M,\ga \times N)$. Hence we have an induced functorial map 
	\[\EEmb^G_{\fr}(M,N) \rightarrow \EEmb^G_{\fr}([\ga \times M/\ga],[\ga \times N/\ga]).\]
	
	Next let us construct the functor $/\ga$. Recall that for $[M/\ga] \in \ga\Quot_n^G$ the group $\ga$ acts freely on $M$. Hence $M\rightarrow M/\ga$ is a principal $\ga$-bundle by Lemma \ref{lem:leefacts} and in particular a submersion. Since $\rho$ is the trivial representation $g_*(P_G(M)) = c_g(P_G(M)) \subset P_G(M)$ for all $g\in \ga$. We then have two commuting free actions of $\ga$ and $G$ on $P_G(M)$. Consider the commuting diagram
	\begin{center}
		\begin{tikzcd}
			P_G(M) \arrow[d] \arrow[r] & M \arrow[d]\\
			P_G(M)/\ga \arrow[r] & M/\ga
		\end{tikzcd}
	\end{center}
	Since three out of four arrows are submersions the induced map $P_G(M)/\ga \rightarrow M/\ga$ is also a submersion by Lemma \ref{lem:leefacts}. 
	Thus the manifold $M/\ga$ has a canonical $G$-structure. We define $/\ga$ on objects by sending $[M/\ga] \mapsto M/\ga$. 
	We will define $/\ga$ on maps as follows.  
	Lemma \ref{lem:leefacts} allows us to functorially assign maps induced maps on quotients:
	\begin{align*}
	&\Emb^\ga(M,N) \rightarrow \Emb(M/\ga, N/\ga),\\
	&\Map^{\ga \times G} (P_G(M),P_G(N) \rightarrow \Emb^G(P_G(M)/\ga, P_G(N)/\ga),\\
	&\Map^{\ga \times GL(n)}(\Fr(M), \Fr(N)) \rightarrow \Map^{GL(n)}(\Fr(M)/\ga, \Fr(N)/\ga)
	\end{align*}
	lying over the map $\Map^\ga(M,N) \rightarrow \Map(M/\ga,N/\ga)$. 
	Hence we have an induced map 
	\[\EEmb^G_{\fr}([M/\ga],[N/\ga]) \rightarrow \EEmb^G_{\fr}(M/\ga,N/\ga).\]
	It is clear that $\ga \circ \ga \times  \simeq \id_{\Mfld_n^G}$ and that the functors are symmetric monoidal. 	
\end{proof}

	We refer to objects in the essential image of the functor $\ga \times$ in $\ga\Quot_n^G$ as \emph{$G$-framed trivial $\ga$-quotients}.

\section{Equivariant configuration spaces} \label{sec:confspaces}
Configuration spaces are manifolds parametrising finite sets of points inside a given manifold.
They appear naturally in factorization homology as a model for the mapping spaces of finite collections of disks into a manifold. 
Similarly, equivariant configuration spaces naturally appear in equivariant factorization homology. 
Namely, a framed embedding of some local models into a global quotient is encoded, up to homotopy, by the position of the images of the origins and the derivatives there.
As a result, the space of framed embeddings is weakly equivalent to a principal $G$-bundle over an equivariant configuration space. 
This is made precise in Theorem \ref{htype}, which is the central result of this section. 
The result implies that any global quotient is covered by local models via framed embeddings (Corollary \ref{cor:cover}).

We now give precise definitions. For a manifold $M$ the \emph{configuration space of $k$ ordered points in $M$} is the open submanifold $F_kM \subset M^{\times k}$ defined by
\[ F_k M := \{ (z_1,\dots,z_k)\in M^{\times k} : z_i \neq z_j \text{ if } i \neq j \}.\]
There are different flavours of configuration spaces in equivariant topology, see \cite{xicot97,rourkesanderson} and references therein. 
Our definitions coincide with \cite[Def. 2.5]{hill17}.
\begin{defn}
	Let $[M/\ga]$ be a global quotient and let $I\leq \ga$ be a subgroup.
	\emph{The configuration space of $k$ ordered $I$-singular points} in $[M/\ga]$ is the open submanifold $F^I_k[M/\ga] \subset M^{\times k}$ defined by
		\[F_k^I [M/\ga] := \{ (z_1,\dots,z_k)\in M^{\times k} : \forall i \; \; z_i \in M^I \text{ and } \forall g\in \ga \; \; g z_i \neq z_j \text{ if } i \neq j \}.\]
\end{defn}

We then immediately get a corresponding natural notion of braid groups.
\begin{defn}Let $[\Sigma/\ga]$ be a two-dimensional global quotient. \label{def:braidmcg}
	We denote with $PB_k[\Sigma/\ga]$ the fundamental group of $F_k^\e[\Sigma/\Z_2]$ and call it the \emph{pure orbifold braid group}. 
	The group $\ga^{\times k} \ltimes S_k$ has an obvious action on  $F_k^\e[\Sigma/\ga]$. We denote with $B_k[\Sigma/\ga]$ the fundamental group of the quotient space $F_k^\e[\Sigma/\Z_2]/(\ga^{\times k} \ltimes S_k)$ and call it the \emph{orbifold braid group}.
\end{defn} 

\subsection{Bundles of fixed points} 
We now define natural principal $G$-bundles over the equivariant configuration spaces. They play a central role in Theorem \ref{htype}.
It will be useful to view equivariant configuration spaces as fixed-point sets in normal configuration spaces:
\begin{lem} \label{lem:fixedpointconf}
	Let $\ga$ act on the $M_{k[\ga:I]}$ by acting on the points in $M$ and acting on the indices via the $\ga$-action on $\ga/I$. 
	Then we have a canonical homeomorphism
	\begin{align*}
		F_{k [\ga:I]}(M)^\ga &\rightarrow F^I_k[M/\ga],\\
		(z_{i,gI})_{1\leq i\leq k}^{gI \in \ga/I,} &\mapsto (z_{1,eI},\dots,z_{k,eI})
	\end{align*}
\end{lem}
\begin{proof}
Choose representatives $g_j$ of $gI \in \ga/I$ and observe that the map
\[ (z_1,\dots,z_k) \mapsto (g_jz_i)_{1 \leq j \leq [\ga:I], 1 \leq i \leq k}\]
is well-defined on $F^I_k[M/\ga]$, independent of the representatives and defines the inverse. 
\end{proof}

\begin{nota} For $I \leq \ga$ and $\rho: \ga \rightarrow N_{GL(n)}(G)$.
	We denote with $C_{G}(I) \subset G$ the centraliser sub Lie group $C_G(I) := \{A \in G: \rho(g)A = A \rho(g) \; \; \forall g \in I\}$.
\end{nota}
\begin{prop} \label{prefixedpointbundle} Let $[M/\ga]$ be a $(\rho,G)$-framed global quotient, $I\leq_\rho \ga$, and restrict the $\ga$-action on $P_G(M)$ to $I$.
	Then the $I$-fixed points of the principal $G$-bundle $P_G(M) \rightarrow M$ define a principal $C_G(I)$-bundle denoted
	\[ P_G^I(M^I) \rightarrow M^I.\]
\end{prop}
\begin{proof}
	Recall that $G$ acts transitively on the fibers. Observe that if two $I$-fixed points in $P_G(M)$ lie the same fiber they must be related by a unique element in $C_G(I) \subset G$.
	Moreover, $P^I_G(M^I)$ clearly projects down to $M^I$ so that we obtain a surjective projection
	\[ P^I_G(M^I) \rightarrow P^I_G(M^I)/C_G(I) = M^I. \]
	Recall an action of a Lie group $H$ on a manifold $N$ is proper iff $\{h \in H: hK \cap K \neq \emptyset \}$ is compact for all compact $K \subset N$ \cite[9.12]{lee03}.
	One then easily verifies that the closed subgroup $C_G(I) \subset G$ acts properly on $P^I_G(M^I)$ since $G$ acts properly on $P_G(M)$.
	Consequently, $P_G^I(M^I) \rightarrow M^I$ is a principal $C_G(I)$-bundle by Lemma \ref{lem:leefacts}.
\end{proof}

Recall that if a manifold $M$ has a $G$-structure $P_G(M) \rightarrow M$ then any product manifold $M^{\times k}$ canonically inherits a $G^{\times k}$-structure $P_{G^{\times k}}(M^{\times k})$, and hence so do the open submanifolds $F_k M \subset M^{\times k}$.
Following Lemma \ref{lem:fixedpointconf} we consider the natural $\ga$-action on $P_{G^{ \times k[\ga:I] } } (F_{k[\ga:I]}M)$ induced from $c_g$ and the action of $\ga$ on the coordinate labels $\ga/I$. We then also have:
\begin{cor} \label{fixedpointbundle} The $I$-fixed points of the $G^{\times k[\ga:I]}$-bundle $P_{G^{ \times k[\ga:I] } } (F_{k[\ga:I]}M) \rightarrow F_{k[\ga:I]}M$ defines a principal $C_{G}(I)^{\times k}$-bundle denoted
\begin{align}
	P^I_{G}(F_k^I[M/\ga]) \rightarrow F_k^I[M/\ga].
\end{align}
\end{cor}
\begin{proof}
	Combine Lemma \ref{lem:fixedpointconf} and Proposition \ref{prefixedpointbundle}.
\end{proof}
\begin{nota}
	The special case $G=GL(n)$ is denoted $\Fr^I(F_k^I[M/\ga]) \rightarrow F_k^I[M/\ga]$.
\end{nota}
\subsection{Spaces of framed embeddings} \label{subsec:sofe}

We now come to the main result of this section.
Using a `derivative at the centers map' $D_0^G$ to be defined below, we have:
\begin{thm} \label{htype}
Let $[M/\ga]$ be $(\rho,G)$-framed.
The derivative at the centers map is a fibration and weak equivalence
	\begin{align} 
		D_0^G: \EEmb^G_{\fr} \left(\coprod_{I \leq_\rho \ga} [\DD_I/\ga]^{\amalg k_I},[M/\ga] \right) \longrightarrow \prod_{I \leq_\rho \ga} P^I_{G}(F^I_{k_I}[M/\ga] )
	\end{align}
for any choices $k_I \geq 0$.
\end{thm}
Before we turn to the proof, let us discuss two corrolaries. 
\begin{cor} \label{cor:cover}
Any $(\rho,G)$-framed global quotient can be covered by local models via framed embeddings.
\end{cor}
\begin{proof}
For any subgroup $I$ of $\ga$ we have the map
\[ D_0^G: \EEmb^G_{\fr} ([\DD_I/\ga], [M/\ga]) \rightarrow P_G^I(M^I),\]
which is a surjection since it is a fibration and weak equivalence. Postcomposing with the projection $P_G^I(M^I) \rightarrow M^I$ then yields a surjection $\EEmb^G_{\fr} ([\DD_I/\ga], [M/\ga]) \rightarrow M^I$ given by mapping a framed embedding $f$ to $f(0)$ the image of the center $0\in \R^n_{eI} \subset \DD_I$.
\end{proof} 

\begin{cor}
Let $\ga$ act via representation $\rho$ on $\R^n$, and let $\Emb_{\mathrm{or}}^\ga(\R^n,\R^n)$ denote the space of oriented self-embeddings of $\R^n$. We have weak equivalences
\begin{align*} 
\Emb^\ga(\R^n,\R^n) \simeq C_{GL(n)}(\ga),\quad \quad \Emb^\ga_{\mathrm{or}}(\R^n,\R^n) \simeq C_{GL^+(n)}(\ga).
\end{align*}
\end{cor}
\begin{proof}
The proofs of both statements are identical, so let us only prove the second weak equivalence holds.  
By Proposition \ref{prop:hfrisgfr} and Theorem \ref{htype} we have weak equivalences
\[ \Emb^\ga_{\mathrm{or}}(\R^n,\R^n) \simeq \EEmb^{GL^+(n)}_{\fr}([\R^n/\ga],[\R^n/\ga]) \simeq P_{GL^+(n)}^\ga (\R^n)^\ga. \] 
We observe that the $\ga$-fixed points $(\R^n)^\ga$ are a linear subspace of $\R^n$ and hence contractible.
As a result, the fixed points bundle is trivial bundle: $P_{GL^+(n)}^\ga (\R^n)^\ga \cong (\R^n)^\ga \times C_{GL^+(n)}(\ga)$. 
Using again that $(\R^n)^\ga$ is contractible we then find
\[ P_{GL^+(n)}^\ga (\R^n)^\ga \cong (\R^n)^\ga \times C_{GL^+(n)}(\ga) \simeq C_{GL^+(n)}(\ga). \] 
\end{proof}

In the remainder of the section we will prove Theorem \ref{htype}. Our proof strategy is a generalisation of that in \cite[Proposition 14.4]{andrade10}, which considers the case $\ga = \e$.
Consider a local model $\DD_I$ with its $G$-structure $\DD_I \times G$. For a disjoint union $\amalg_k \DD_I$ we denote with $0_i$ the center of the disk $\R^n_{eI}$ in the $i$th copy of $\DD_I \subset \amalg_k \DD_I$.
We define an evaluation at the centers map as follows
\begin{align*} 
	\ev_0: \Map^{\ga \ltimes G}( \amalg_k \DD_I\times G, P_G(M)) &\rightarrow P_{G^{\times k}} (M^{\times k}), \\
	f &\mapsto \big(f(0_1,e),\dots,f(0_k,e)\big),
\end{align*}
in particular we have evaluation maps for $G=GL(n)$ and the frame bundle. 

\begin{lem}
	The evaluation maps extend naturally to evaluation maps
	\begin{align}
		\ev_0: \Map^{\ga \ltimes G}( \coprod_{I \leq_\rho \ga}  (\DD_I\times G)^{\amalg k_I}, P_G(M)) \rightarrow \prod_{I \leq_\rho \ga} P^I_{G} (M^I)^{\times k_I}. \label{eac}
	\end{align}
	that define weak equivalences for any choices $k_I \geq 0$.
\end{lem}

\begin{proof}
	One extends the evaluation maps by evaluating at each $I\leq_\rho \ga$ component separately. 
	Now note these map into disjoint $P^I_{G} (M^I)^{\times k}$ components. 
	It remains to show these maps define weak equivalences.
	It suffices to consider $k_I = 0$ for all but one $I \leq_\rho \ga$ where $k_I=1$ since we have 
	\[ \Map^{\ga \ltimes G}( \coprod_{I \leq_\rho \ga}  (\DD_I\times G)^{\amalg k_I}, P_G(M)) = \prod_{I \leq_\rho \ga} \prod_{k_I} \Map^{\ga \ltimes G}( \DD_I\times G, P_G(M)).\]
	We will prove the stronger statement that $\ev_0$ is a homotopy equivalence. 
	We define a section
	\[ s: P_G(M^I) \rightarrow \Map^{\ga \ltimes G}(\DD_I\times G, P^I_G(M^I))\]
	 as follows. For $p\in P^I_G(M^I)$ we have the map 
	 \[ \tilde{s}_p: \ga \times \R^n \times G \rightarrow P^I_G(M^I), \quad \quad (g,x,A) \mapsto c_g(p) \cdot A \]
	 since $p \in P^I_G(M^I)$ this descends to $\ga \times_I \R^n \times G$ i.e. we have $s_p: \DD_I \times G \rightarrow P_G(M)$.
	 Clearly the map $s_p$ is $\ga \ltimes G$-equivariant and $s$ defines a section. 
	 Next we show that $s \circ \ev_0 \simeq \id$. Consider the homotopy 
	 \[ h_t: \DD_I \times G \rightarrow \DD_I \times G,\quad  (gI,x,A) \mapsto (gI,tx,A),\] 
	 since it is linear in the $x$-coordinate it is $\ga$-equivariant.  
	 Clearly it is also $G$-equivariant. 
	 Precomposition with $h_t$ defines a homotopy $h_t^*$ on $\Map^{\ga \ltimes G}(\DD_I\times G, P^I_G(M^I))$.
	 This concludes the proof since $h_1^*(f) = f \circ h_1 = f$ and $h^*_0(f) = f \circ h_0 = s \circ \ev_0 (f)$. 
 \end{proof}

We define the \emph{derivative at the centers} map $D_0^G$ to be the composite
\begin{align} 
	\EEmb^G_{\fr} \big( \coprod_{I \leq_\rho \ga} [\DD_I/\ga]^{\amalg k_I},[M/\ga] \big) &\xrightarrow{D^G} \Map^{\ga \ltimes G} \big( \coprod_{I \leq_\rho \ga}(\DD_I\times G)^{\amalg k_I},P_G(M) \big) \\
	&\xrightarrow{\ev_0} \prod_{I \leq_\rho \ga} P^I_{G}(F_{k_I}^I[M/\ga]) \label{dac}
\end{align}
where $D^G$ is the canonical projection of the relative homotopy pullback onto a component of the defining diagram \eqref{pbdiagram}.

\begin{prop}
	The composite map $D_0 := \ev_0 \circ D$, 
	\[D_0: \Emb^\ga \big( \coprod_{I \leq_\rho \ga} \DD_I^{\amalg k_I},M ) \rightarrow \prod_{I \leq_\rho \ga} \Fr^I(F^I_{k_I}[M/\ga]),  \]
	is a fibration and a weak equivalence for any choices $k_I \geq 0$.
\end{prop}

We adapt the proof of \cite[Proposition 6.4]{horel17}, which treats the case $\ga = \e$.
\begin{proof}
	Note that the center of $\DD_I$ is mapped into $M^I$. Hence we obtain a weak equivalence 
	\[ \Emb^\ga \big( \coprod_{I \leq_\rho \ga} \DD_I^{\amalg k_I},M \big) \simeq \prod_{I \leq_\rho \ga}\Emb^\ga ( \amalg_{k_I} \DD_I,M ),  \]
	essentially by shrinking disks. It thus suffices to show 
	\[D_0: \Emb^\ga ( \amalg_k \DD_I ,M ) \rightarrow \Fr^I(F^I_{k_I}[M/\ga])  \]
	is a weak equivalence for any $I \leq_\rho \ga$, $k\geq 0$.
	Using the Cerf-Palais fibration criterion (\cite[Lemme 2,p.240]{cerf61}, \cite[Theorem A]{palais60}) one shows $D_0$ is a fibration in the standard way (see \cite[Proposition 6.4]{horel17}).
	Therefore, the proof is finished once we show the fibers of $D_0$ are contractible. 
	
	For $k=1$ and $[M/\ga] = [\DD_I/\ga]$ we may choose the fiber to lie over $(0,A)$ for some $A  \in GL(n)$ (note that in fact we must have $A \in C_{GL(n)}(I))$. 
	The homotopy $h_t(f) := \frac{f(xt)}{t}$ with $h_0(f) = D_0(f)$ contracts the fiber.\footnote{To be precise $f$ consists of a collection of $f_{gI}$ for the connected components of $\DD_I$ and we apply the homotopy $h_t$ component-wise.} 
	For $k=1$ and general $[M/\ga]$, $p\in \Fr^I[M/\ga]$, let $m=\pi(p) \in M^I$ and let $U$ be a $\ga$-neighbourhood of $m$. 
	For $r \in \N$ we denote with $D(r)$ the open ball of points in $\R^n$ of radius less than $1/r$, and set $\DD_I(r) := \ga \times_I D(r)$ (note that $\ga$ preserves $D(r)$ since it is finite). 
	Since $\DD_I$ and $\DD_I(1)$ are $\ga$-equivariantly homeomorphic it suffices to show that
	\[ \Emb^\ga (\DD_I(1), M) \rightarrow \Fr^I[M/I]\]
	is a weak equivalence. We set $M_r := \{ f \in \Emb^\ga ( \DD_I(1) ,M ): f(\DD_I(r)) \subset U \}$. 
	Then $\Emb^\ga (\DD_I(1), M) = \cup_r M_r$ is an increasing filtration, for which the inclusions $M_r \subset M_{r+1}$ are deformation retracts.
	By Lemma \cite[6.1]{horel17} it then suffices to show 
	\[M_1 \rightarrow \Fr^I[M/I]\] 
	is a weak equivalence. Indeed, since by definition maps $f \in M_1$ embed $\DD_I(1)$ into the $\ga$-neighbourhood $U$ this reduces to the case $k=1$ and $[M/\ga]=[\DD_I/\ga]$ treated above. 
	The case of general $k\geq 0$ then easily follows from $k=1$ and is left to the reader.
\end{proof}

We will now prove Theorem \ref{htype} using some standard facts about homotopy pullbacks such as the pasting lemma and homotopy invariance. 
We recommend Martin Frankland's lecture notes on homotopy pullbacks to the reader unfamiliar with (the proofs of) these facts \cite{frankland}.  

\begin{proof}[Proof of Theorem \ref{htype}]
We first prove $D_0^G$ is a fibration.
We can factorise $D_0^G$ as a composite
\begin{align*} \EEmb^G_{\fr} (\amalg_I [\DD_I/\ga]^{\amalg k_I},[M/\ga] ) &\rightarrow \Emb^\ga (\amalg_I \DD_I^{\amalg k_I}, M) \underset{\Map^\ga(\amalg_I \DD_I^{\amalg k_I}, M)}{\times} \Map^G (P_G(\amalg_I \DD_I^{\amalg k_I}), P_G(M)) \\
&\rightarrow \prod_I P_{G}^I( F_{k_I}^I[M/\ga])
\end{align*}
where the first map is a fibration by \ref{hpbmodel} part 2. 
That the second map is a fibration follows from the fact that we have fibrations
\begin{align*}
&\Emb^\ga ( \amalg_I \DD_I^{\amalg k_I}, M) \xrightarrow{D_0} \prod_I P^I_{G} F_{k_I}^I[M/\ga],\\
&\Map^{\ga\ltimes G} (P_G (\amalg_I \DD_I^{\amalg k_I}), P_G(M)) \rightarrow \Map^\ga (\amalg_I \DD_I^{\amalg k_I},M),
\end{align*}
and that $P^I_{G} (F_{k_I}^I [M/\ga])$ are principle $C_{G}(I)^{\times k}$-bundles. 

	Now we will show $D_0^G$ is a weak equivalence. 
	Let \textbf{HPB} denote the homotopy pullback of Diagram \eqref{pbdiagram} for $M= \coprod_{I \leq_\rho \ga} \DD_{I}^{\amalg k_I}$. 
	We can factorise the map $D_0^G$ as a composite
	\begin{align*} \EEmb^G_{\fr} (\amalg_k [\DD_I/\ga],[M/\ga] ) &\lhook\joinrel\longrightarrow \mathbf{HPB} \\ 
	&\xrightarrow{D_0 \underset{ev_0}{\overset{h}{\times}} ev_0} \prod_{I \leq_\rho \ga }\Fr^I(F^I_{k_I}[M/\ga])  \underset{\prod_{I \leq_\rho \ga } \Fr^I((M^I)^{\times k_I})}{\overset{h}{\times}}   \prod_{I \leq_\rho \ga } P^I_{G}((M^I)^{\times k_I}) \\ &\xrightarrow{\mathbf{proj}} \prod_{I \leq_\rho \ga } P^I_{G}(F^I_{k_I}[M/\ga] ).
	\end{align*}
	To show that $D_0^G$ is a weak equivalence, we will show that all the component maps of the factorisation are weak equivalences. 
	The inclusion is a weak equivalence by Lemma \ref{hpbserrefibr} and Proposition \ref{hpbinclusion}. 
	Then Proposition \ref{dac} and Lemma \ref{dac} imply that the components of the map $D_0 \underset{ev_0}{\overset{h}{\times}} ev_0$ are weak equivalences.
	It follows by homotopy invariance that the map $D_0 \underset{ev_0}{\overset{h}{\times}} ev_0$ is a weak equivalence.
	It remains to show \textbf{proj} is weak equivalence. 
	From the model in Proposition \ref{hpbmodel} it is clear that 
	\begin{gather*} 
		\prod_{I \leq_\rho \ga }\Fr^I(F^I_{k_I}[M/\ga])  \underset{\prod_{I \leq_\rho \ga } \Fr^I((M^I)^{\times k_I})}{\overset{h}{\times}}   \prod_{I \leq_\rho \ga } P^I_{G}((M^I)^{\times k_I}) \\ 
		\simeq \\
		\prod_{I \leq_\rho \ga } \big( \Fr^I(F^I_{k_I}[M/\ga])  \underset{\Fr^I((M^I)^{\times k_I})}{\overset{h}{\times}} P^I_{G}((M^I)^{\times k_I}) \big).
	\end{gather*}
	Hence to show \textbf{proj} is a weak equivalence it suffices to show the maps
	\[ \mathbf{proj}_I : \Fr^I(F^I_{k_I}[M/\ga])  \underset{\Fr^I((M^I)^{\times k_I})}{\overset{h}{\times}} P^I_{G}((M^I)^{\times k_I}) \rightarrow P^I_{G}(F^I_{k_I}[M/\ga] ).\]
	are weak equivalences. 
	Consider the following diagram (we relabel $k_I = k$)
	\begin{center}
		\begin{tikzcd}
			P^I_{G}(F_k^I[M/\ga]) \arrow[r,hookrightarrow] \arrow[d]	& P^I_{G}((M^I)^{\times k}) \arrow[d] \\
				\Fr^I(F^I_k[M/\ga]) \arrow[r,hookrightarrow] \arrow[d] &  \Fr^I((M^I)^{\times k}) \arrow[d] \\
				F^I_k[M/\ga] \arrow[r,hookrightarrow] & M^{\times k} 
		\end{tikzcd}
	\end{center}
	and note the outer square is a pullback diagram whose vertical arrows are fibrations. Then the outer square is a homotopy pullback diagram. 
	Moreover, the lower square is a pullback diagram whose vertical arrows are fibrations and hence is a homotopy pullback.
	The pasting lemma then implies the upper square is a homotopy pullback diagram.
	Consequently, we obtain a canonical weak equivalence 
	\begin{align*}
	\mathbf{can}: P^I_{G}(F_k^I[M/\ga]) &\longrightarrow \Fr^I(F^I_k[M/\ga]) \underset{\Fr^I((M^I)^{\times k})}{\overset{h}{\times}} P^I_{G}((M^I)^{\times k}),\\
	 x &\longmapsto (x, \text{constant}_x,x)
	\end{align*}
	using the model of Proposition \ref{hpbmodel}. 
	We obtain a commutative diagram
	\begin{center}
	\begin{tikzcd}
		& P^I_{G}(F_k^I[M/\ga])\\
		P^I_{G}(F_k^I[M/\ga]) \arrow[r,"\textbf{can}"] \arrow[ru,"\id"] & \Fr^I(F^I_k[M/\ga]) \underset{\Fr^I((M^I)^{\times k})}{\overset{h}{\times}} P^I_{G}((M^I)^{\times k}) \arrow[u, "\textbf{proj}_I"]
	\end{tikzcd}
	\end{center}
	and conclude \textbf{proj} is a weak equivalence by the two-out-of-three property of weak equivalences. 
	This finishes the proof.
\end{proof}

\section{Equivariant factorization homology} \label{sec:facthom}
In this section we define equivariant factorization homology and establish its properties. Subsections \ref{subsec:diskalg} and \ref{subsec:invariants} contain the definitions and some basic properties. 
In the subsections \ref{subsec:excision} and \ref{subsec:eht} we prove the theory satisfies $\tensor$-exicison and is characterised by this property. 
We also discuss several consequences of this characterisation. 

\begin{conv}
For $\infty$-categorical notions we refer to \cite{lhtt,lha}. 
By colimit we will mean colimit in the $\infty$-category, which corresponds to a homotopy colimit in the associated topological (or simplicial) category \cite[Th. 4.2.4.1]{lhtt}. Only in the $\infty$-category $\Spaces$ do we write hocolim instead of colim.
\end{conv}

\subsection{Disk algebras or coefficient systems} \label{subsec:diskalg}
\begin{defn}
	Let $\gar\Disk^{G}_n$ denote the full subcategory of $\gar\Orb^{G}_n$ whose objects are equivalent to a disjoint union of local models.
\end{defn}
The category $\gar\Disk_n^G$ inherits the symmetric monoidal structure $\amalg$ from $\gar\Orb_n^G$. 
\begin{defn} \label{def:diskalg}
Let $\catC$ be a symmetric monoidal $\infty$-category. A \emph{$\gar\Disk^G_n$-algebra in $\catC$} is a symmetric monoidal functor $\gar\Disk_n^G \rightarrow \catC$. 
\end{defn}

\begin{rmk}
	The $\gar\Disk^G_n$-algebras will be the \emph{coefficient systems} for equivariant factorization homology, compare to equivariant homology in Definition \ref{coeffsystem}.
\end{rmk}

We have the following natural $\infty$-category of $\gar\Disk_n^G$-algebras.
\begin{defn} Let $\catC$ be a symmetric monoidal $\infty$-category.
	The $\infty$-category of $\gar\Disk_n^G$-algebras in $\catC$
	, denoted $\mathrm{Alg}_{\gar\Disk_n^G}(\catC)$, 
	is the $\infty$-category of symmetric monoidal functors\footnote{That is, the full subcategory of $\mathrm{Fun}_{N(\mathrm{Fin}_*)}(\gar\Disk_n^G,\catC)$ with objects being the symmetric monoidal functors \cite[Def. 2.1.3.7]{lha}.}
	\[ \mathrm{Fun}^{\tensor}(\gar\Disk_n^G,\catC). \]
\end{defn}

\begin{ex} \label{ex:enalg}
	An $\e\Disk^G_n$-algebra in $\catC$ is exactly a $\Disk^{BG}_n$-algebra in the sense of \cite{af15}. In particular $\e\Disk^{\mathrm{fr}}$-algebras are the same thing as $E_n$-algebras \cite[Rmk. 2.10]{af15}, and $\e\Disk^{\mathrm{or}}$-algebras are the same thing as algebras over the (confusingly named, in this context) framed disk operad $fE_n$ \cite[Ex. 2.11]{af15}.
\end{ex}
\begin{nota}
	We will denote with $\Disk^G_n$ the category $\e^\triv\Disk_n^G$.
\end{nota}

Let $\mathrm{B}\ga$ denote the the one object groupoid with automorphisms $\ga$. 
We view it as an $\infty$-category by taking its nerve. 
\begin{nota}We denote by $\catC^\ga := \Fun(B\ga,\catC)$ the category $\ga$-equivariant objects in $\catC$.
\end{nota}

\begin{prop} \label{prop:equivariantalgebras}
	The $\infty$-category $\Alg_{\ga^\triv\Disk_n^G}(\catC)$ of $\ga^\triv\Disk_n^G$-algebras in $\catC$ is equivalent to the $\infty$-category $\Alg_{\Disk_n^G}(\catC)^\ga$ of $\ga$-equivariant $\Disk_n^G$-algebras in $\catC$.
	\end{prop}

\begin{proof}
	First, we reformulate everything in terms of operads.
	We let $\ga\mathbf{D}_n^G$ denote the $\infty$-operad generated by the object $[\DD_\e/\ga]$ in $\ga^\triv \Disk_n^{\fr}$. 
	Since the objects of $\ga^\triv \Disk_n^{\fr}$ are generated under $\amalg$ by $[\DD_\e/\ga]$ we can identify $\Env(\ga\mathbf{D}_n^G)^\tensor \cong \ga^\triv\Disk_n^G$ and so $\Alg_{\ga\mathbf{D}_n^G}(\catC) \cong \Alg_{\ga^\triv\Disk_n^G}(\catC)$ (see \cite[\S 2.4.4]{lha}).
	Similarly $D_n^G$ is the $\infty$-operad generated by $\R^n$ in $\Disk_n^{\fr}$, and
	$\Alg_{\mathbf{D}_n^G}(\catC) \cong \Alg_{\Disk_n^G}(\catC)$. 
	It thus suffices to show the following equivalence
	\[\Alg_{\ga\mathbf{D}_n^G}(\catC) \cong \Alg_{\mathbf{D}_n^G}(\catC)^\ga. \]
	Using \cite[Constr. 2.4.3.3]{lha} we have the $\infty$-operad $B\ga^\amalg \rightarrow N(\Fin_*)$, where $B\ga^\amalg$ is obtained as the nerve of the category with
	\begin{align*}
	&\text{objects} & &\langle k \rangle = \{*_1,\dots,*_k\} \text{ for } k \geq 0 \\
	&\text{morphisms}	& &(\alpha: \langle k \rangle \rightarrow \langle m \rangle \in \Fin_*, g_i: *_i \rightarrow *_{\alpha(i} )
	\end{align*} 
	We then form the $\infty$-operad $\catO^\tensor := B\ga^\amalg \times_{N(\Fin_*)} \mathbf{D}_n^G \rightarrow N(Fin_*)$, and by \cite[Thm 2.4.3.18]{lha} we have
	\[ \Alg_{\catO} (\catC) \cong \Fun(B\ga, \Alg_{\mathbf{D}_n^G}(\catC) = \Alg_{\mathbf{D}_n^G}(\catC)^\ga.\] 
	It thus remains to show that $\ga\mathbf{D}_n^G \cong \catO^\tensor$ as $\infty$-operads.
	This is a straightforward (but tedious) verification; we provide a sketch. 
	To $f: \amalg_k \R^n \rightarrow \R^n \in \mathbf{D}_n^G(k)$ we associate
	\[\ga \times f: \amalg_{\ga} \amalg_k \R^n = \amalg_k [\DD_\e/\ga] \xlongrightarrow{\amalg_\ga f} [\DD_\e/\ga] \text{ in } \ga\mathbf{D}_n^G\]
	and to $g\in \ga$ we associate $g: [\DD_\e/\ga] \rightarrow [\DD_\e/\ga]$ which swaps the indices of $\DD_\e = \amalg_\ga \R^n$. 
	A morphism $(f, \alpha,g_i)$ in $\catO^\tensor$ then corresponds to the framed embedding
	\[ \amalg_k [\DD_e/\ga] \xrightarrow{ \amalg_k g_i} \amalg_k [\DD_e/\ga] \xrightarrow{\ga \times f} [\DD_e/\ga]. \]
	Vice versa, we can decomposes morphisms in $\ga \mathbf{D}_n^G$ canonically into a $\ga \times f$ and $g_i\in G$ components.
	To identify the $k$-simplices of $\catO^\tensor$ and $\ga\mathbf{D}_n^G$ we observe that a $k$-simplex, i.e. a (higher coherent) isotopy in $\ga\mathbf{D}_n^G$, decomposes as a (higher coherent) isotopy between the $\ga\times f$ part and equations between the group elements. 
	This is in turn interpreted as a $k$-simplex in $\mathbf{D}_n^G$ and a $k$-simplex in $B\ga^\amalg$ which lie over the same $k$-simplex in $N(\Fin_*)$ so that they define a $k$-simplex in $\catO^\tensor$.
\end{proof} 
\begin{ex}
Combining Proposition \ref{prop:equivariantalgebras} and Example \ref{ex:enalg} yields many concrete examples of $\ga^\triv\Disk_n^G$-algebras. 
For example, a $\ga^\triv \Disk^{\mathrm{fr}}_1$-algebra in chain complexes is a $\ga$-equivariant $A_\infty$-algebra \cite[Ex. 2.11]{aft17}, and a $\ga^\triv \Disk_2^{\mathrm{or}}$-algebra in $\Rex$ is a $\ga$-equivariant balanced braided tensor category \cite[Rmk 3.8]{bzbj}.
\end{ex}
We can also associate $\ga^\triv\Disk_n^G$-algebras to $\Disk_n^G$-algebras and vice versa. Namely, one has the canonical functors
\begin{align} 
	/\ga^*: \mathrm{Alg}_{\Disk_n^G}(\catC) \rightarrow \mathrm{Alg}_{\ga^\triv\Disk_n^G}(\catC) \quad \text{and} \quad \ga\times^*: \mathrm{Alg}_{\ga^\triv\Disk_n^G}(\catC) \rightarrow  \mathrm{Alg}_{\Disk_n^G}(\catC) \label{eq:forgetequivariance}
\end{align}
obtained respectively by precomposition with the functors $/\ga$ and $\ga \times$ of Proposition \ref{prop:timesandquotient}. 
\begin{defn}
	We call a $\ga^\triv\Disk_n^G$-algebra \emph{trivially $\ga$-equivariant} if it is in the essential image of the functor $/\ga^*$. 
\end{defn}

We give some further examples of coefficients. 

\begin{ex} \label{ex:diskrotalg} Let $p$ be a prime number, $\mathrm{rot}: \Z_p \rightarrow SO(n)$ be the rotation representation, $\vect$ be the plain category of $\kk$-vector spaces with its standard tensor product.
	\begin{enumerate}
		\item A $\Z_2^\mathrm{rot}\Disk_1^{\mathrm{fr}}$-algebra in $\vect$ consists of a pair of vector spaces $(A,M)$ so that $A$ is a unital algebra with anti-involution $\phi: A \rightarrow A^{op}$, and $M$ is a pointed right $A$-module.
		\item When $n\geq 2$ a $\Z_p^\mathrm{rot}\Disk_n^{\mathrm{fr}}$-algebra in $\vect$ consists of a pair of vector spaces $(A,M)$ so that $A$ is an commutative algebra with an order $p$ automorphism $\phi: A \rightarrow A$, and $M$ is a pointed right $A$-module so that $m \cdot \phi(a) = m\cdot a$ for all $m\in M$ and $a \in A$. 
	\end{enumerate}
\end{ex}

\begin{ex}
Let $D_2 = \Z_2 \times \Z_2$ act via its standard representation $\mathrm{st}$ on $\R^2$. 
A $D_2^{\mathrm{st}}\Disk^{\fr}_2$-algebra in $\vect$ consists of a commutative algebra $C$ with commuting involutions $h$ and $v$, a unital algebra $A$ with anti-involution $\sigma$, together with algebra maps $f_v,f_h:C \rightarrow Z(A)$ such that $\sigma (f_h(c)) = f_h(h(c))$ and $\sigma(f_v(c)) = f_v(v(c))$, and a pointed $A_{f_h} \tensor_C {}_{f_v}A$-module $M$. 
\end{ex}

\begin{ex} \label{ex:brpair} In \cite{TOQSP} we introduced coloured topological operads $\Z_2\catD_n$ for $n\geq 1$ with colours $\DD$ and $\DD_*$ corresponding to the global quotients $[\DD_e/\Z_2]$ and $[\DD_{\Z_2}/\Z_2]$ for the sign representation $\sigma$. 
The operation spaces in $\Z_2\catD_n$ are rectilinear equivariant embeddings, and are weakly equivalent to the spaces $F_k^\e[\DD_\e/\Z_2]$ and $F_k^\e[\DD_{\Z_2}/\Z_2]$ via the evaluation at the centers map $\ev_0$ \cite[Prop. 2.9]{TOQSP}. The monoidal envelope naturally includes 
\[ i:\Env(\Z_2\catD_n)^\tensor \rightarrow \Z_2^\sign\Disk_n^{\fr} \]
by interpreting a rectilinear equivariant embedding as a framed embedding. 
Since $D_0^\e(f) = \ev_0(f)$ for a framed embedding $f$ the inclusion of hom spaces defined by $i$ lies over the identity map on the configuration spaces.
Therefore, Theorem \ref{htype} implies $i$ defines weak equivalences on mapping spaces. It follows that $i$ is a monoidal equivalence (see \cite[Rmk 2.1.3.8]{lha}), and hence  the notion of $\Z_2\catD_n$-algebra in \cite{TOQSP} coincides with that of $\Z_2^\sign\Disk_n^{\fr}$-algebra. See \cite{TOQSP} for a classification of $\Z_2\catD_n$-algebras in $\Rex$, see also Example \ref{ex:qsp}.
\end{ex}

\begin{rmk}
	For an $n$-dimensional representation $\rho: \ga \rightarrow GL(V)$ one has the little $V$-disk operad. 
	Equivariant iterated loopspaces provide examples of algebras over these operads.
	See \cite{blumberghill13,guilloumay17}, and references therein, for more details. 
	We expect that the algebras over the little $V$-disks operads are closely related to our $\gar\Disk^{\mathrm{fr}}_n$-algebras. 
\end{rmk}

\subsection{Invariants of global quotient orbifolds}  \label{subsec:invariants}
Following \cite{af15,aft17} we define our invariants via a left Kan extension. 
Let us briefly recall left Kan extensions along inclusions of $\infty$-categories.
Let $i: \catD \rightarrow \catO$ be an inclusion functor of $\infty$-categories. 
The restriction along $i$ induces a functor
\[ i^*: \Fun(\catO, \catC) \rightarrow \Fun(\catD, \catC).\]
The \emph{global left Kan extension} is a left adjoint to $i^*$.
The left Kan extension can be computed via a colimit formula. Denote the canonical projection functor $P: I\downarrow O \rightarrow \catD$, and suppose that for every $O\in \catO$ the colimit 
\begin{align} \colim (I\downarrow O \xrightarrow{P} \catD \xrightarrow{F} \catC ) \label{LKE} \end{align}
exists. 
Then this assignment on objects extends to a functor $\mathrm{Lan}_i(F): \catO \rightarrow \catC$, \emph{the Left Kan extension of $F$ along $i$}, such that composition along $\mathrm{Lan}_i(F)$ defines a left adjoint $\mathrm{Lan}_i(F)_* \dashv i^*$ \cite[Lemma 4.3.2.13, Proposition 4.3.2.15]{lhtt}. 
\begin{rmk}
	The corresponding homotopy colimit computes what is known as the homotopy left Kan extension between topological categories \cite[Proposition A.2.8.9]{lhtt}. 
\end{rmk}

We will make use of the symmetric monoidal left Kan extension. Namely, let $i^{\tensor}: \catD^{\tensor} \rightarrow \catO^{\tensor}$ be a symmetric monoidal functor. 
The restriction along $i^{\tensor}$ induces a functor
\[ (i^\tensor)^*: \Fun^{\tensor}(\catO^{\tensor}, \catC^{\tensor}) \rightarrow \Fun^{\tensor}(\catD^{\tensor}, \catC^{\tensor}).\]
The \emph{symmetric monoidal left Kan extension} is a symmetric monoidal functor $\mathrm{Lan}^{\tensor}_{i}(F): \catO^{\tensor} \rightarrow \catC^{\tensor}$ together with the data of an adjunction $\mathrm{Lan}^{\tensor}_{i}(F)_* \dashv (I^\tensor)^*$.

\begin{defn}
Let $A$ be a $\gar\Disk_n^G$-algebra in $\catC$,  and let $i^{\tensor}: \gar\Disk_n^G \rightarrow \gar\Orb_n^G$ denote the (symmetric monoidal) inclusion. 
\emph{Equivariant factorization homology with coefficients in $A$}, provided it exists, is defined as the symmetric monoidal left Kan extension of $A$ along $i^{\tensor}$, and denoted as 
\begin{align}
	\int_- A: \gar\Orb_n^G &\rightarrow \catC^{\tensor},\quad \quad \quad \quad
	[M/\ga] \mapsto \int_{[M/\ga]} A.
\end{align} 
\end{defn}

To obtain a formula for the symmetric monoidal left Kan extension we need to make an assumption on target categories.
\begin{defn} \label{def:tensorcocomplete}
	Let $\catC$ be a symmetric monoidal $\infty$-category. We say that $\catC$ is \emph{$\tensor$-cocomplete} if $\catC$ admits arbitrary colimits and the tensor product $\tensor$ commutes with geometric realisations and filtered colimits in both variables.
\end{defn}

\begin{ex} \label{ex:tensorcocomplete} The following symmetric monoidal categories are $\tensor$-cocomplete.
	\begin{enumerate}
		\item The plain category $\vect^\tensor$ of $\kk$-vector spaces with the standard $\tensor$-product.
		\item The (2,1)-category $\Rex^\boxtimes$ of finitely cocomplete $\kk$-linear categories and right exact functors with $\boxtimes$ the Deligne-Kelly tensor product \cite[Prop. 3.5]{bzbj}.
		\item The $\infty$-category $\Ch_\kk^\oplus$ of chain complexes with $\oplus$ the direct sum \cite[Ex 3.14]{aft17}.
		\item The $\infty$-category $\Ch_\kk^\tensor$ with the standard tensor product \cite[Ex. 3.5]{af15}.
	\end{enumerate}
\end{ex}

We then have the following central result: 
\begin{prop} \label{aftlem}
	Assume $\catC$ is $\tensor$-cocomplete and let $A^{\tensor}$ be a $\gar\Disk_n^G$-algebra. 
	Equivariant factorization homology with coefficients in $A^{\tensor}$ exists and is computed on objects by the left Kan extension 
	\begin{align} 
	\int_{[M/\ga]} A^{\tensor} \simeq \colim (i \downarrow [M/\ga] \xrightarrow{P} \gar\Disk_n^G \xrightarrow{A} \catC ) \label{fhcolim} 
	\end{align}
	of the underlying functor $A$ of $A^{\tensor}$.
\end{prop} 
\begin{proof}
	This is a general statement about symmetric monoidal homotopy left Kan extensions, see \cite[Lemma 2.17]{aft17}.
	By Lemma 2.17 in \cite{aft17} we only need to verify that the functors
	\begin{align*}
		&(\gar\Disk_n^G)_{/\emptyset} \rightarrow (\gar\Orb_n^G)_{/\emptyset},\\
		&(\gar\Disk_n^G)_{/[M/\ga]}\times (\gar\Disk_n^G)_{/[N/\ga]} \rightarrow (\gar\Disk_n^G)_{/[M\amalg N/\ga]},
	\end{align*}
	are final for all $[M/\ga]$ and $[N/\ga]$. 
	By inspection both functors are equivalences so that they are in particular final. 
\end{proof}

\begin{rmk}
	The restriction of $\tensor$-cocompleteness is probably unnecessarily strong.
 For $\ga= \e$ it suffices to assume $\tensor$-sifted cocompleteness \cite[Thm. 2.15]{aft17}. 
\end{rmk}

\begin{cor} \label{cor:localcomp}
	Let $A$ be a $\gar\Disk_n^G$-algebra in $\catC$ and assume $\catC$ is $\tensor$-cocomplete. 
	Then for a local model $[\DD_I/\ga] \in \gar\Orb_n^G$ we have
	\[ \int_{[\DD_I/\ga]}A \simeq A([\DD_I/\ga]). \]
\end{cor}
\begin{proof}
	The left Kan extension of any functor $A: \catD \rightarrow \catC$ along a full and faithful functor $I:\catD \rightarrow \mathcal{O}$ restricts to $\catD$ as $A$, see for example the proof of  \cite[Proposition 4.3.2.17]{lhtt}.
\end{proof}

Aside from the colimit formula \eqref{fhcolim} one also has a (derived) coend formula that computes the (homotopy) left Kan extension, see e.g. \cite[Ex. 9.2.11]{riehl14}. We have
\begin{align} 
\int_{[M/\ga]} A \quad \simeq \quad \EEmb^G_{\fr}(D,[M/\ga]) \bigotimes_{D \in \gar\Disk_n^G} A(D), \label{coendformula}
\end{align}
where $\EEmb^G_{\fr}(D,[M/\ga])$ acts on $A(D) \in \catC$ via tensoring over spaces in $\catC$.\footnote{Recall $S \tensor c\in \catC$ for $S \in \mathrm{Spaces}$ and $c \in \catC$ is the homotopy colimit of the constant functor $S \xrightarrow{c} \catC$ where we view $S$ as a Kan complex.}
Also, compare to \cite[Def. 3.2]{af15}.

\subsection{The proof of excision} \label{subsec:excision}

We will now formulate and prove the $\tensor$-excision property.
We start by specifying how one is allowed to decompose global quotients. 

\begin{conv}
	From now on we will assume that any symmetric monoidal $\infty$-category $\catC$, appearing as a target for factorization homology, is $\tensor$-cocomplete. 
\end{conv}

\begin{defn}
	Let $[M/\ga]$ be a $\ga$-global quotient, and endow $[-1,1]$ with the trivial $\ga$-action. A \emph{collar-gluing} is a choice of $\ga$-equivariant surjective map
	\[ f: M \rightarrow [-1,1] \]
	that is a manifold bundle over $(-1,1)$. 
\end{defn}

Let $M_- := f^{-1}[-1,1)$, $M_+ := f^{-1}(-1,1]$ and $M_0 := f^{-1}(-1,1)$ denote the open submanifolds of $M$ defined by the collar-gluing. 
We have the decomposition
\[ M  \cong M_- \cup_{M_0} M_+\]
where the global quotient $[M_0/\ga]$ is $\ga$-equivariantly diffeomorphic $M_0 \cong N \times \R$ and $\ga$ acts trivially on $\R$. 

\begin{lem} \label{lem:e1str}
	Let $[M/\ga]$ be a $(\rho,G)$-framed global quotient with a collar-gluing and let \\$H:\gar\Orb_n^G \rightarrow \catC$ be a symmetric monoidal functor. 
	Then the object
	\[ H[M_0/\ga] \]
	naturally has the structure of an $E_1$-algebra in $\catC$.
\end{lem}

\begin{proof}
	Recall that specifying an $E_1$-algebra structure is equivalent to specifying a $\Disk_1^{\mathrm{fr}}$-algebra structure, by Example \ref{ex:enalg}.
	Let us denote the free local model by $\R\in \Disk_1^{\mathrm{fr}}$.
	Any framed embedding $(g,h,r): \amalg_k \R \rightarrow \amalg_m \R$ induces a \ $G$-framed embedding
	\[ (\id_N \times g, \id_{\Fr(N)} \times h, \id_{\Fr(N)} \times r) : \amalg_k [N \times \R/\ga] \rightarrow \amalg_m [N\times \R/\ga],\] 
	where we used $\Fr(N\times \R) \cong \Fr(N) \times \Fr(\R)$ and $P_G(N \times \R) \cong P_G(N) \times P_G(\R)$.
	Functoriality of the symmetric monoidal functor $H$ yields a map
	\[  H[N\times \R/\ga]^{\tensor k} \rightarrow H[N\times \R/\ga] \]
	which defines the $E_1$-structure on $H[N\times \R/\ga] \cong H[M_0/\ga]$. 
\end{proof} 

We fix oriented embeddings between the intervals
\begin{align*}
&\mu_-: [-1,1) \amalg (-1,1) \rightarrow [-1,1),\\
&\mu: (-1,1) \amalg (-1,1) \rightarrow (-1,1),\\
&\mu_+: (-1,1) \amalg (-1,1] \rightarrow (-1,1]
\end{align*}
such that $\mu_-(-1) = -1$ and $\mu_+(1) = 1$.
For a $(\rho,G)$-framed global quotient $[M/\ga]$ with a collar-gluing we abuse notation by also denoting 
\begin{align}
&\mu_-: [M_-/\ga] \amalg [M_0/\ga] \rightarrow [M_-/\ga],\\
&\mu: [M_0/\ga] \amalg [M_0/\ga] \rightarrow [M_0/\ga],\\
&\mu_+: [M_-/\ga] \amalg [M_0/\ga] \rightarrow [M_-/\ga]
\end{align}
the induced framed embeddings. 
\begin{defn} \label{barconstr}
	Let $H: \gar\Orb_n^G \rightarrow \catC^{\tensor}$ be a symmetric monoidal functor, and let $[M/\ga]$ be a $(\rho,G)$-framed global quotient with a collar-gluing.
	We define the \emph{two-sided bar construction} 
	\[H[M_+/\ga] \; \underset{H[M_0/\ga]}{\bigotimes} \; H[M_-/\ga] \] 
	to be the geometric realization of the following simplicial object\footnote{Note that this is a simplicial object in the sense of an $\infty$-functor $\Delta^{op} \rightarrow \catC$, meaning a homotopically coherent simplicial object in the associated topological category.} in $\catC$
	\begin{align}
	H_+ \tensor H_- \quad \substack{\leftarrow\\[-1em] \leftarrow} \quad  H_+ \tensor H_0 \tensor H_- \quad \substack{\leftarrow\\[-1em] \leftarrow \\[-1em] \leftarrow} \quad H_+ \tensor H_0^{\tensor 2} \tensor H_- \quad \substack{\leftarrow\\[-1em] \leftarrow \\[-1em] \leftarrow \\[-1em] \leftarrow} \label{simplbar}
	\end{align}
	where $H_i := H[M_i/\ga]$ for $i\in \{+,-,0\}$ and the simplicial maps are the standard bar construction maps using $H(\mu_-)$, $H(\mu_0)$ and $H(\mu_+)$. 
\end{defn}

\begin{rmk}
	Though we will not give the definitions of modules over $E_1$-algebras, we are in fact computing the two-sided bar construction of the modules $H_+$ and $H_-$ over the $E_1$-algebra $H_0$. As suggested by the notation this simplicial object computes the relative tensor products of the modules over the algebra, see \cite[\S 4.4.2]{lha} for details. 
\end{rmk}
\begin{lem} \label{canarrow}
	For $H$ and $[M/\ga]$ as in Definition \ref{barconstr} there is a canonical arrow
	\[ H[M_+/\ga] \; \underset{H[M_0/\ga]}{\bigotimes} \; H[M_-/\ga] \quad \longrightarrow \quad H[M/\ga]. \]
\end{lem}
\begin{proof}
	The simplicial object \ref{simplbar} has a canonical augmentation 
	\[ H[M_-/\ga] \tensor H[M_+/\ga] \cong H[M_- \amalg M_+/\ga] \rightarrow H[M/\ga])\]
	given by  $H(\mu_- \amalg \mu_+)$.\footnote{To be precise the image under $H$ of the restriction of $\mu_+ \amalg \mu_-$ to $[M_-/\ga] \amalg [M_+/\ga]$.} This induces a map from the geometric realization. 
\end{proof}

\begin{defn} \label{def:excision}
Let $H: \gar\Orb_n^G \rightarrow \catC^{\tensor}$ be a symmetric monoidal functor. 
We say $H$ satisfies \emph{$\tensor$-excision}, if for every collar-gluing of a $(\rho,G)$-framed global quotient $[M/\ga]$ the canonical arrow of Lemma \ref{canarrow} is an equivalence.
\end{defn}

\begin{thm} \label{thm:excision}
	Let $A$ be a $\gar\Disk_n^G$-algebra in a $\tensor$-cocomplete category $\catC$. Then equivariant factorization homology satisfies $\tensor$-excision, i.e. for $[M/\ga]$  a $(\rho,G)$-framed global quotient with a collar-gluing the canonical arrow 
	\[ \int_{[M_+/\ga]} A \; \; \underset{\int_{[M_0/\ga]}A}{\bigotimes} \; \; \int_{[M_-/\ga]} A \quad \longrightarrow \quad \int_{[M/\ga]} A \]
	is an equivalence. 
\end{thm}

To prove the theorem we will rewrite the colimit based on a collar-gluing. In order to do that we need to have a good understanding of how the mapping spaces decompose based on a collar-gluing. We will first answer this question, making use of the model for the mapping spaces provided by the configuration spaces (Theorem \ref{htype}). Let $[M/\ga]$ be a $(\rho,G)$-framed global quotient with a collar-gluing, and fix integers $k_I \geq 0$ for $I \leq_\rho \ga$. 
The bar maps $\mu_-, \mu$ and $\mu_+$ induce maps between the configuration spaces
\begin{align*}
F_{k_I}^I[M_+ \amalg M_0/\ga] \rightarrow  F_{k_I}^I[M_+/\ga],\\
F_{k_I}^I[M_0 \amalg M_0/\ga] \rightarrow  F_{k_I}^I[M_0/\ga],\\
F_{k_I}^I[M_0 \amalg M_-/\ga] \rightarrow  F_{k_I}^I[M_-/\ga],
\end{align*}
which allows us to define the following simplicial object 
\begin{align}
\label{confsspace}
\prod_{I \leq_\rho \ga} P^I_{G} F_{k_I}^I[M_+ \amalg M_-/\ga] \; \substack{\leftarrow\\[-1em] \leftarrow} \; \prod_{I \leq_\rho \ga} P^I_G F_{k_I}^I[M_+ \amalg M_0 \amalg M_-/\ga] \; \substack{\leftarrow\\[-1em] \leftarrow \\[-1em] \leftarrow}
\end{align}
in $\Spaces$.
\begin{prop} \label{prop:confsspacedecomp}
	The geometric realization of
	the simplicial object in $\Spaces$ of Equation \eqref{confsspace} is modelled by the space
	\[ \prod_{I \leq_\rho \ga} P_G^I F_{k_I}^I[M/\ga] .\]
\end{prop}
Let us call the simplicial space of Equation \eqref{confsspace} $C_\bullet$. We will compute $\hocolim C_\bullet$ by strictifying the simplicial diagram $C_\bullet$.
We are thankful for Hiro Lee Tanaka's suggestion of strictifying the diagram by using a variation of the Moore path space.
\begin{proof}
	The proof for general $G$ is verbatim the same as the proof for $G=\e$. So to simplify notation we will treat the case $\prod_{I \leq_\rho \ga} F_{k_I}^I[M/\ga]$. 
	Recall finite products commute with homotopy colimits of simplicial spaces. Therefore, it suffices to treat the case $F_k^I[M/\ga]$ for arbitrary $I\leq_\rho \ga$ and $k > 0$. Thus we fix from now on $I$ and some $k \geq 0$, and we introduce some notation. We let $\pi:M \rightarrow [-1,1]$ be the collar-gluing and we set
	\begin{align*}
	M_+^i := \pi^{-1}\left[-1,\frac{i}{k}\right),\quad \quad M_0^{i,j} := \pi^{-1}\left(\frac{i}{k},\frac{i+j}{k}\right), \quad \quad M_-^i := \pi^{-1}\left(1-\frac{i}{k},1\right].
	\end{align*}
	We can then define the following (strict) simplicial space $\tilde{S}_\bullet$:
	\begin{align*} 
	&\coprod_{k_1+k_2=k} F_{k_1}^I[M_{+}^{k_1}/\ga] \times F_{k_2}^I[M_{-}^{k_2}] \; \substack{\leftarrow\\[-1em] \leftarrow} \; \coprod_{k_1+k_2+k_3=k} F_{k_1}^I[M_{+}^{k_1}/\ga] \times F_{k_2}[M_0^{k_1,k_2}/\ga] \times F_{k_3}^I[M_{-}^{k_3}]\\
	&\substack{\leftarrow\\[-1em] \leftarrow \\[-1em] \leftarrow} \; \coprod_{\sum_{i=1}^4 k_i=k} F_{k_1}^I[M_{+}^{k_1}/\ga] \times F_{k_2}[M_0^{k_1,k_2}/\ga] \times F_{k_3}[M_0^{k_1+k_2,k_3}/\ga] \times F_{k_4}^I[M_{-}^{k_4}].
	\end{align*}
	Here the face maps are a variation on bar maps where one combines pieces as follows
	\begin{align*}
	F_i[M_+^i/\ga]\times F_j[M_0^{i,j}/\ga] \rightarrow F_{i+j}[M^{i+j}_+/\ga],\\
	F_j[M_0^{i,j}/\ga]\times F_l[M_0^{i+j,l}/\ga] \rightarrow F_{j+l}[M^{i,j+l}_0/\ga],\\
	F_j[M_0^{i,j}/\ga]\times F_l[M_-^{l}/\ga] \rightarrow F_{j+l}[M^{j+l}_-/\ga],
	\end{align*}
	by mapping to a configuration in the cocatenation of $M_+^i$ and $M_0(i,j)$, and so forth. 
	The degeneracies are standard (identifying a piece with the matching part of the coproduct).
	We need to adapt the simplicial space $\tilde{S}_\bullet$ slightly to account for the following. The configuration spaces $F_k^I$ are ordered configurations so that 
	\[F^I_k[M_1 \amalg M_2/\ga] =  \coprod_{k_1+k_2 = k} ( \coprod_{\binom{k}{k_1,k_2}} F_{k_1}^I[M_1/\ga]\times F_{k_2}^I[M_2/\ga] ).\] 
	We define the simplicial space $S_\bullet$ exactly as $\tilde{S}_\bullet$, but accounting for these multiplicities, e.g. 
	\[S_0 = \coprod_{k_1+k_2=k} ( \coprod_{\binom{k}{k_1,k_2}} F_{k_1}^I[M_{+}^{k_1}/\ga] \times F_{k_2}^I[M_{-}^{k_2}] ).\]
	Note that we have a (homotopy coherent) map $C_\bullet \rightarrow S_\bullet$ obtained by rescaling components. This is a levelwise weak equivalence and therefore we find $\hocolim C_\bullet \simeq \hocolim S_\bullet$. Observe that the levels of $S_i$ are smooth manifolds (and in particular, CW complexes) and the maps embeddings. By definition $S_\bullet$ is then Reedy cofibrant. 
	Therefore, $\hocolim S_\bullet$ can be computed as the geometric realization $|S_\bullet|$ of $S_\bullet$ \cite[Cor. A.2.9.30]{lhtt}. 
	
	Clearly $F_k^I[M/\ga]$ has a cone structure over $S_\bullet$ so we get an induced map $p: |S_\bullet| \rightarrow F_k^I[M/\ga]$. 
	Now observe that the non-degenerate connected components of $S_0$ are a cover of $F^I_k[M/\ga]$.
	Moreover, the non-degenerate connected components of $S_k$ are exactly the intersections of $k$ distinct pieces of the cover $S_0$. 
	Therefore, $S_\bullet$ is the simplest type of example of a hypercover in the sense of \cite{duggerisakson}. 
	In particular, since $S_\bullet$ is a hypercover $F^I_k[M/\ga]$ we find that the induced map $p$ is a weak equivalence
	\[ |S_\bullet | \simeq F^I_k[M/\ga],\]
	by \cite[Th. 1.3]{duggerisakson}.
	This completes the proof.
\end{proof}

\begin{cor} \label{cor:mapspacedecomp}
	Let $[M/\ga]$ be a $(\rho,G)$-framed global quotient with a collar-gluing. The  geometric realization of
	the simplicial object in $\Spaces$
	\begin{align*}
	\EEmb^G_{\fr} \big( \coprod_{I \leq \ga} [\DD_I/I]^{\amalg k_I},[M_+ \amalg M_- /\ga] \big) \; \substack{\leftarrow\\[-1em] \leftarrow} \; \EEmb^G_{\fr} \big( \coprod_{I \leq \ga} [\DD_I/I]^{\amalg k_I},[M_+ \amalg M_0 \amalg M_- /\ga] \big) \; \substack{\leftarrow\\[-1em] \leftarrow \\[-1em] \leftarrow}
	\end{align*}
	is modelled by the space
	\[\EEmb^G_{\fr} \big( \coprod_{I \leq \ga} [\DD_I/I]^{\amalg k_I},[M/\ga] \big).\]
\end{cor}
\begin{proof}
	Let us call the simplicial space of Equation \eqref{confsspace} $C_\bullet$ and the simplicial space above $E_\bullet$. Furthermore, let us denote $\prod_{I \leq_\rho \ga} P_G^I F_{k_I}^I[M/\ga]$ by $C$ and $\EEmb^G_{\fr} \big( \coprod_{I \leq \ga} [\DD_I/I]^{\amalg k_I},[M/\ga] \big)$ by $E$. 
	Recall that we have the weak equivalences 
	\[ D_0^G: \EEmb^G_{\fr} \left(\coprod_{I \leq_\rho \ga} [\DD_I/\ga]^{\amalg k_I},[M/\ga] \right) \longrightarrow \prod_{I \leq_\rho \ga} P^I_{G^{\times k_I}}(F^I_{k_I}[M/\ga] )\]
	of Theorem \ref{htype}.
	Observe that applying $D_0^G$ levelwise defines a map $D_0^G: E_\bullet \rightarrow C_\bullet$. 
	This map is a levelwise weak equivalence and hence induces a weak equivalence between the homotopy colimits. 
	We obtain the following commutative diagram
	\begin{center}
		\begin{tikzcd}
			\hocolim E_\bullet \arrow[r] \arrow[d,"D_0^G"] & E \arrow[d,"D_0^G"] \\
			\hocolim C_\bullet \arrow[r] & C
		\end{tikzcd}
	\end{center}
	where have shown that the vertical maps are a weak equivalences.
	The bottom horizontal map is a weak equivalence by Proposition \ref{prop:confsspacedecomp}. 
	By two-out-of-three it follows that the top horizontal map is also a weak equivalence.
\end{proof}

We can now prove Theorem \ref{thm:excision}.
Our proof is a generalisation and expansion of Francis's proof of excision for factorization homology of manifolds given in \cite[Proposition 3.24]{francis13}.

\begin{proof}[Proof of Theorem \ref{thm:excision}]
	We will make use of the coend formula \eqref{coendformula} for factorization homology. By definition of the derived coend $\int_{[M/\ga]} A$ is then computed as the geometric realization of the simplicial object
	\begin{align} \label{eq:coenddef}
	\coprod_{d \in \catD} \bE_{M}^d \tensor A(d) \; \substack{\leftarrow\\[-1em] \leftarrow} \; \coprod_{d_1,d_2 \in \catD} \bE_{M}^{d_2} \tensor \catD_{d_2}^{d_1} \tensor A(d_1) \; \substack{\leftarrow\\[-1em] \leftarrow \\[-1em] \leftarrow} 
	\end{align}
	where we used short-hand notations $\catD = \gar\Disk^{G}_n$, and $\catD^d_{d_2} = \catD(d_1,d_2)$ and $\bE_{M}^d = \EEmb^G_{\fr}(d,[M/\ga])$. 
	Using Corollary \ref{cor:mapspacedecomp}, we can vertically resolve every mapping space $\bE_{M}^d$ as the geometric realization of a simplicial object.
	As such, we obtain a bisimplicial object $X_{\bullet,*}$. Using the assumption that tensor products commute with geometric realisations, we find
	\[ | \; \;| X_{\bullet, *}|_* \; \;|_\bullet \simeq \int_{[M/\ga]} A.\]
	Now recall that for any bisimplicial object $X_{\bullet,*}$ we have the following canonical equivalences
	\[ | \; \;| X_{\bullet, *}|_* \; \;|_\bullet \quad \simeq \quad | X_{\bullet,*} |_{\mathrm{diagonal}} \quad \simeq \quad | \; \;| X_{\bullet, *}|_\bullet \; \;|_*.\]
	Thus $\int_{[M/\ga]} A$ is also computed by first computing the vertical realisation of $X_{\bullet,*}$ and then computing the horizontal realisation. 
	At level $i$ we are thus computing the realisation of 
	\[ X_{i,0} \; \substack{\leftarrow\\[-1em] \leftarrow} \; X_{i,1} \; \substack{\leftarrow\\[-1em] \leftarrow \\[-1em] \leftarrow} \; X_{i,2}\]
	which for us is the realisation of the simplicial object
	\begin{align*} 
	&X_{i,0} & &\coprod_{d\in \catD } \left(\coprod_{\substack{d_1,\dots,d_i \in \catD,\\d_1 \amalg \dots \amalg d_i = d}} \bE_{M_+}^{d_i} \tensor \bigotimes_{j=i-1}^{2}\bE_{M_0}^{d_j} \tensor \bE_{M_-}^{d_1} \right) \tensor A(d), \\
	&X_{i,1} & &\coprod_{d^1,d^2\in \catD } \left(\coprod_{\substack{d^l_j \in \catD, \\ \amalg_j d^l_j = d^l}} \bE_{M_+}^{d_i^l} \tensor \bigotimes_{j=i-1}^{2}\bE_{M_0}^{d_j^l} \tensor \bE_{M_-}^{d^1} \right) \tensor \catD(d^1,d^2) \tensor A(d^1). 
	\end{align*}
	Now using $A$ is monoidal and how $\catD(d^2,d_1)$ decomposes when $d^2 = \amalg_{j} d^2_j$ and $d^1=\amalg_j d^1_j$ we can rewrite the terms as
	\begin{align*} 
	&\left( \coprod_{d_i\in \catD } \bE_{M_+}^{d_i} \tensor A(d_1) \right)  \tensor \left( \coprod_{d_j \in \catD}\bE_{M_0}^{d_j} \tensor A(d_j) \right)^{\bigotimes_{j=i-1}^{2}} \tensor \left( \coprod_{d_1\in \catD } \bE_{M_-}^{d_1}  \tensor A(d^1)\right), \\
	&\left( \coprod_{d_i^1,d_i^2\in \catD } \bE_{M_+}^{d_i^2} \tensor \catD^{d_i^1}_{d_i^2} \tensor A(d_1) \right)  \tensor \left( \coprod_{d_j \in \catD}\bE_{M_0}^{d_j} \tensor \catD^{d_j^1}_{d_j^2}\tensor A(d_j^1) \right)^{\bigotimes_{j=i-1}^{2}} \tensor \left( \coprod_{d_1\in \catD } \bE_{M_-}^{d_1^1} \tensor \catD^{d_1^1}_{d_1^2} \tensor A(d^1_1)\right), 
	\end{align*}
	
	Thus we recognise the distinct terms of our coend formula computing factorization homology. Using again our assumption that the tensor product commutes with geometric realizations we find that 
	\[ |X_{i,*}|_* \simeq  \int_{[M_+/\ga]} A \tensor \left( \int_{[M_0/\ga]} A \right)^{\tensor i}\tensor \int_{[M_-/\ga]} A\]
	with the face maps being exactly the bar maps. By definition we thus find
	\[ | \; \;| X_{\bullet, *}|_* \; \;|_\bullet \quad \simeq \quad \int_{[M_+/\ga]} A \; \; \bigotimes_{\int_{[M_0/\ga]}A} \; \; \int_{[M_-/\ga]} A.\]
	This completes the proof.
\end{proof}
\subsection{The characterisation of equivariant factorization homology} \label{subsec:eht}

We now prove $\tensor$-excision characterises factorization homology.

\begin{defn} \label{defn:homthy}
	The $\infty$-category of \emph{$\catC$-valued equivariant factorization homology theories} for $(\rho,G)$-framed global quotients is the full $\infty$-subcategory
	\[ \cH (\gar\Orb_n^G, \catC) \subset \Fun^{\tensor}( \gar\Orb_n^G, \catC)  \]
	of those symmetric monoidal functors that 
	\begin{enumerate}
	\item satisfy $\tensor$-excision,
	\item preserve sequential colimits, i.e. for inclusions $M_1 \subset M_2 \subset \dots$ of manifolds with $\cup_i M_i = M$ for a $(\rho,G)$-framed global quotient $[M/\ga]$ then $\underset{i}{\colim} \int_{[M_i/\ga]} A \cong \int_{[M/\ga]} A$. 
	\end{enumerate}
\end{defn}

\begin{rmk}
	Preserving sequential colimits is an unimportant technical restriction only made so that we can do transfinite induction, see Remark \ref{findecomp}.
\end{rmk}

\begin{ex}
Fix some $I\leq_\rho \ga$. The functor $\gar\Orb_n^G \rightarrow \Spaces^{\amalg}$ sending a $[M/\ga]$ to the space $M^I$ is an equivariant factorization homology theory.
\end{ex}

\begin{lem} \label{lem:efhseqcol}
Let $A$ be a $\gar\Disk_n^G$-algebra in $\tensor$-cocomplete category $\catC$. Equivariant factorization homology with coefficients in $A$ preserves sequential colimits.
\end{lem}
\begin{proof}
Note that since $M_i \cap M_j = M_{\mathrm{max}(i,j)}$ that $M_i$ is a hypercover of $M$ in the sense of \cite{duggerisakson}. Denote $C_i := \prod_{I \leq_\rho \ga} P^I_{G}(F^I_{k_I}[M_i/\ga] )$, for some $k_I\geq 0$. 
Then $C_i$ is a hypercover of $\prod_{I \leq_\rho \ga} P^I_{G}(F^I_{k_I}[M/\ga])$.
By \cite[Th. 1.3]{duggerisakson} we then have that 
\[ \underset{i}{\hocolim} C_i \simeq  \prod_{I \leq_\rho \ga} P^I_{G}(F^I_{k_I}[M/\ga] )\]
for any $k_I \geq 0$.
It follows by Theorem \ref{htype} that
\[ \underset{i}{\hocolim} \EEmb^G_{\fr}(D,[M_i/\ga]) \simeq \EEmb^G_{\fr}(D,[M/\ga])\]
for any $D \in \gar\Disk_n^G$. 
Using filtered colimits commute with $\tensor$ we compute:
\begin{align*}
\underset{i}{\colim} \int_{[M_i/\ga]} A &\simeq \underset{i}{\colim} \left( \EEmb^G_{\fr}(D,[M_i/\ga]) \bigotimes_{D \in \gar\Disk_n^G} A(D) \right) \\
&\simeq \left(\underset{i}{\hocolim} \EEmb^G_{\fr}(D,[M_i/\ga]) \right) \bigotimes_{D \in \gar\Disk_n^G} A(D)\\
&\simeq \EEmb^G_{\fr}(D,[M/\ga]) \bigotimes_{D \in \gar\Disk_n^G} A(D) \simeq \int_{[M/\ga]} A.
\end{align*}
\end{proof}

Recall that by definition factorization homology defines an adjunction 
\[ \int: \mathrm{Alg}_{\gar\Disk_n^G}(\catC) \rightleftharpoons \Fun^{\tensor}(\gar\Orb_n^G, \catC): (i^\tensor)^*  .\]
\begin{nota}
	For a $\catC$-valued homology theory $H$ for $(\rho,G)$-framed global quotients we denote with $A_H:= H \circ i^\tensor$ the $\gar\Disk_n^G$-algebra in $\catC$ obtained by restriction.
\end{nota}
By Theorem \ref{thm:excision} and Lemma \ref{lem:efhseqcol} the adjunction restricts to an adjunction
\begin{align}
\label{efhadj} \int: \mathrm{Alg}_{\gar\Disk_n^G}(\catC) \rightleftharpoons   \cH(\gar\Orb_n^G, \catC): (i^\tensor)^*.
\end{align}
\begin{thm} \label{thm:classification}
	Let $\catC$ be $\tensor$-cocomplete. The adjunction \ref{efhadj} is an equivalence.
	In particular, any $\catC$-valued homology theory $H$ for $(\rho,G)$-framed global quotients is computed via
	\[ H \cong \int A_H\] 
	and thus is an example of factorization homology. 
\end{thm}

Before proving the theorem, we will discuss some corollaries.

\begin{cor}
	Let $A$ be a $\gar\Disk_n^G$-algebra in $\catC$, $G: \catC \rightarrow \catD$ be a symmetric monoidal functor preserving geometric realizations and sequential colimits, and suppose $\catC$ and $\catD$ are $\tensor$-cocomplete.
	Then we have 
	\[ G \circ \int A = \int G\circ A\]
	where $G\circ A$ is the $\gar\Disk_n^G$-algebra in $\catD$ defined by composing functors.
\end{cor}
\begin{proof}
	Observe that $G \circ \int A$ by our assumption of $G$ defines a homology theory that restricts as $G \circ A$ on $\gar\Disk_n^G$ by Corollary \ref{cor:localcomp}.
\end{proof}

Consider a $(\rho,H)$-framed manifold $[M/\ga]$ where $H\subset G$ and $\rho: \ga \rightarrow N_{G}(\ga) \cap N_{H}(\ga)$. Then $[M/\ga]$ is also canonically a $(\rho,G)$-framed manifold, and framed embeddings for the $(\rho,H)$-structure are examples of framed embeddings for the $(\rho,G)$-structure. 
Thus we have subcategories $\gar\Orb_n^H \subset \gar\Orb_n^G$, and $\gar\Disk_n^H \subset \gar\Disk_n^G$.
\begin{nota} 
	We will denote with $\Ind_H^G: \gar\Orb_n^H \rightarrow \gar\Orb_n^G$ the symmetric monoidal inclusion. 
\end{nota}

The notation for factorization homology is somewhat ambiguous as the structure group $G$ is suppressed. 
	Corollary \ref{cor:unamb} removes this ambiguity. 
	For example, for a $\rho$-framed global quotient it is irrelevant whether one computes its factorization homology as a $\rho$-oriented global quotient, or as a $\rho$-framed global quotient.
\begin{cor} \label{cor:unamb}
	Let $[M/\ga] \in \gar\Orb_n^H$, $\rho: \ga \rightarrow N_{G}(\ga) \cap N_{H}(\ga)$ and $A$ be an $\gar\Disk_n^G$-algebra in $\catC$.
	Then 
	\[ \int_{\Ind_H^G [M/\ga]} A  \cong \int_{[M/\ga]} A \circ \Ind_H^G \] 
\end{cor}
\begin{proof}
	Both define $\catC$-valued equivariant factorization homology theories of $(\rho,H)$-framed manifolds that restrict to $\gar\Disk_n^H$ as $A\circ \Ind_H^G$.
\end{proof}

In special cases, equivariant factorization homology can be computed using factorization homology of manifolds.
\begin{nota} \label{nota:forgetequivariance}
	We denote with $A[\DD_\e/\ga]$ the $\Disk_n^G$-algebra associated to an $\ga^\triv\Disk_n^G$-algebra $A$.
\end{nota}
\begin{cor} \label{cor:efhreduces}
	Let $[M/\ga] \in \ga\Quot_n^G$, $\catC^{\tensor}$ be $\tensor$-cocomplete, $A \in \mathrm{Alg}_{\ga\Disk_n^G}(\catC)$.
	Then
	\begin{align} 
	\int_{[M/\ga]} A \simeq \int_{M/\ga} A[\DD_\e/\ga]
	\end{align}
	if either $[M/\ga]$ is a $G$-framed trivial $\ga$-quotient or if $A$ is trivially $\ga$-equivariant.
\end{cor}

\begin{proof}
	In both cases we will prove the result by showing two functors are both equivariant factorization homology theories restricting to the same disk algebra. That all functors involved preserve sequentials colimits is clear. 
	Now let us first treat the case of trivial global quotients.
	A collar-gluing of the manifold $M$ induces a $\ga$-invariant collar-gluing of $[\ga \times M/\ga]$ so that $\tensor$-excision implies that $\int_{\ga \times -} A$ is a $\catC$-valued factorization homology theory of $G$-framed manifolds. 
	The functor $\int_{\bullet} A[\DD_\e/\ga]$ also defines a $\catC$-valued factorization homology theory of $G$-framed manifolds. 
	Moreover, both restrict to $\Disk_n^G$ as $A[\DD_\e/\e]$ by \ref{cor:localcomp}. 
	Therefore they must agree. 
	
	Next we treat the case of $A$ being trivially $\ga$-equivariant. 
	We have $\int_{-} A \in \cH(\ga\Quot_n^G,\catC)$. 
	Since collar-gluings of global quotients are $\ga$-invariant, a collar-gluing of a free quotient $[M/\ga]$ induces a collar gluing of the quotient manifold $M/\ga$.
	Therefore, $\tensor$-excision implies that $\int_{-/\ga} A[\DD_\e/\ga]$ is also a $\catC$-valued equivariant homology theory of $G$-framed free quotients. 
	By assumption $A = \ga \times A[\DD_\e/\ga]$, so that both functors restrict to $\ga^\triv\Disk_n^G$ as $A$ by \ref{cor:localcomp}. 
	We conclude they must agree. 
\end{proof}

\begin{rmk}
	This is the analogue of Proposition \ref{prop:bcreduces} for Bredon cohomology.
\end{rmk}

We now show equivariant factorization homology satisfies a Fubini formula \eqref{fubini}.
To do this we extend the functor $\times \ga$ of Proposition \ref{prop:timesandquotient} as follows.
Given global quotients $[M/\ga] \in \ga^{\rho}\Orb_{m}^{G}$ and $[N/J] \in J^{\chi}\Orb_{n}^{H}$ we can canonically associate a global quotient $[M\times N/\ga \times J] \in \ga \times J^{\rho \oplus \chi}\Orb_{m+n}^{G \times H}$. 
This defines a functor
\[ \ga^{\rho}\Orb_{m}^{G} \times J^{\chi}\Orb_{n}^{H} \longrightarrow \ga \times J^{\rho \oplus \chi}\Orb_{m+n}^{G \times H}.\] 
We recover the functor $\times \ga$ in the case $J=\e, m=0$.
\begin{cor} \label{cor:fubini}
	Consider a global quotient $[M\times N/\ga \times J] \in \ga \times J^{\rho \oplus \chi}\Orb_{m+n}^{G \times H}$ and a $\ga \times J^{\rho \oplus \chi}\Disk_{m+n}^{G \times H}$-algebra $A$ in $\catC$. Then
	\begin{align} \label{fubini}
	\int_{[M/\ga]} \int_{[N/J]} A \quad \simeq \quad \int_{[M \times N/\ga \times J]} A \quad \simeq \quad \int_{[N/J]} \int_{[M/\ga]} A.
	\end{align}
\end{cor}
\begin{proof}
	Let us only show the first equality holds, the second computation being completely analogous. 
	Fixing $[N/J]$ we have the $\ga^{\rho}\mathrm{Disk}_{n}$-algebra $\int_{[N/J]} A$ defined via
	\[ [\DD_I/\ga] \mapsto \int_{[\DD_I \times N/\ga \times J]} A.\]
	A collar-gluing of $[M/\ga]$ induces a collar-gluing of $[M \times N/\ga \times J]$ so that $\tensor$-excision implies that $\int_\bullet \int_{[N/J]} A$ is a $\catC$-valued factorization homology theory of $(\ga,G)$-framed global quotients (preserving sequential colimits is clear).
	On the other hand, the functor
	\begin{align*} 
	&\int_{- \times [N/J]} A: \ga^{\rho}\mathrm{Orb}_{n} \rightarrow \catC,\\
	& [M/\ga] \mapsto \int_{[M\times N/\ga \times J]} A,
	\end{align*}
	clearly also defines a $\catC$-valued factorization homology theory of $(\ga,G)$-framed global quotients.
	Moreover, both restrict to the disk algebra $\int_{[N/J]} A$. 
	Therefore, they must agree.
\end{proof}

We will now prove Theorem \ref{thm:classification} by making use of the so-called equivariant handle bundle decompositions of \cite{wasserman69}.
\begin{nota}
	For $x \in M$ a global quotient denote by $\ga_m := \{ g \in \ga : gx = x\}$ the isotropy group. For orthogonal representations $V$ of some group denote $B(V) := \{ v\in V : |v| \leq 1\}$ and $S(V):= \{ v\in V : |v| = 1\}$.
\end{nota}
\begin{defn} \cite{wasserman69} Let $\ga$ act on $n$-dimensional manifold $M$. For $x\in X$ let $V_x$ and $W_x$ be orthogonal representations of $\ga_x$. 
	\begin{enumerate}
		\item The bundle $B(V) \oplus B(W) := (B(V_x) \oplus B(W_x)) \times_\ga \ga_x \rightarrow \gamma x$ over the orbit $\ga x$ of $x$ is called a \emph{handle bundle of index $\dim(W_x)$ and type $\ga_x$}. 
		\item Suppose $Z$ is a $\ga$-invariant submanifold of $M$ with boundary, an equivariant embedding $\phi: S(V)\oplus D(W) \rightarrow \partial Z$  so that $M = Z \cup_{S(V)\oplus B(W)} B(V)\oplus B(W)$. Then we say that \emph{$M$ is obtained from $Z$ by attaching a handle bundle of index $\dim(W)$}. 
	\end{enumerate}
\end{defn}

\begin{thm} \cite[Corollary 4.11]{wasserman69} \label{wasserman}
	Let $M$ be an $n$-dimensional compact manifold with a $\ga$-action. Then $M$ admits a finite handle bundle decomposition
	\[ M \cong  B(V_0)\oplus B(W_0) \cup_{ S(V_1)\oplus B(W_1)} B(V_1)\oplus B(W_1)  \dots \cup_{S(V_N)\oplus B(W_N) } B(V_N)\oplus B(W_N)  \]
	for handle bundles $B(V_i)\oplus B(W_i)$. 
\end{thm}

A handle bundle attachment of index $q$ can easily be adapted to an attachment of an \emph{open handle bundle} 
\[\mathbf{S}_\ep(V) \times D(W) := \{ (v,w)\in V \times W : 1 - \ep < |v| < 1, |w|<1 \}\] 
for some $\ep > 0$. Namely for $M = Z \cup_\phi B(V)\oplus B(W)$  we have the gluing of open submanifolds of $M$: 
\[ M = (Z\setminus \partial Z) \cup_{\mathbf{S}_\ep(V) \times D(W)}  D(V) \oplus D(W). \]
Note that this defines a collar-gluing of $M$ as $\ga$ acts orthogonally on $V$ and hence trivially on the radial direction of the annulus $\mathbf{S}_\ep (V) \cong \R \times S^{n-q-1}$.
\begin{cor} \label{handlebundledecomp}
	Every $(\rho,G)$-framed global quotient $[M/\ga]$ is the sequential colimit of $(\rho,G)$-framed global quotients with finite open handle bundle decompositions. 
\end{cor}
\begin{proof}
	It is well known that any manifold can be decomposed as $M = \cup_i M_i$ such that all $M_i$ are contained in the interior of a compact manifold.
	Now for a $(\rho,G)$-framed global quotient we have $M = \cup_i \ga M_i$ where $\ga M_i$ is the $\ga$-global quotient obtained as the $\ga$-orbit of $M_i \subset M$. 
	Again $\ga M_i \subset \overline{\ga M_i}$, and the pieces are $\ga M_i$ inherit $(\rho,G)$-framing as open submanifolds of $M$. 
	By Theorem \ref{wasserman} and the discussion directly thereafter any compact global quotient has a finite open handle bundle decomposition. 
	Consequently, $\ga$-invariant open submanifolds of compact global quotients also have finite open handle bundle decompositions. 
	This proves the result as we then have $[M/\ga] = \cup_i [\ga M_i/\ga]$ and we've shown the $[\ga M_i/\ga]$ have the required properties.
\end{proof}

\begin{rmk} \label{findecomp}
	By Corollary \ref{handlebundledecomp} any global quotient has a handle-bundle decomposition, but it may be infinite e.g. for surfaces of infinite genus. 
\end{rmk}

\begin{nota}
	For a handle bundle $B(V)\oplus B(W)$ attachement of index $q$ and type $I$ to $M$ we will denote the corresponding open handle bundle collar gluing by $M\cup_{S^{n-q}\times \R^q} \R^n$ and the $\ga$-global quotient associated to the collar by $[S^{n-q}\times \R^q/I]$.
\end{nota}

\begin{proof}[Proof of Theorem \ref{thm:classification}]
	We will prove the theorem by showing the unit and counit of the adjunction are isomorphisms. 
	The unit of the adjunction $A \rightarrow \int_{-} A \circ i^{\tensor}$ evaluated on some $\gar\Disk_n^G$-algebra $A$ is an isomorphism since $\int$ is the left Kan extension along a fully faithful inclusion functor. 
	The counit evaluated on some equivariant homology theory $H$ is a morphism $\int A_H \rightarrow H$. 
	Since both functors preserve sequential colimits, Corollary \ref{handlebundledecomp} implies it suffices to show $\int_{[M/\ga]} A_H \cong H[M/\ga]$ for $[M/\ga]$ with finite handle bundle decomposition.
	We proceed by induction lowering the index of the handle-bundles.
	\vspace{3mm}
	
	\emph{Claim 1}: 
	The functors are isomorphic on all global quotients whose handle bundle decomposition only involve handle bundles of index $n$.\\
	\emph{Proof claim 1}: Indeed, let the decomposition be of length $0$. Then $[M/\ga] = [D(V)/\ga]$ is a handle bundle, say of type $I \leq \ga$. 
	Since the connected components of $D(V)$ are contractible $D(V)$ has the trivial $G$-structure and one easily verifies that $[D(V)/\ga] \simeq [\DD_I/\ga]$ so that
	\[ \int_{[\DD_I/\ga]} A_H = H[\DD_I/\ga] \]
	by Proposition \ref{cor:localcomp}.
	We then use induction on the length of the handle body decomposition. 
	To verify the inductive step let $M$ be obtained from $M_0$ by attaching an open handle $D(V)$ of index $n$ and type $I$. 
	We compute
	\begin{align*} 
	\int_{[M/\ga]} A_H &\simeq \int_{[M_0/\ga]} A_H \bigotimes_{\int_{[S^0\times \R^n/I]} A_H} \int_{[D(V)/I]}A_H \\
	&\simeq H[M_0/\ga] \bigotimes_{H[S^0\times \R^n/I]} H[D(V)/I] \simeq H[M/\ga] 
	\end{align*}
	where we used $[S^0\times \R^n/I] \simeq [\DD_I/\ga]\amalg [\DD_I/\ga]$ so that $H$ and $\int_A$ agree on it by Proposition \ref{cor:localcomp}.
	Thus $\int A_H$ and $H$ agree on all global quotients with a finite handle bundle decomposition only involving handles of index $n$.
	\vspace{3mm}
	
	\emph{Claim 2}: 
	Suppose that $\int A_H$ and $H$ agree on all global quotients whose handle bundle decompositions only involve handles of index larger or equal to $q$ for some given $q<n$. 
	Then $\int A_H$ and $H$ agree on a global quotient whose handle bundle decompositions only involve handles of index larger or equal to $q-1$. \\
	\emph{Proof claim 2:} 
	Indeed, let the decomposition be of length $0$, then $[M/\ga]$ is handle bundle of index $i$ where $i\geq q-1$. The cases $i>q-1$ are covered by the inductive hypothesis. 
	For the handle bundles $S^{n-q+1} \times \R^{q-1}$ of index $q-1$ and type $I$ we decompose the $I$-global quotient $S^{n-q+1}$ into handle bundles. Consequently we have a $\ga$-equivariant handle bundle decomposition of $S^{n-q+1} \times \R^{q-1}$ involving handles of index strictly greater than $q-1$. 
	By the induction hypothesis we find 
	\[\int_{[S^{n-q+1} \times \R^{q-1}/I]} A_H = H[S^{n-q+1} \times \R^{q-1}/I]. \]
	For global quotients with decompositions of greater length we can apply an induction argument on length exactly as in claim 1. 
	\vspace{3mm}
	
	Claim 1 and claim 2 prove the induction on index, and finishes the proof. 
\end{proof}

\section{Examples and applications} \label{sec:applications}

In this section we discuss various examples and applications of equivariant factorization homology. 
Using the characterisation of Theorem \ref{thm:classification} we recognise that classical (co)homology theories for global quotient orbifolds, such as Bredon homology and Chen-Ruan cohomology, are examples of equivariant factorization homology theories. 
We then turn to computing invariants of algebras in \S \ref{subsec:thh} and tensor categories in \S \ref{subsec:crt}. 
We also discuss constructions of braid group actions in the latter subsection. 

\subsection{Borel equivariant homology and Bredon homology} \label{subsec:bredon}

We recall the definition of Borel equivariant homology. For a $\ga$-global quotient $[M/\ga]$ we have the \emph{universal bundle} $\mathrm{E}\ga \rightarrow \mathrm{B}\ga$,\footnote{Recall $\mathrm{E}\ga$ is a contractible space on which $\ga$ acts freely; it is unique up to homotopy equivalence.} which we use to define the \emph{Borel construction} $(M \times \mathrm{E}\ga)/\ga$.
Then the Borel $\ga$-equivariant (co)homology of $M$ is defined as
\[ H_*^\ga (M) := H_* ((M \times \mathrm{E}\ga)/\ga)\]
the singular (co)homology of the Borel construction. 
More generally, we can consider Bredon (co)homology where the group $\ga$ is allowed to act on the coefficients. 

\begin{defn} \label{coeffsystem} Let $\ga$ be a finite group, let $\AB$ denote the category of Abelian groups. 
	\begin{enumerate}
		\item \emph{The orbit category} $\catO(\ga)$ of $\ga$ has as objects the $\ga$-sets $\ga/I$ where $I \leq \ga$ is a subgroup. Morphisms are $\ga$-equivariant maps. 
		\item A \emph{(contraviariant) coefficients system for $\ga$} is a (contravariant) functor $\catO(\ga) \rightarrow \AB$. 
	\end{enumerate}
\end{defn}
We recall the axiomatic approach to Bredon (co)homology.
Let $\ga\mathbf{CW}_{\mathrm{pairs}}$ denote the category whose objects are pairs of $\ga$-CW complexes $(X,A)$ where $A \subset X$ is a subcomplex, and maps are $G$-equivariant maps preserving subcomplexes.
\begin{defn} \cite{bredon67a,bredon67b}
	A \emph{Bredon (co)homology theory} consists of a collection of (contravariant) functors 
	\[\big\{ H_n: \ga\mathbf{CW}_{\mathrm{pairs}} \rightarrow \AB \big\}_{n \in \Z}\]
	and a collection of natural homomorphism $\{ \partial_n: H_n(A,\emptyset) \rightarrow H_{n-1}(X,A) \}_{n \in \Z}$ such that
	\begin{enumerate}
		\item (Homotopy) Homotopic maps are assigned the same map in homology.
		\item (Additivity) $H(\coprod_\alpha X_\alpha ) \cong \oplus_\alpha H_n(X_\alpha)$ for abitrary disjoint unions.
		\item (Excision) An inclusion $(X\setminus U,A\setminus U) \rightarrow (X,A)$ where $\bar{U} \subset A^{\circ}$ induces an isomorphism on homology.
		\item (Exactness) The inclusions $A \rightarrow X$ and $(X,\emptyset) \rightarrow (X,A)$ induces a long exact sequence on homology
		\[ \cdots \rightarrow H_n(A) \rightarrow H_n(X) \rightarrow H_n(X,A) \xrightarrow{\partial_n} H_{n-1}(A) \rightarrow \cdots \]
		\item (Dimension) $H_n (\ga/I) = 0$ for all $\ga/I \in \catO(\ga)$, $n \neq 0$ and $H_{-i}(X) = 0$ for all $i\geq 1$ and all $X$.
	\end{enumerate}
	where $H_n(X)$ is shorthand for $H_n(X,\emptyset)$.\footnote{For cohomology theories we of course rather have $\partial_n: H_{n-1}(X,A) \rightarrow H_n(A,\emptyset)$ and the exact sequence moves in the opposite direction.}
\end{defn}
We have the following formal consequence of the axioms:
\begin{prop} (Mayer-Vietoris) 
Let $X = X_+ \cup_{X_0} X_-$ be a $\ga$-equivariant decomposition of $X$ into $\ga$-CW complexes. 
Then for any Bredon (co)homology theory we have 
\begin{align} H(X) \cong H(X_+) \oplus_{H(X_0)} H(X_-). \label{mayer-vietoris} \end{align}
\end{prop}

\begin{thm} \cite[\S5]{bredon67b} \label{bredonclass} 
	For a given (contravariant) coefficient system $A: \catO(\ga) \rightarrow \AB$ there is a unique Bredon (co)homology theory $H^\ga_*(-,A)$ satisfying 
	$H^\ga_0(\ga/I,M) = A(\ga/I)$ for all $\ga/I \in \catO(\ga)$.
\end{thm}

\begin{ex}
	Borel equivariant (co)homology is the special case $H_*^\ga = H_*^\ga(-,\Z)$ of Bredon equivariant (co)homology with constant coefficient system $\Z$. 
\end{ex}

We have the following well-known property of Bredon (co)homology:
\begin{prop} \label{prop:bcreduces}
	Let $\ga$ act freely on $M$. Then $H_*^\ga(M,A) = H_*(M/\ga,A(\ga/\e))$ if either $A$ is a constant coefficient system or if $M = \ga \times N$ is a trivial global quotient.
\end{prop}

We now consider the symmetric monoidal $\infty$-category $\Ch_k^{\oplus}$ of chain complexes over some field $\kk$. 
The Bredon (co)homology of a $\ga$-CW complex $X$ also has an explicit construction as the homology of a certain assigned chain complex $C_*^{\ga}(X,M)$ \cite{bredon67a,bredon67b}.
Since every $\ga$-manifold admits a $\ga$-CW structure \cite[Cor. 4.1]{illmanthesis}, we can view Bredon homology as a symmetric monoidal functor
\[ C_*^\ga(-,A): \gar\Orb_n \rightarrow \Ch_\kk^{\oplus}\]
such that $H^\ga_0 [\DD_I/\ga]  = A(\ga/I)$.\footnote{Here the structure group is thus $GL(n)$, and hence framed embeddings can be taken to be $\ga$-equivariant embeddings by Proposition \ref{prop:hfrisgfr}.} Bredon cohomology yields a symmetric monoidal functor 
\[ C^*_\ga(-,A): \gar\Orb_n \rightarrow (\Ch_\kk^{\oplus})^{op}.\]
In both cases there are no restrictions on $\rho$.
\begin{prop} \label{prop:beh}
	Any Bredon (co)homology theory with (contra/co)variant coefficients $A$ is computed by factorization homology, i.e.
	\[H^\ga_*(M,A) \simeq \int_{[M/\ga]} C_*^\ga(-,A) \quad \text{and} \quad H_\ga^*(M,A) \simeq \int_{[M/\ga]} C^*_\ga(-,A). \]
\end{prop}
\begin{proof}
Preserving sequential colimits is clear, so we only need to verify the $C_*^\ga$ satisfies $\tensor$-excision. 
Indeed, a collar-gluing for $X$ induces a decomposition $X = X_- \cup_{X_0} X_+$ and taking singular chains then give a quasi-isomorphism
\[  C^\ga_*(M_+) \bigoplus_{C^\ga_*(M_0)} C^\ga_*(M_-) \rightarrow C^\ga_* (M) \]
by the Mayer-Vietoris property \ref{mayer-vietoris}. The proof for $C_\ga^*$ is identical.
\end{proof}

\subsection{Chen-Ruan cohomology of almost complex global quotients} \label{subsec:crcoh}

W. Chen and Y. Ruan developed a cohomology theory for almost complex orbifolds that is  not necessarily restricted to global quotient orbifolds \cite{chenruan04}.
This cohomology theory takes inspiration from string theories constructed out of orbifolds, as were considered in \cite{vafawitten} for example, and provides the definitions of orbifold Euler characteristic and orbifold Betti numbers. 
The Chen-Ruan cohomology theory allows a simple description for global quotients as we will now recount.
For the general theory we recommend \cite[Chapter 4]{ademleidaruan}.

\begin{defn} \cite[Def. 1.27]{ademleidaruan}
	An \emph{almost complex global quotient} is a global quotient $[M/\ga]$ with a $\ga$-equivariant bundle automorphism $J: TM \rightarrow TM$ such that $J^2=-\id$.
\end{defn}

For $g \in \ga$ we denote with $\amalg_g M^g := \{ (g,x) \in \ga \times M: gx = x  \}$. The space $\amalg_g M^g$ has a natural $\ga$-action via $h \cdot (g,x) = (hgh^{-1},hx)$.
\begin{prop} (See for example \cite[p.117]{ademleidaruan}) 
	The \emph{Chen-Ruan cohomology of $[M/\ga]$},denoted $H^*_{\mathrm{CH}}[M/\ga]$, can be computed as
	\begin{align*}
		H^*_{\mathrm{CH}}[M/\ga] \cong H^* \left(\coprod_{[g] \in \mathrm{Conj}(\ga)} M^g/C_G(g) \right)
	\end{align*}
	where $C_G(g)$ is the centraliser of $g$ in $G$ and $\mathrm{Conj}(\ga)$ is the set of conjugacy classes in $\ga$.
\end{prop}

\begin{ex} \cite[Th. 5.11]{ademleidaruan}
	Consider an almost complex global quotient $[M^{\times k}/S_k]$ as in Example \ref{ex:stringytopology}.
	The direct sum $\bigoplus_{k \geq 0} H^*_{\mathrm{CR}} [M^{\times k}/S_k]$ is naturally an irreducible highest weight representation of the super Heisenberg algebra $\catA (H^*(M))$.\footnote{The Heisenberg algebra is a super Lie algebra one can associate to a $\Z_2$-graded vector space; it plays an important role in supersymmetric quantum field theory.}
\end{ex}

\begin{lem} 
	Let $\rho: \ga \rightarrow GL(\C^n)$ be a complex representation. A $(\rho,GL(\C^n))$-framed global quotient is an almost complex global quotient.
\end{lem}
\begin{proof}
	It is well known that a $GL(\C^n)$-structure on $M$ corresponds to an automorphisms $J: TM \rightarrow TM$ such that $J^2=-\id$. 
	Namely, acting by $i$ on the $GL(\C^n)$-bundle $P_{GL(\C^n)}(M)$ induces an involution $i$ of $\Fr(M)$. 
	Now take the associated vector bundle map $J$ to $i$ obtained through the canonical identification 
	\[TM \cong \Fr(M) \times_{GL(2n)} \R^{2n} .\] 
	We have that $c_g (P_{GL(\C^n)}(M)) \subset P_{GL(\C^n)}(M)$ for all $g\in \ga$ with $c_g(p) = g_*(p) \cdot \rho (g)^{-1}$.
	By our assumption on $\rho$ this implies that $g_*(P_{GL(\C^n)}(M)) \subset P_{GL(\C^n)}(M)$. 
	In particular $g_*$ commutes with the $i$-action on $P_{GL(\C^n)}(M)$ and hence $g_*$ and $J$ commute i.e. $J$ is $\ga$-equivariant. 
\end{proof}

We note that the assignment 
	$[M/\ga] \mapsto \coprod_{[g] \in \mathrm{Conj}(\ga)} M^g/C_G(g)$
	is functorial and sends disjoint unions to disjoint unions so that
	\[C^*_{\mathrm{CR}}: \gar\Orb^{GL(\C^n)}_{2n} \rightarrow (\Ch_\kk^{\oplus})^{\mathrm{op}},\quad \quad [M/\ga] \mapsto C^*\left(\coprod_{[g] \in \mathrm{Conj}(\ga)} M^g/C_G(g)\right)\] 
	defines a symmetric monoidal functor.\footnote{Here we are forgetting a framed embedding down to the underlying equivariant embedding, see Remark \ref{rmk:underlyingemb}, which induces a map on the level of $C^*_{\mathrm{CR}}$.}

\begin{prop}
	For any $\rho: \ga \rightarrow GL(\C^n)$ Chen-Ruan cohomology is computed by equivariant factorization homology, i.e.
	\[ H^*_{\mathrm{CR}}[M/\ga] \simeq  \int_{[M/\ga]} C^*_{\mathrm{CR}}.\]
\end{prop}

\begin{proof}
	Preserving sequential colimits is clear, so it only remains to verify $\tensor$-excision. 
	Indeed, if $M = M_- \cup_{M_0} M_+$ is a collar-gluing we have 
	\[ \amalg_{[g] \in \mathrm{Conj}(\ga)} M^g/C_G(g) \quad = \quad  \amalg_{[g] \in \mathrm{Conj}(\ga)} M_-^g/C_G(g) \;  \bigcup_{\amalg_{[g] \in \mathrm{Conj}(\ga)} M_0^g/C_G(g)} \;   \amalg_{[g] \in \mathrm{Conj}(\ga)} M_+^g/C_G(g)\]
	so that the $\tensor$-excision follows from the Mayer-Vietoris property of singular homology. 
\end{proof}

\subsection{Twisted Hochschild homology} \label{subsec:thh}

Recall from Example \ref{ex:tensorcocomplete} that $\vect$ and $\Ch_\kk^{\tensor}$ are $\tensor$-cocomplete such that factorization homology exists and satisfies $\tensor$-excision. 
We give an example of a choice of coefficients. 

\begin{ex}
By combining Example \ref{ex:enalg} and Proposition \ref{prop:equivariantalgebras} we find that a $\Z_2^\triv\Disk_1^{\mathrm{fr}}$-algebra in $\vect$ consists of a unital algebra $A$ together with an algebra involution $\phi: A \rightarrow A$.
In Figures \ref{fig:algemb1} and \ref{fig:algemb2} we have illustrated the embeddings $f_\phi$, $f_\mu$ that correspond to the involution and multiplications maps on $A$, in the sense that $\phi = A(f_\phi): A \rightarrow A$ and $\mu = A(f_\mu): A \tensor A \rightarrow A$.
\end{ex}

\begin{figure}[h]
	$f_\phi: [\DD_\e/\Z_2] \rightarrow [\DD_\e/\Z_2]$ \quad
	\begin{tikzpicture}
		\draw [(-),blue,thick] (0,0) -- (1,0);
		\draw [(-),blue,thick] (1.5,0) --(2.5,0);
		\node at (.5,.5) {$b$};
		\node at (2,.5) {$b_\phi$};
	\end{tikzpicture}
	\; \; $\longrightarrow$ \; \;
	\begin{tikzpicture}
		\draw [(-),blue,thick] (0,0) -- (1,0);
		\draw [(-),blue,thick] (1.5,0) --(2.5,0);
		\node at (.5,.5) {$b_\phi$};
		\node at (2,.5) {$b$};
	\end{tikzpicture}
\caption{The $\Z_2$-equivariant embedding that induces the involution on $A$.}
	\label{fig:algemb1}
\end{figure}

\begin{figure}[h]
	$f_\mu: [\DD_\e/\Z_2]_b \amalg [\DD_\e/\Z_2]_r \rightarrow [\DD_\e/\Z_2]$ \quad   
	\begin{tikzpicture} 
		\draw [(-),black,dotted] (0,0) -- (3.5,0);
		\draw [(-),black,dotted] (4,0) --(7.5,0);
		\draw [(-),blue,thick] (0.5,0) -- (1.5,0);
		\draw [(-),red,thick] (2,0) -- (3,0);
		\draw [(-),blue,thick] (4.5,0) -- (5.5,0);
		\draw [(-),red,thick] (6,0) -- (7,0);
		\node at (1,.5) {$b$};
		\node at (2.5,.5) {$r$};
		\node at (5,.5) {$b_\phi$};
		\node at (6.5,.5) {$r_\phi$};
	\end{tikzpicture} 
	\caption{The $\Z_2$-equivariant embedding that induces the multiplication on $A$.}
	\label{fig:algemb2}
\end{figure}

We have the following natural generalization of traces and Hochschild homology in the presence of involutions.

\begin{defn} \label{def:twhh} Let $(A,\phi)$ be a unital $\kk$-algebra with involution.
	\begin{enumerate}
		\item We define the $\phi$-twisted traces of $A$ to be 
		\[\mathrm{Tr}^\phi(A) := \{f: A \rightarrow \kk \; | \; f(ab) = f(\phi(b) a) \; \; \forall a,b \in A\}. \]
		\item The $\phi$-twisted Hochschild homology of $A$, denoted $HH_*^\phi(A)$, is the Hochschild homology $HH_*(A, {}_{A}A_\phi)$ where ${}_{A}A_\phi$ is the $A$-bimodule $A$ with the action $(a,b)\cdot c = ac\phi(b)$. 
	\end{enumerate}
\end{defn}

\begin{rmk}
	Note that the linear dual of $HH^\phi_0(A)$ is isomorphic to $\mathrm{Tr}^\phi(A)$. 
\end{rmk}

\begin{prop} \label{prop:twhh}
	Let $A$ be a $\Z_2^\triv\Disk_1^{\mathrm{fr}}$-algebra in $\vect$, and denote also with $A$ the associated $\Z_2^\triv\Disk_1^{\mathrm{fr}}$-algebra in $\Ch_\kk^{\tensor}$. Let $[S^1/\Z_2]$ be the $\triv$-framed global quotient obtained by rotating the circle. 
	Then factorization homology over $[S^1/\Z_2]$ with coefficients in $A$ is computed by
	\[ \int_{[S^1/\Z_2]} A \simeq HH_*^\phi(A) \text{ in } \Ch_\kk^{\tensor} \]
	the $\phi$-twisted Hochschild homology of $A$.
\end{prop}

\begin{proof}
	We utilise the collar-gluing of Figure \ref{fig:circledecomp}. 
	We thus have a decomposition of $[S^1/\Z_2]$ into $M_+ = [r \sqcup r_\phi /\Z_2]$ and $M_- = [b \sqcup b_\phi/\Z_2]$, with overlapping regions $[g\amalg g_\phi /\Z_2]$ and $[p\amalg p_\phi/\Z_2]$.
	Note that all pieces in the decomposition are diffeomorphic to $[\DD_\e/\Z_2]$, and get assigned $A$. 
	Let us call those assignments $A_r$, $A_b$, $A_g$ and $A_p$ respectively. 
	The orientation of $[0,1]$ in the collar-gluing matches up with $g \amalg g_\phi$, but traverses $p\amalg p_\phi$ inversely. 
	Therefore, the $E_1$-algebra structure induced on 
	\[\underset{[g \amalg g_\phi/\Z_2] \amalg [p \amalg p_\phi/\Z_2]}{\int} A \quad \cong \quad A_g \tensor A_p \] 
	from Lemma \ref{lem:e1str} is the algebra structure $A \tensor A^{op}$. 
	Next, consider how $g\amalg g_\phi$ and $p\amalg p_\phi$ embed into $b \amalg b_\phi$ and $r\amalg r_\phi$. 
	We find that $A_b$ as an $A \tensor A^{op}$ module is ${}_A A_A$, whereas $A_r$ as an $A \tensor A^{op}$-module is ${}_A A_\phi$. 
	Using $\tensor$-excision, we then compute
	\begin{align*} 
	\int_{[S^1/\Z_2]} A \quad \cong \quad \int_{[b\amalg b/\Z_2]} A \; \; \underset{\int_{[g\amalg g_\phi/\Z_2]\amalg [p \amalg p_\phi/\Z_2]}A}{\bigotimes}  \; \; \int_{[r \amalg r_\phi/\Z_2]}A \quad \cong \quad {}_{A}A_A \underset{A \tensor A^{op}}{\overset{\mathbb{L}}{\tensor}} {}_{A} A_\phi
	\end{align*}
	which by definition is $CH_*(A,{}_{A} A_\phi) = CH_*^\phi(A)$.
\end{proof}

\begin{figure}[h]
	\includegraphics[trim={0 18.5cm 3cm 2.5cm},clip,scale=.5]{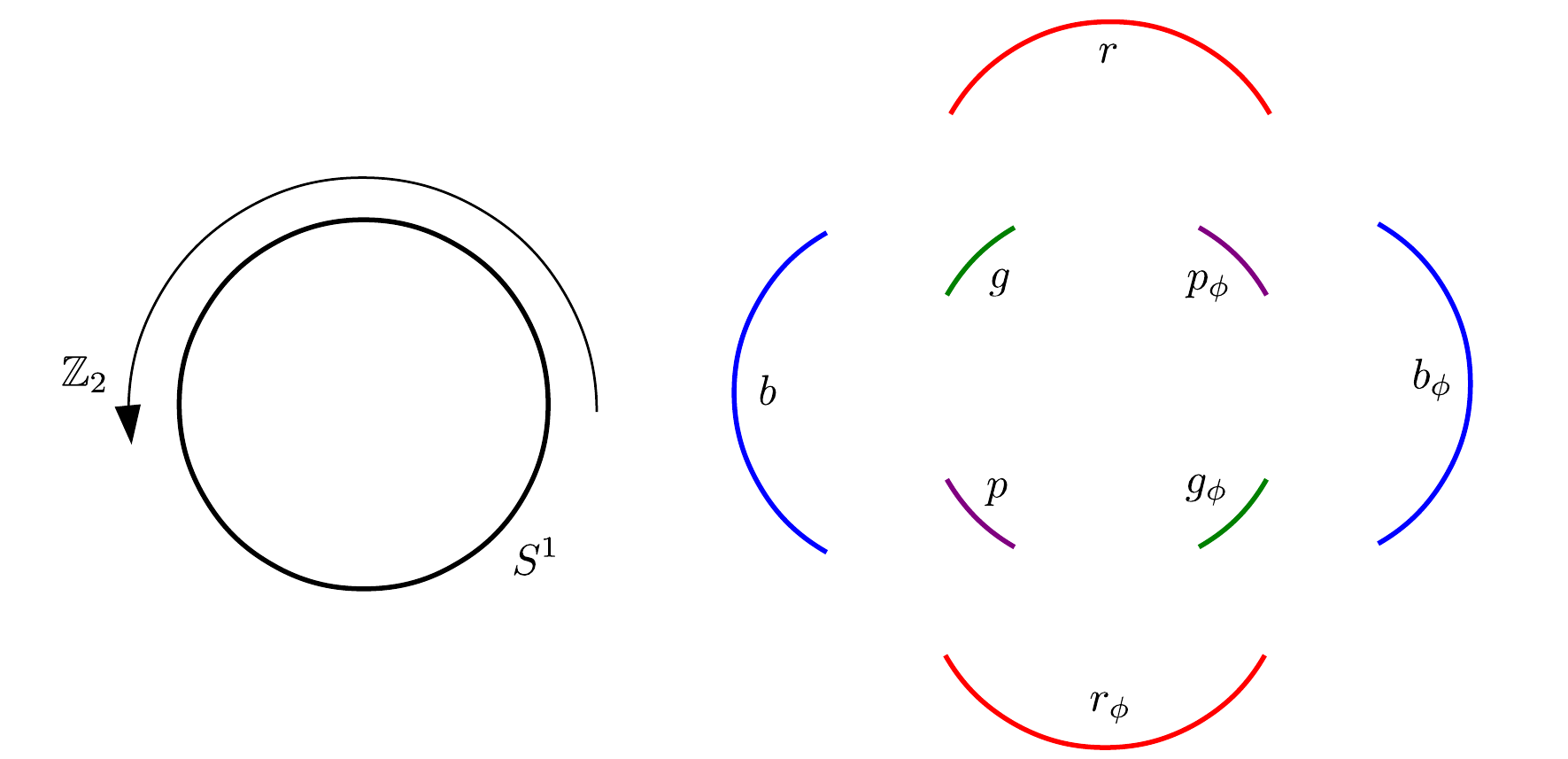}
	\caption{A $\Z_2$-equivariant decomposition of $S^1$ for the orbifold circle $[S^1/\Z_2]$.}
	\label{fig:circledecomp}
\end{figure}

\begin{rmk}
	If one computes the factorization homology of $A$ in the plain category $\vect$ one recovers only $HH^\phi_0(A)$.\footnote{The computation is identical, but the derived tensor products are then replaced by plain tensor products.}
\end{rmk}

\begin{rmk}
	The natural $S^1$-action on itself by rotations is an action by framed $\Z_2$-equivariant diffeomorphisms of $S^1$.
	By the functoriality of factorization homology, the $\phi$-twisted Hochschild homology inherits a canonical (homotopy) $S^1$-action.
\end{rmk}

\subsection{Categorical representation theory} \label{subsec:crt}

Let $\kk$ be a field. 
We consider factorization homology with values in the $(2,1)$-category of categories $\Rex$. 
By a \emph{$\kk$-linear category} we mean a category enriched over $\vect$. 
We follow the exposition of $\Rex$ given in \cite{bzbj}.

Recall that a functor is called \emph{right exact} if it preserves finite colimits. A category is \emph{essentially small} if it is equivalent to a small category.\footnote{A category is called \emph{small} if it has a set of objects and a set of morphisms, rather than proper classes.}
\begin{defn}
	$\Rex$ is the $2$-category of $\kk$-linear essentially small categories that admit finite colimits with morphisms right exact functors and 2-morphisms natural isomorphisms.
\end{defn}

For $\catC, \catD, \mathcal{E}$ $k$-linear we let $\text{Bilin}(\catC \times \catD, \mathcal{E})$ denote the category of $\kk$-bilinear functors from $\catC \times \catD$ to $\mathcal{E}$ that preserve finite colimits in each variable separately.
\begin{defn} 
	The \emph{Deligne-Kelly tensor product} $\catC \boxtimes \catD$ for $\catC,\catD \in \Rex$ is uniquely characterised by the natural equivalences
	\[ \Rex( \catC \boxtimes \catD, \mathcal{E}) \simeq \text{Bilin}(\catC \times \catD, \mathcal{E}) \simeq \Rex(\catC, \Rex(\catD, \mathcal{E})).\]
\end{defn} 
Recall from Example \ref{ex:tensorcocomplete} that $\Rex$ is $\tensor$-cocomplete such that factorization homology is well-defined and satisfies $\tensor$-excision.
In order to use the $\tensor$-excision property, we need a concrete realisation of the two-sided bar construction in $\Rex$. 

An $E_1$-algebra in $\Rex$ is a tensor category $(\catA, \tensor: \catA \boxtimes \catA \rightarrow \catA$) \cite[\S 3.3]{bzbj}.
Let $\catA$ be a tensor category, and let $(\catM, \tensor: \catA \boxtimes \catM \rightarrow \catM)$, and $(\catN, \tensor: \catN \boxtimes \catA \rightarrow \catN)$ be left and right $\catA$-module categories respectively.  
A functor $F: \catM \boxtimes \catN \rightarrow \catD$ is called \emph{$\catA$-balanced} if it is equipped with natural isomorphisms
\[ B_{m,X,n} : F(m \tensor X \boxtimes n) \cong F(n \boxtimes X \tensor n)\]
for $m \in \catM$, $X \in \catA$ and $n \in \catN$ satisfying some coherence equations \cite[Def. 3.1]{eno10}.
We denote with $\mathrm{Bal}_\catA(\catM \boxtimes \catN, \catD)$ the category of balanced functors. 

\begin{defn} \cite[Def 3.14]{bzbj}
	The \emph{relative Kelly tensor product} $\catM \boxtimes_\catA \catN$ of left and right module categories $\catM$ and $\catN$ over a tensor category $\catA$ is uniquely characterised by the natural equivalence
	\[ 	\Rex (\catM \boxtimes_\catA \catN, \catD) \cong \mathrm{Bal}_\catA(\catM \boxtimes \catN, \catD). \]
\end{defn}
The two-sided bar construction in $\Rex$ exactly computes the relative Kelly tensor product \cite[Remark 3.16]{bzbj}. 
\begin{ex} (\cite[Def. 3.1.34]{dsps13})
	A tensor category $\catA$ is naturally a module category over $\catA \boxtimes \catA^{mp}$ via the left and right self-action.\footnote{Here $\catA^{mp}$ is the tensor category with opposite multiplication.} The categorical trace of $\catA$, denoted $\mathrm{tr}(A)$, is by definition the relative Kelly tensor product $\catA \boxtimes_{\catA \boxtimes \catA^{mp}} \catA$. 
\end{ex}
For a monoidal functor $\Phi: \catA \rightarrow \catA$ and a right $\catA$-module category $\catM$ one can \emph{twist} the module category \cite[Def. 2.4.2]{dsps13}. Concretely, the $\Phi$-twisted module category $\catM_\Phi$ is the category $\catM$ where the action of $\catA$ is precomposed with the functor $\Phi$.
\begin{defn} \label{def:cattrace}
	Let $\catA$ be a tensor category and $\Phi: \catA \rightarrow \catA$ a monoidal functor. The $\Phi$-twisted trace of $\catA$ is denoted $\mathrm{tr}^\Phi(\catA)$ and defined as $\catA \boxtimes_{\catA \boxtimes \catA^{mp}} \catA_\Phi$.
\end{defn}
\begin{rmk}
	Definition \ref{def:cattrace} can be seen as a categorified version of the twisted Hochschild homology of Definition \ref{def:twhh}.
\end{rmk}
After choosing coefficients we will give an example computation of factorization homology.

\begin{ex} By combining Example \ref{ex:enalg} and Proposition \ref{prop:equivariantalgebras} we find that a $\Z_2^\triv\Disk_1^{\mathrm{fr}}$-algebra in $\Rex$ consists of a tensor category $\catA \in \Rex$ endowed with a monoidal involution, i.e. a monoidal functor $(\Phi,\phi,\phi_0): \catA \rightarrow \catA$, and a monoidal isomorphism $t: \Phi^2 \cong \id$. 
\end{ex}

\begin{prop} \label{prop:twdrinfeld}
	Let $\catA$ be a $\Z_2^\triv\Disk_1^{\mathrm{fr}}$-algebra. Let $[S^1/\Z_2] \in \ga\Quot_1^{\mathrm{fr}}$ be the framed free quotient obtained by rotating the circle. 
	Then factorization homology over $[S^1/\Z_2]$ with coefficients in $\catA$ is given by
	\[ \int_{[S^1/\Z_2]} \catA \cong \mathrm{tr}^\Phi(\catA)\]
	the $\Phi$-twisted trace of $\catA$.
\end{prop}
\begin{proof}
	We utilise the collar-gluing of Figure \ref{fig:circledecomp}, exactly as we did in the proof of Proposition \ref{prop:twhh}, and compute using $\tensor$-excision
	\begin{align*} 
		\int_{[S^1/\Z_2]} \catA \quad \cong \quad \int_{[\DD_\e/\ga]} \catA \; \; \underset{\int_{[\DD_\e\amalg \DD^{op}_\e/\ga]}A}{\bigotimes} \; \; \int_{[\DD_\e/\ga]} \catA \quad \cong \quad \catA \underset{\catA\boxtimes \catA^{mp}}{\boxtimes} \catA_\Phi.
	\end{align*}
\end{proof}

Next we will show that invariants computed from factorization homology in $\Rex$ automatically come equipped with braid group actions and diffeomorphism group actions. 

\begin{prop} \label{prop:brgract} Let $A$ be a $\gar\Disk_2^{G}$-algebra in $\Rex$, $[\Sigma/\ga] \in \gar\Orb_2^{G}$ and denote $\catA := A[\DD_\e/\ga]$.
	\begin{enumerate}
		\item Let $G=GL^+(n)$. The invariant $\int_{[\Sigma/\ga]} A$ carries a canonical action of the fundamental groupoid $\Pi_{\leq 1} \mathrm{Diff}^{+}[\Sigma/\ga]$ of the space of oriented $\ga$-equivariant diffeomorphisms. 
		\item Let $G=\e$. For $X_1,\dots,X_k \in \catA$ there exists an associated object 
		\[\int_{[\Sigma/\ga]} X_1 \tensor \dots \tensor X_k \in \int_{[\Sigma/\ga]} A\] 
		carrying a canonical action of the orbifold braid group $B_k[\Sigma/\ga]$.
	\end{enumerate}
\end{prop}
\begin{proof}
	The first result follows directly from the functoriality of factorization homology, Proposition \ref{prop:hfrisgfr} and the fact that $\Rex$ is a $(2,1)$-category so that the homotopy action is truncated.
	For part two, we fix a composite of framed embeddings 
	\[ \amalg_k [\DD_\e/\Z_2] \rightarrow [\DD_\e/\Z_2] \rightarrow [\Sigma/\Z_2],\]
	where the first embedding corresponds to the $n$-fold tensor product in the tensor category $\catA$.   
	Factorization homology then assigns a functor 
	\begin{align*} 
		\catA^{\boxtimes k} \rightarrow \int_{[\Sigma/\Z_2]} A,\quad \quad \quad 
		X_1 \boxtimes \dots \boxtimes X_k \mapsto  \int_{[\Sigma/\Z_2]} X_1 \tensor \dots \tensor X_n.
	\end{align*}
	The functor inherits $\ga^{\times k} \ltimes S_k$-equivariance from $A$. 
	In the (2,1)-category $\Rex$ the homotopy action on the functor truncates to a strict action of the fundamental group $\pi_1 F^\e_{k}[\Sigma/\Z_2]$, which is the pure braid group $PB_k[\Sigma/\ga]$ by Theorem \ref{htype}. 
	This action descends, by equivariance, to an action of the braid group $B_k[\Sigma/\ga]$ on the functor.
	In particular, by evaluating the functor on objects $X_1,\dots, X_k$ we obtain an action on the object $\int_{[\Sigma/\Z_2]} X_1 \tensor \dots \tensor X_n$.
\end{proof}

\begin{rmk}
	Proposition \ref{prop:brgract} of course extends to other structure groups. 
	Part two generalises to actions of what could be called $G$-framed orbifold braid groups. 
\end{rmk}

\begin{rmk}
	For $\ga=\e$ and $\Sigma$ an oriented surface the fundamental group $\pi_1 \mathrm{Diff}^{+}(\Sigma)$ is trivial so that the $\Pi_{\leq 1}\mathrm{Diff}^{+}(\Sigma)$-action reduces to an action of $\pi_0 \mathrm{Diff}^{+}(\Sigma)$, the mapping class group of the surface. 
\end{rmk}

\section{Variations} \label{sec:variations}

In this section we discuss two minor generalisations of equivariant factorization homology. The first variation simplifies the algebraic set-up needed to construct invariants for concrete examples. The second variation allows for more flexibility in constructing invariants.

\subsection{Restricting singularity types} \label{subsec:restrsing}

If one wishes to compute factorization homology of a $(\rho,G)$-global quotient $[M/\ga]$ one first fixes a $\gar\Disk_n^G$-algebra $A$. 
However, if there are many $I\leq_\rho \ga$ then describing $A$ will involve a lot of data, whereas the global quotient $[M/\ga]$ will typically only have certain types of singularities. 
We will now explain that one can compute the factorization homology in such cases only focussing on the relevant singularity types.

\begin{defn} Given a representation $\rho: \ga \rightarrow GL(n)$.
	A \emph{family of $\rho$-singularity types}  is a subset $\cS$ of the set $\{ I : I \leq_\rho\ga\}$.
\end{defn}

\begin{defn} Let $\cS$ be a family of $\rho$-singularity types. 
	\begin{enumerate}
		\item A $(\rho,G)$-framed global quotient $[M/\ga]$ is called \emph{$\cS$-singular} if every point $m\in M$ admits a $\rho$-framed neighbourhood equivalent to a local model of singularity type $I \in \cS$. 
		\item The $\infty$-category of $\cS$-singular $(\rho,G)$-framed global quotients, denoted $\cS\Orb_n^G$, is the full subcategory of $\cS$-singular $(\rho, G)$-framed global quotients.
		\item $\cS\Disk_n^G \subset  \cS\Orb_n^G$ is the full subcategory consisting of local models.
		\item Let $\catC$ be a symmetric monoidal $\infty$-category a \emph{$\cS\Disk_n^G$-algebra in $\catC$} is a symmetric monoidal functor $\cS\Disk_n^G \rightarrow \catC$. 
		\item Given a $\cS\Disk_n^G$-algebra $A$ in $\catC$ the \emph{factorization homology of $\cS$-singular $(\rho,G)$-framed global quotients with coefficients in $A$} is the symmetric monoidal left Kan extension \[\int_{-} A: \cS\Orb_n^G \rightarrow \catC.\] 
		along the symmetric monoidal inclusion $i^{\tensor}: \cS\Disk_n^G \rightarrow \cS\Orb_n^G$.
	\end{enumerate}
\end{defn}

The proofs of Proposition \ref{aftlem}, Theorem \ref{thm:excision} and Theorem \ref{thm:classification} easily generalise:

\begin{prop} \label{aftlemagain}
	Let $\catC$ be $\tensor$-cocomplete. Factorization homology of $\cS$-singular $(\rho,G)$-framed global quotients exists and is computed by the left Kan extension on objects.
\end{prop}

\begin{thm} \label{thm:classsing}Let $\catC$ be $\tensor$-cocomplete. Factorization homology provides an equivalence of categories
	\begin{align*}
		\int : \mathrm{Alg}_{\cS\Disk_n^G} \rightleftharpoons   \cH(\cS\Orb_n^G, \catC): (i^\tensor)^*
	\end{align*}
	where $\cH(\cS\Orb_n^G, \catC)$ is the $\infty$-category of symmetric functors that satisfy $\tensor$-excision and preserve sequential colimits.  
\end{thm}

The classification of Theorem \ref{thm:classsing} ensures us that we recover the usual invariants.\footnote{Alternatively, it is also clear that we obtain the usual invariants by inspection of the colimit in Proposition \ref{aftlemagain}.}

\begin{cor}
	Let $\catC$ be $\tensor$-cocomplete, $A$ be a	$\cS\Disk_n^G$-algebra $A$ in $\catC$ and $\tilde{A}$ be any $\gar\Disk_n^G$-algebra in $\catC$ such that 
	$\tilde{A}|_{\cS\Disk_n^G} = A$. Then
	\[ \int_{[M/\ga]} A \simeq \int_{[M/\ga]} \tilde{A}\]
	for any $\cS$-singular $(\rho,G)$-framed global quotient $[M/\ga]$.
\end{cor}

\subsection{Colouring singular strata} \label{subsec:colouring}

In the coend formula \eqref{coendformula} that computes the invariant of some global quotients different singular points of the same singularity type are assigned the same value. This is a simple consequence of the definition of $\gar\Disk_n^G$-algebras.
\begin{ex} \label{ex:circlenocolour}
 Consider the circle $S^1$ with the reflection $\Z_2$-action and an $r$-framed global quotient for $r: \Z_2 \rightarrow GL(2)$ the reflection representation. The global quotient has two isolated singularities. A $\Z_2^r\Disk_n^{\mathrm{fr}}$-algebra in $\vect$ consists of a unital algebra $A$ with anti-involution $\phi$ and a pointed module $M$. 
 Then $\int_{[S^1/\Z_2]} (A,M) \simeq M \tensor_A M$ by $\tensor$-excision. 
\end{ex}
However, one might wish to assign different values to the various singular strata. To accommodate this we introduce colourings.

\begin{defn} Let $[M/\ga]$ be a global quotient, and $(C,c_*\in C)$ be a pointed set. A
	\begin{enumerate}
		\item A \emph{singular stratum} of $[M/\ga]$ is the $\ga$-orbit of a connected component of $M^I$ for some $I \leq \ga$, $I\neq \e$. 
		\item A \emph{$C$-colouring} of $[M/\ga]$ is an assignment of elements $c \in C$ to the singular strata of $[M/\ga]$.
	\end{enumerate}
\end{defn}

\begin{conv}
	We always colour smooth strata with the default colour $c_*$.
\end{conv}

A map between global quotients with $C$-colourings is a map such that $f(x) = y$ implies that the strata containing $x$ and $y$ have the same colour. 
Consequently we have a natural notion of framed embedding beteen $C$-coloured $(\rho,G)$-framed global quotients.
\begin{defn}
	Let $\cS$ be a choice of $\rho$-singularity types. A \emph{family of $C$-colourings of $\cS$-singularities} is a choice of subset $\cF \subset \cS \times C$ such that $\cF \cap ( \e \times C )  =  \e \times \{c_*\}$. 
\end{defn}

We leave it to the reader to formulate the corresponding notions of $\cF\Disk_n^G$-algebra, factorization homology and the classification analogue of Theorem \ref{thm:classsing}. 
As the colourings are just a formal addition, the proofs are verbatim the same. 

\begin{ex}
	We consider $\Z_2$-global quotients thar are $r$-framed as in Example \ref{ex:circlenocolour} with singular strata coloured by either $c$ or  $c_*$. 
	A $\cF\Disk_1^{\mathrm{fr}}$-algebra in $\vect$ consists of a unital algebra $A$ with anti-involution $\phi$ and two pointed modules $M$ and $N$. 
	We colour the two singular points of $[S^1/\Z_2]$ by $c$ and $c_*$ respectively. We then have:
	\[\int_{[S^1/\Z_2]} (A,M,N) \simeq M \tensor_A N.\]
\end{ex}

\setcounter{section}{0}
\setcounter{equation}{0}
\renewcommand{\theequation}{\thesection.\arabic{equation}}
\setcounter{figure}{0}
\setcounter{table}{0}
 

\end{document}